\documentclass[10pt]{amsart}

\usepackage{comment}
\usepackage[noadjust]{cite}
\usepackage{hyperref}
\usepackage[capitalise,noabbrev]{cleveref}
\usepackage{amsmath}
\usepackage{enumitem}
\usepackage{kbordermatrix}
\usepackage{multicol}

\usepackage{todonotes}

\usepackage[labelformat=simple]{subcaption}
\usepackage{tikz}
\usepackage{tkz-graph}
\tikzstyle{VertexStyle} = [shape = circle, draw, fill]
\tikzset{pre/.style={-}}    %hide all loop arrow heads
\tikzstyle{every node}=[circle, inner sep=0pt, minimum width=4pt]
\usetikzlibrary{backgrounds}
\DeclareMathOperator{\Cr}{Cr}
\DeclareMathOperator{\si}{si}
\DeclareMathOperator{\co}{co}

\newcommand{\dY}{$\Delta$\nobreakdash-$Y$}
\newcommand{\Yd}{$Y$\nobreakdash-$\Delta$}

\newcommand{\GF}{\mathrm{GF}}
\newcommand{\ba}{\backslash}

\newcommand{\AG}{\mathit{AG}}

\newcommand{\BBR}{\mathit{BB}}

\newcommand{\FF}{\mathit{FF}}

\newcommand{\FN}{\mathit{FN}}

\newcommand{\Fq}{\mathit{Fq}}
\newcommand{\Fr}{\mathit{Fr}}
\newcommand{\Fs}{\mathit{Fs}}
\newcommand{\Ft}{\mathit{Ft}}
\newcommand{\Fu}{\mathit{Fu}}
\newcommand{\Fv}{\mathit{Fv}}

\newcommand{\FX}{\mathit{FX}}

\newcommand{\GP}{\mathit{GP}}

\newcommand{\KP}{\mathit{KP}}

\newcommand{\KR}{\mathit{KR}}

\newcommand{\LP}{\mathit{LP}}

\newcommand{\PP}{\mathit{PP}}
\newcommand{\SP}{\mathit{Sp}}

\newcommand{\TQ}{\mathit{TQ}}

\newcommand{\UG}{\mathit{UG}}

\newcommand{\WQ}{\mathit{WQ}}

\newcommand{\pappus}{\mathsf{Pappus}}
\newcommand{\spikey}{spiky $3$-separator}

\newcommand{\tcn}{nest}
\newcommand{\tcns}{nests}
\newcommand{\quadflower}{quad-flower}

\newtheorem{theorem}{Theorem}[section]
\newtheorem{lemma}[theorem]{Lemma}
\newtheorem{proposition}[theorem]{Proposition}

\newtheorem{corollary}[theorem]{Corollary}

\theoremstyle{definition}
\newtheorem{definition}[theorem]{Definition}

\setenumerate{label=\rm(\roman*),midpenalty=2}

\begin{document}

\title[Excluded minors for $\GF(5)$-representable matroids on $10$ elements]{The excluded minors for $\GF(5)$-representable matroids on ten elements}
\author{Nick Brettell} 
\address{School of Mathematics and Statistics\\
  Victoria University of Wellington\\
  %Wellington\\
  New Zealand}
\email{nick.brettell@vuw.ac.nz}
\thanks{Supported by the ERC consolidator grant 617951, and the New Zealand Marsden Fund.}

\maketitle

\begin{abstract}
  Mayhew and Royle (2008) showed that there are $564$ excluded minors for the class of $\GF(5)$-representable matroids having at most $9$ elements.
  We enumerate the excluded minors for $\GF(5)$-representable matroids having $10$ elements: there are precisely 2128 such excluded minors.
  In the process we find, for each $i \in \{2,3,4\}$, the excluded minors for the class of $\mathbb{H}_i$-representable matroids having at most $10$ elements,
  and the excluded minors for the class of $\mathbb{H}_5$-representable matroids having at most $13$ elements.
  %Our approach uses the theory developed by Pendavingh and van Zwam; in doing so, we find the excluded minors on at most ten elements for having at least $i$ inequivalent $\GF(5)$-representations.
  %The excluded minors for $\GF(5)$-representable matroids on at most nine elements were found by Mayhew and Royle by enumerating all matroids; our results are consistent with theirs.
\end{abstract}

\section{Introduction}

The excluded-minor characterisations for binary, ternary, and quaternary matroids are seminal results in matroid theory: there are 1, 4, and 7 excluded minors for these classes, respectively \cite{Tutte1958,Seymour1979,Bixby1979,GGK2000}.
For the class of matroids representable over $\GF(5)$, the excluded-minor characterisation is not known.
However, in 2008 Mayhew and Royle enumerated all matroids on at most nine elements, and thereby found that there are precisely 564 excluded minors on at most nine elements \cite{MR2008}.

These days, one can easily enumerate all 383172 matroids on at most nine elements on their home computer should they wish\footnote{For example, see the MatroidUnion blog post here: http://matroidunion.org/?p=301.}, but enumerating all $10$-element matroids still appears well out of reach (Mayhew and Royle estimate there are at least $2.5 \times 10^{12}$ sparse paving matroids with rank 5 and $10$ elements \cite{MR2008}).
However, here we enumerate all those that are excluded minors for the class of $\GF(5)$-representable matroids.

\begin{theorem}
  \label{thm1}
  There are $2128$ excluded minors for $\GF(5)$-representability on $10$ elements.
\end{theorem}

Our approach to find these excluded minors uses the theory of \emph{Hydra-$i$} partial fields developed by Pendavingh and van Zwam~\cite{PvZ2010b}.
A $3$-connected matroid representable over $\GF(5)$ has at most six inequivalent representations \cite{OVW1996}.
Let $M$ be a $3$-connected matroid with a $U_{2,5}$- or $U_{3,5}$-minor.
Pendavingh and van Zwam showed that, for each $i \in \{1,2,3,4,5,6\}$, there is a partial field $\mathbb{H}_i$ such that $M$ is $\mathbb{H}_i$-representable if and only if $M$ has at least $i$ inequivalent representations over $\GF(5)$.
In particular, $\mathbb{H}_1 = \GF(5)$.
Moreover, $\mathbb{H}_5=\mathbb{H}_6$; that is, if $M$ has at least five inequivalent representations over $\GF(5)$, then it has precisely six.
%Let $i$ and $j$ be positive integers with $1 \le i < j \le 5$, and let $\mathcal{N}_i$ be the set of excluded minors for $\mathbb{H}_{j}$-representability that are $\mathbb{H}_{i}$-representable, the matroids in $\mathcal{N}_i$ are \emph{stabilizers};
For $i \in \{1,2,3,4\}$, let $\mathcal{N}_i$ be the set of excluded minors for the class of $\mathbb{H}_{i+1}$-representable matroids that are $\mathbb{H}_{i}$-representable.
Then, the matroids in $\mathcal{N}_i$ are \emph{strong stabilizers} for the class of $\mathbb{H}_{i}$-representable matroids \cite{vanZwam2009}; loosely speaking, this means that when considering a $3$-connected matroid $M$ in this class with some matroid~$N$ in $\mathcal{N}_i$ as a minor, then
a representation of $N$ uniquely extends to a representation of $M$.
%after fixing a representation of $N$, one need not worry about inequivalent representations of $M$.
Thus, using this theory, one can first find the excluded minors for $\mathbb{H}_5$-representability on at most $10$ elements, then let $\mathcal{N}_4$ be those that are $\mathbb{H}_4$-representable, and then use these as seeds in order to find the excluded minors for $\mathbb{H}_4$-representability on at most $10$ elements.
This process can then be repeated to find the excluded minors for $\mathbb{H}_3$-representability on at most $10$ elements.
After two further iterations, one obtains the excluded minors for $\mathbb{H}_1$-representability on at most $10$ elements, and these are the excluded minors for the class of $\GF(5)$-representable matroids, just under another name.

As a slightly more subtle point, note that when an excluded minor $N$ for the class of $\mathbb{H}_{i}$-representable matroids is $\GF(5)$-representable, for some $i \in \{2,3,4,5\}$, it need not be $\mathbb{H}_{i-1}$-representable.
For example, $F_7^-$ is an excluded minor for $\mathbb{H}_5$-representable matroids that is $\mathbb{H}_2$-representable, but not $\mathbb{H}_3$- or $\mathbb{H}_4$-representable.
%Say that $N$ is an excluded minor for the class of $\mathbb{H}_{i}$-representable matroids that is $\mathbb{H}_{j}$-representable, for some $j < i-1$, but is not $\mathbb{H}_{j+1}$-representable (for example, $N=F_7^-$ is an example where $i=5$ and $j=2$).
%Then $N$ is an excluded minor for the class of $\mathbb{H}_{\ell}$-representable matroids for each $\ell \in \{j+1,j+2,\dotsc,i\}$.
%In particular, as $N$ is a $\mathbb{H}_j$-representable excluded minor for the class of $\mathbb{H}_{j+1}$-representable matroids, it is a member of $\mathcal{N}_j$, so will be used as a seed to find excluded minors for the class of $\mathbb{H}_j$-representable matroids.
In this case, where $N$ is $\mathbb{H}_{j}$-representable, for some $j < i-1$, but is not $\mathbb{H}_{j+1}$-representable,
$N$ is an excluded minor for the class of $\mathbb{H}_{\ell}$-representable matroids for each $\ell \in \{j+1,j+2,\dotsc,i\}$, and, in particular, 
%In particular, as $N$ is a $\mathbb{H}_j$-representable excluded minor for the class of $\mathbb{H}_{j+1}$-representable matroids,
it is a member of $\mathcal{N}_j$, so will be used as a seed to find excluded minors for the class of $\mathbb{H}_j$-representable matroids.

The resolution of Rota's conjecture \cite{GGW14} guarantees that there is a finite number of excluded minors for the class of $\GF(5)$-representable matroids.
There are certainly excluded minors for this class on more than 10 elements; in particular, there is an excluded minor on 12 elements, on 14 elements, and on 16 elements \cite{BP20}.
It is natural to ask the value of finding an incomplete list of excluded minors.
For the excluded-minor characterisations of representable matroids that are currently known exactly, the excluded minors have at most 8 elements \cite{GGK2000,HMvZ2011,BOSW2023b}, and there are at most 17 of them in total \cite{BOSW2023b}, so there is little that can be gleaned regarding how one might expect the excluded minors of increasing sizes to be distributed.
However, in light of \cref{thm1}, it seems likely that so far we have only discovered the tip of the iceberg (whereas one could have been a little bit more optimistic if only aware of those on at most 9 elements).
With effort, or more computing power, one might be able to extend the search beyond 10 elements, but reaching (or exceeding) 16 elements appears a long way off.
Nonetheless, in the author's opinion, it is worthwhile to explore how bad things might be, and there is precedent for doing this for minor-closed classes of graphs having similarly unwieldy excluded-minor characterisations \cite{Yu2006,MW2018}.
%
%For the classes of representable matroids for which we know excluded-minor characterisations, the ; though there are no results guaranteeing this behaviour.
%Moreover, if the number of excluded minors on 10 elements is anything to go by, the \ldots

More concretely, a component of this work, when combined with \cite{BOSW2023b}, gets us tantalisingly close to a complete list of the excluded-minors for the class of $\mathbb{H}_5$-representable matroids.
%An excluded-minor characterisation for $\mathbb{H}_4$-representable matroids, though very difficult, could perhaps also be achievable.
%
%Furthermore,
It turns out that the class of $\mathbb{H}_5$-representable matroids is closely related to another well-studied class.
For an integer $k \ge 2$, the class of \emph{$k$-regular} matroids is an algebraic generalisation of the class of regular and near-regular matroids, where regular matroids correspond to $0$-regular matroids, and near-regular matroids correspond to $1$-regular matroids (a formal definition of these classes is given in \cref{prelims}).
Although the class of near-regular matroids coincides with the class of matroids representable over all fields of size at least three~\cite{Whittle1997},
the class of $2$-regular matroids is a proper subclass of matroids representable over all fields of size at least four~\cite{Semple1998}.
However, the class of $3$-regular matroids coincides with the class of $\mathbb{H}_5$-representable matroids~\cite{BOSW2023b}.
Thus, attempts to discover the excluded minors for $\mathbb{H}_5$-representable matroids can equivalently be viewed as attempts to discover the excluded minors for $3$-regular matroids.

Brettell, Oxley, Semple, and Whittle showed that an excluded minor for the class of $2$-regular matroids, or for the class of $3$-regular matroids, has at most $15$ elements~\cite{BOSW2023b}.
Brettell and Pendavingh computed the excluded minors for the class of $2$-regular matroids of size at most $15$~\cite{BP20}.
Thus, an excluded-minor characterisation for the class of $2$-regular matroids is now known, and this is an important step towards understanding matroids representable over all fields of size at least four. %, just as finding the excluded minors for the class of $\mathbb{H}_5$-representable matroids is seen as an important step towards understanding $\GF(5)$-representable matroids; work is ongoing towards both of these goals (see \cite{CMWvZ2015,BCOSW2018}, for example).
%The class of $3$-regular matroids coincides with the class of $\mathbb{H}_5$-representable matroids, and that an excluded minor for one of these classes has at most 15 elements \cite{BOSW2023b}.

To find the excluded minors for $\GF(5)$-representability of size at most $10$, using the approach described earlier, the first step is to find the excluded minors for the class of $\mathbb{H}_5$-representable matroids of size at most $10$.
However, to complement the aforementioned known bound on an excluded minor for the class of $\mathbb{H}_5$-representable matroids, we extend our search beyond matroids of size $10$.
We do not quite match the bound of $15$, however; we find all excluded minors for the class having at most $13$ elements.
%Towards the excluded-minor characterisation for $2$-regular matroids, Brettell and Pendavingh~\cite{BP20} computed the excluded minors for the class of $2$-regular matroids of size at most 15.

\begin{theorem}
  \label{thm2}
  There are exactly $33$ matroids that are excluded minors for the class of $\mathbb{H}_5$-representable matroids and have at most $13$ elements.
\end{theorem}

\noindent
Descriptions of these $33$ matroids appear in \cref{h5u3sec} (see \cref{h5exminors}).

For \cref{thm1}, we borrow from the approach taken in \cite{BP20}; however, to find the excluded minors for $\GF(5)$-representability of size at most $10$, the computational techniques we require are less advanced, since we are only dealing with matroids of size at most $10$.  On the other hand, as we are dealing with much larger sets of excluded minors and stabilizers, careful book-keeping is required.
For \cref{thm2}, the more advanced techniques employed in \cite{BP20} become useful; in particular, when considering matroids of size $n \ge 11$, we can dramatically speed up the computations by only considering compatible pairs of extensions of matroids of size $n-2$ (rather than all extensions of matroids of size $n-1$).
For this approach to be exhaustive while only keeping track of $3$-connected matroids with a certain $N$-minor, we require a splitter theorem for $N$-detachable pairs~\cite{bww3}.
Moreover, there are particular $12$-element matroids for which we cannot guarantee the existence of $N$-detachable pairs; we study these further, to show they are not excluded minors for the class of $\mathbb{H}_5$-representable matroids.

%Finally, Pendavingh and van Zwam observed that a $3$-regular matroid is $\mathbb{H}_5$-representable \cite[Figure~2]{PvZ2010b}.
%We prove, as \cref{3reglemma}, that the Hydra-$5$ partial field is in fact isomorphic to the $3$-regular partial field; so, in particular, a matroid is $\mathbb{H}_5$-representable if and only if it is $3$-regular.

One final observation from this work is that an excluded-minor characterisation for $\mathbb{H}_4$-representable matroids, though extremely difficult, could perhaps also be achievable, putting it in a similar category to dyadic matroids and matroids representable over all fields of size at least four.

This paper is structured as follows.
In the next section, we review partial fields, stabilizers, and other related notions that we need to describe the methodology, and argue the correctness, of our approach.
In the two subsequent sections, we consider the innermost layer of the Hydra hierarchy, $\mathbb{H}_5$-representable matroids.
First, we study $12$-element matroids that have no $N$-detachable pairs, in \cref{connsec}. %, to understand \remark{their representability properties.}
Then, in \cref{h5u3sec}, we obtain the excluded minors for the class of $\mathbb{H}_5$-representable matroids on at most $13$ elements.
In \cref{hmidsec}, we consider the excluded minors for the remaining layers of the Hydra hierarchy: that is, the class of $\mathbb{H}_i$-representable matroids for $i \in \{1,2,3,4\}$.
We finish with some final remarks in \cref{finalsec}, including the observation that each Hydra partial field is universal.
%regarding the excluded minors for $\GF(5)$-representable matroids. %, before closing with some final remarks in \cref{finalsec}.

Our notation and terminology follows Oxley~\cite{oxley}, %with a few exceptions.
unless specified otherwise.
We note that some of the results of this paper have also appeared, informally, on the MatroidUnion blog\footnote{See http://matroidunion.org/?p=3815.}.

%TODO comment on ``point'' of the paper? i.e. describe the approach used, and allow one to attempt to reproduce the results.  Towards this end, have provided ``breadcrumbs'' (e.g, when enumerating all $3$-connected $\mathbb{H}_i$-representable matroids with an $\mathcal{N}$-minor, for a set of $\mathbb{H}_i$-stabilizers $\mathcal{N}$, providing tables with the total number of such matroids of various ranks and sizes.

\section{Preliminaries}
\label{prelims}

For a matroid $M$ and a set of matroids $\mathcal{N}$, we say that $M$ has an $\mathcal{N}$-minor if $M$ has a minor isomorphic to $N$ for some $N \in \mathcal{N}$.

\subsection*{Partial fields}

A \textit{partial field} is a pair $(R, G)$, where $R$ is a commutative ring with unity, and $G$ is a subgroup of the group of units of $R$ such that $-1 \in G$.
Note that $(\mathbb{F}, \mathbb{F}^*)$ is a partial field for any field $\mathbb{F}$.
%If $\mathbb{P}=(R,G)$ is a partial field, then we write $p\in \mathbb{P}$ whenever $p\in G\cup \{0\}$.

For disjoint sets $X$ and $Y$, we refer to a matrix with rows labelled by elements of $X$ and columns labelled by elements of $Y$ as an \emph{$X \times Y$ matrix}.
Let $\mathbb{P}=(R, G)$ be a partial field, and let $A$ be an $X\times Y$ matrix with entries from $R$. Then $A$ is a $\mathbb{P}$-\textit{matrix} if every square submatrix of $A$ with a non-zero determinant is in $G$. If $X'\subseteq X$ and $Y'\subseteq Y$, then we write $A[X',Y']$ to denote the submatrix of $A$ with rows labelled by $X'$ and columns labelled by $Y'$.
%When $X$ and $Y$ are disjoint, if $Z\subseteq X\cup Y$, then we denote by $A[Z]$ the submatrix induced by $X\cap Z$ and $Y\cap Z$, and we denote by $A-Z$ the submatrix induced by $X-Z$ and $Y-Z$.
%
We also %write $\mathbb{P}$ to denote the set $G \cup \{0\}$, so we
say that $p$ is an element of $\mathbb{P}$, and write $p \in \mathbb{P}$, if $p \in G$ or $p=0$. 
 
\begin{lemma}[{\cite[Theorem 2.8]{PvZ2010b}}]
\label{pmatroid}
Let $\mathbb{P}$ be a partial field, and let $A$ be an $X\times Y$ $\mathbb{P}$-matrix, where $X$ and $Y$ are disjoint sets. Let
\begin{equation*}
\mathcal{B}=\{X\}\cup \{X\triangle Z : |X\cap Z|=|Y\cap Z|, \det(A[X\cap Z,Y\cap Z])\neq 0\}. 
\end{equation*}
 Then $\mathcal{B}$ is the family of bases of a matroid on ground set $X\cup Y$.
\end{lemma}

For an $X\times Y$ $\mathbb{P}$-matrix $A$, we let $M[A]$ denote the matroid in \cref{pmatroid}, and say that $A$ is a \emph{$\mathbb{P}$-representation} of $M[A]$.
Note that this is sometimes known as a reduced $\mathbb{P}$-representation in the literature; here, all representations will be ``reduced'', so we simply refer to them as representations.
A matroid~$M$ is $\mathbb{P}$-\textit{representable} if there exists some $\mathbb{P}$-matrix $A$ such that %is a $\mathbb{P}$-representation of $M$.
$M \cong M[A]$.
%We refer to a matroid~$M$ together with a $\mathbb{P}$-representation $A$ of $M$ as a \emph{$\mathbb{P}$-represented} matroid.

For partial fields $\mathbb{P}_1$ and $\mathbb{P}_2$, a function
$\phi : \mathbb{P}_1 \rightarrow \mathbb{P}_2$ is a \emph{homomorphism} if
\begin{enumerate}
  \item $\phi(1) = 1$,
  \item $\phi(pq) = \phi(p)\phi(q)$ for all $p, q \in \mathbb{P}_1$, and
  \item $\phi(p) +\phi(q) = \phi(p +q)$ for all $p, q \in \mathbb{P}_1$ such that $p +q \in \mathbb{P}_1$.
\end{enumerate}
The existence of a %partial-field
homomorphism from $\mathbb{P}_1$ to $\mathbb{P}_2$ certifies that any $\mathbb{P}_1$-representable matroid is also $\mathbb{P}_2$-representable:
%The next lemma shows how a partial-field homomorphism can be used to shine light on the partial fields that a matroid is representable over.
%Let $\phi([a_{ij}])$ denote $[\phi(a_{ij})]$.

\begin{lemma}[{\cite[Corollary 2.9]{PvZ2010b}}]
  Let $\mathbb{P}_1$ and $\mathbb{P}_2$ be partial fields and let $\phi : \mathbb{P}_1 \rightarrow \mathbb{P}_2$ be a %nontrivial
  homomorphism.
  If a matroid is $\mathbb{P}_1$-representable, then it is also $\mathbb{P}_2$-representable.
  In particular, 
  %if A is a $\mathbb{P}_1$-representation of a matroid $M$, then $\phi(A)$ is a $\mathbb{P}_2$-representation of $M$.
  if $[a_{ij}]$ is a $\mathbb{P}_1$-representation of a matroid $M$, then $[\phi(a_{ij})]$ is a $\mathbb{P}_2$-representation of $M$.
\end{lemma}

%It is known that
Representability over a partial field can be used to characterise representability over each field in a set of fields.  Indeed,
for any finite set of fields $\mathcal{F}$, there exists a partial field~$\mathbb{P}$ such that a matroid is %representable over each field $\mathbb{F} \in \mathcal{F}$
$\mathcal{F}$-representable
if and only if it is $\mathbb{P}$-representable \cite[Corollary~2.20]{PvZ2010a}.

%Pendavingh and van Zwam introduced the concept of a \emph{universal partial field} $\mathbb{P}_M$ of a matroid $M$; a partial field with the property that $M$ is $\mathbb{P}$-representable if and only if there is a partial-field homomorphism $\phi : \mathbb{P}_M \rightarrow \mathbb{P}$.
%As this property is sufficient for the purposes of this paper, we define universal partial fields in this way. The equivalence of this definition (and the uniqueness of a universal partial field) follow from (cite).

  Let $M$ be a matroid.
  Pendavingh and van Zwam described~\cite[Section~4.2]{PvZ2010b} a canonical construction of a partial field with the property that for every partial field~$\mathbb{P}$, the matroid $M$ is $\mathbb{P}$-representable if and only if there exists a homomorphism $\phi : \mathbb{P}_M \rightarrow \mathbb{P}$ (see also \cite{BL21}).
  We call the partial field $\mathbb{P}_M$ the \emph{universal partial field of $M$}.

\subsection*{Stabilizers}

Let $\mathbb{P} = (R,G)$ be a partial field, and let $A$ and $A'$ be $\mathbb{P}$-matrices.
We say that $A$ and $A'$ are \emph{scaling equivalent} if $A'$ can be obtained from $A$ by scaling rows and columns by elements of $G$.
If $A'$ can be obtained from $A$ by scaling, pivoting, and permuting rows and permuting columns (permuting labels at the same time), then we say that $A$ and $A'$ are \emph{projectively equivalent}.
If $A'$ can be obtained from $A$ by scaling, pivoting, permuting rows and permuting columns, and also applying automorphisms of $\mathbb{P}$, then we say that $A$ and $A'$ are \emph{algebraically equivalent}.

We say that a matroid $M$ is \emph{uniquely $\mathbb{P}$-representable} if any two $\mathbb{P}$-representations of $M$ are algebraically equivalent.
On the other hand, two $\mathbb{P}$-representations of $M$ are \emph{inequivalent} if they are not algebraically equivalent.

Let $M$ and $N$ be $\mathbb{P}$-representable matroids, where $M$ has an $N$-minor.
Then \emph{$N$ stabilizes $M$ over $\mathbb{P}$} if for any scaling-equivalent $\mathbb{P}$-representations $A_1'$ and $A_2'$ of $N$ that extend to $\mathbb{P}$-representations $A_1$ and $A_2$ of $M$, respectively, $A_1$ and $A_2$ are scaling equivalent.

For a partial field~$\mathbb{P}$, let $\mathcal{M}(\mathbb{P})$ be the class of matroids representable over %a partial field
$\mathbb{P}$.
%A matroid $N \in \mathcal{M}(\mathbb{P})$ stabilizes $\mathcal{M}(\mathbb{P})$ if, for any $3$-connected matroid $M \in \mathcal{M}(\mathbb{P})$, a $\mathbb{P}$-representation of $M$ is uniquely determined by a representation of any one of its $N$-minors.
A matroid $N \in \mathcal{M}(\mathbb{P})$ is a \emph{$\mathbb{P}$-stabilizer} if, for any $3$-connected matroid $M \in \mathcal{M}(\mathbb{P})$ having an $N$-minor, the matroid $N$ stabilizes $M$ over $\mathbb{P}$.
%
%We have a $3$-connected matroid $N$ with the property that every $\mathbb{P}$-representation of $N$ extends uniquely to a $\mathbb{P}$-representation of any $3$-connected $\mathbb{P}$-representable matroid having $N$ as a minor.
%Such a matroid $N$ is called a \emph{strong stabilizer} for the class of $\mathbb{P}$-representable matroids.
%
%Whittle proved that one can verify a $3$-connected matroid %$N$
%is a $\mathbb{P}$-stabilizer %, % for a class of matroids, %$\mathcal{M}$,
%%one need only check that $N$ stabilizes $3$-connected matroids in $\mathcal{M}(\mathbb{P})$ that are at most two elements away from $N$.
%by a finite case check.
%
%\begin{theorem}[Whittle's Stabilizer Theorem~\cite{Whittle1999}]
  %\label{whittle-stabilizer}
  %Let $\mathbb{P}$ be a partial field, and let $M$ and $N$ be $3$-connected $\mathbb{P}$-representable matroids, where $M$ has an $N$-minor.
  %Then exactly one of the following holds:
  %\begin{enumerate}
    %\item $N$ stabilizes $M$ over $\mathbb{P}$, or
    %\item $M$ has a $3$-connected minor $M'$ such that
      %\begin{itemize}
        %\item $N$ does not stabilize $M'$ over $\mathbb{P}$,
        %\item $N$ is isomorphic to $M'/x$, $M' \ba y$, or $M' /x \ba y$, for some $x, y \in E(M')$, and
        %\item if $N$ is isomorphic to $M'/x \ba y$ then at least one of $M'/x$, $M' \ba y$ is $3$-connected.
      %\end{itemize}
  %\end{enumerate}
%\end{theorem}
%
Following Geelen et al.~\cite{GOVW1998}, we say that %a matroid $N$ \emph{strongly stabilizes $M$ over $\mathbb{P}$} if $N$ stabilizes $M$ over $\mathbb{P}$, and every $\mathbb{P}$-representation of $N$ extends to a $\mathbb{P}$-representation of $M$.
%We say that $N$ is a \emph{strong $\mathbb{P}$-stabilizer} %for $\mathcal{M}(\mathbb{P})$
%if $N$ is a $\mathbb{P}$-stabilizer and $N$ strongly stabilizes every matroid in $\mathcal{M}(\mathbb{P})$ with an $N$-minor.
a matroid $N$ is a \emph{strong $\mathbb{P}$-stabilizer} if, for any $3$-connected matroid $M \in \mathcal{M}(\mathbb{P})$ having an $N$-minor, the matroid $N$ stabilizes $M$ over $\mathbb{P}$, and every $\mathbb{P}$-representation of $N$ extends to a $\mathbb{P}$-representation of $M$.

%\begin{proposition}[{\cite[Proposition~3.1]{GOVW1998}}]
  %Let $\mathbb{P}$ be a partial field, and let $N$ be a $\mathbb{P}$-stabilizer.
  %If $N'$ is a $3$-connected matroid with an $\{N,N^*\}$-minor, and $N'$ is uniquely representable over $\mathbb{P}$, then $N'$ is a strong $\mathbb{P}$-stabilizer. %for $\mathcal{M}(\mathbb{P})$.
%\end{proposition}

\subsection*{The Hydra-\texorpdfstring{$i$}{i} partial fields}

The Hydra-$5$ partial field is:
$$\mathbb{H}_5 = \left(\mathbb{Q}(\alpha, \beta, \gamma), \left<-1, \alpha, \beta, \gamma, \alpha - 1, \beta - 1, \gamma - 1, \alpha - \gamma, \gamma - \alpha\beta, (1 - \gamma) - (1 - \alpha)\beta\right>\right)$$ where $\alpha,\beta,\gamma$ are indeterminates.

The Hydra-$4$ partial field is:
$$\mathbb{H}_4 = \left(\mathbb{Q}(\alpha,\beta), \left<-1,\alpha,\beta,\alpha-1,\beta-1,\alpha\beta-1,\alpha+\beta-2\alpha\beta\right>\right)$$ where $\alpha,\beta$ are indeterminates.

The Hydra-$3$ partial field is:
$$\mathbb{H}_3 = \left(\mathbb{Q}(\alpha), \left<-1,\alpha,1-\alpha,\alpha^2-\alpha+1\right>\right)$$ where $\alpha$ is an indeterminate.

The Hydra-$2$ partial field is:
$$\mathbb{H}_2 = \left(\mathbb{Z}\left[i,1/2\right], \left<i,1-i\right>\right)$$ where $i$ is a root of $x^2 +1=0$.

We also let $\mathbb{H}_1 = \GF(5)$, and $\mathbb{H}_6 = \mathbb{H}_5$.

%It is easily seen that for each $i \in \{1,2,3,4,5\}$, there is a homomorphism from $\mathbb{H}_{i+1}$ to $\mathbb{H}_i$.
For each $i \in \{1,2,3,4\}$, there is a homomorphism $\phi_i :  \mathbb{H}_{i+1} \rightarrow \mathbb{H}_i$.  For example, let $\phi_4$ be determined by $\phi_4(\alpha) = \alpha$, $\phi_4(\beta) = \frac{1}{1-\alpha}$, and $\phi_4(\gamma) = \frac{\alpha(\beta-1)}{\beta(1-\alpha)}$; let $\phi_3$ be determined by $\phi_3(\alpha) = \alpha$ and $\phi_4(\beta) = \frac{\alpha}{\alpha-1}$; let $\phi_2$ be determined by $\phi_2(\alpha) = i$; and let $\phi_1$ be determined by $\phi_1(i) = 2$.

\begin{lemma}[{\cite[Theorem~5.7]{PvZ2010b}}]
  \label{hi}
  Let $M$ be a $\GF(5)$-representable matroid that is $3$-connected and has a $\{U_{2,5},U_{3,5}\}$-minor, and let $i \in \{1,2,3,4,5,6\}$.  The matroid $M$ is $\mathbb{H}_i$-representable if and only if $M$ has at least $i$ inequivalent representations over $\GF(5)$.
\end{lemma}

For $\mathbb{H}_2$-representable matroids that are $3$-connected with a $\{U_{2,5},U_{3,5}\}$-minor, the two inequivalent $\GF(5)$-representations can be obtained by substituting $i \in \{2,3\}$.
For $\mathbb{H}_3$-representable matroids, three inequivalent $\GF(5)$-representations can be obtained by substituting $\alpha \in \{2,3,4\}$.
For $\mathbb{H}_4$-representable matroids, four inequivalent $\GF(5)$-representations can be obtained by substituting $$(\alpha,\beta) \in \{(2,2), (3,3), (3,4), (4,3)\}.$$
For $\mathbb{H}_5$-representable matroids, six inequivalent $\GF(5)$-representations can be obtained by substituting $$(\alpha,\beta,\gamma) \in \{(2,3,3), (3,2,2), (4,3,3), (2,4,4), (3,2,4), (4,4,2)\}.$$

\begin{corollary}[see {\cite[Lemma~5.17]{PvZ2010b}}]
  \label{h5h6}
Let $M$ be a $\GF(5)$-representable matroid that is $3$-connected and has a $\{U_{2,5},U_{3,5}\}$-minor.  The matroid $M$ is $\mathbb{H}_5$-representable if and only if $M$ has six inequivalent representations over $\GF(5)$.
\end{corollary}

A matroid is \emph{dyadic} if it is representable over $\GF(3)$ and $\GF(5)$.
Equivalently, a matroid is dyadic if it is $\mathbb{D}$-representable, where $\mathbb{D}$ is the partial field $$(\mathbb{Z}[1/2],\left<-1,2\right>).$$
There is a homomorphism from $\mathbb{D}$ to every field with characteristic not two.
There is also a homomorphism from $\mathbb{D}$ to $\mathbb{H}_2$.
Note that, although a dyadic matroid is $\mathbb{H}_2$-representable, it has no $\{U_{2,5},U_{3,5}\}$-minor, so \cref{hi} does not apply.

\subsection*{The \texorpdfstring{$k$}{k}-regular partial fields}

For a non-negative integer $k$, the \emph{$k$-regular} partial field is:
$$\mathbb{U}_k = (\mathbb{Q}(\alpha_1,\dots,\alpha_k), \left<\{x-y : x,y \in \{0,1,\alpha_1,\dotsc,\alpha_k\}\textrm{ and }x \neq y\}\right>),$$ where $\alpha_1,\dotsc,\alpha_k$ are indeterminates.
A matroid is \emph{$k$-regular} if it is representable over the $k$-regular partial field.

Note that $\mathbb{U}_k$ is the universal partial field of $U_{2,3+k}$ \cite[Theorem 3.3.24]{vanZwam2009}.
It is easy to see that, for any non-negative integer $k$, there is a homomorphism from $\mathbb{U}_k$ to $\mathbb{U}_{k+1}$.

In particular, the \emph{$0$-regular} (or just \emph{regular}) partial field is
$$\mathbb{U}_0 = (\mathbb{Q}, \{1,-1\});$$
the \emph{$1$-regular} (or \emph{near-regular}) partial field is
$$\mathbb{U}_1 = (\mathbb{Q}(\alpha), \left<-1,\alpha,\alpha-1\right>),$$
where $\alpha$ is an indeterminate; the \emph{$2$-regular} partial field is
$$\mathbb{U}_2 = (\mathbb{Q}(\alpha_1,\alpha_2), \left<-1,\alpha_1,\alpha_2,\alpha_1-1,\alpha_2-1,\alpha_1-\alpha_2\right>),$$ where $\alpha_1,\alpha_2$ are indeterminates; and the \emph{$3$-regular} partial field is
$$\mathbb{U}_3 = (\mathbb{Q}(\alpha_1,\alpha_2,\alpha_3), \left<-1,\alpha_1,\alpha_2,\alpha_3,\alpha_1-1,\alpha_2-1,\alpha_3-1,\alpha_1-\alpha_2,\alpha_1-\alpha_3,\alpha_2-\alpha_3\right>),$$ where $\alpha_1,\alpha_2,\alpha_3$ are indeterminates.

%We next recall that a matroid is $\mathbb{H}_5$-representable if and only if it is $3$-regular~\cite{BOSW2023b}.
Van Zwam~\cite{vanZwam2009} observed that there is a homomorphism from $\mathbb{U}_3$ to $\mathbb{H}_5$.
%We say that partial fields $\mathbb{P}$ and $\mathbb{P}'$ are \emph{isomorphic} if there exists a homomorphism $\phi : \mathbb{P} \rightarrow \mathbb{P}'$ such that $\phi$ is a bijection, and \ldots.
It turns out that there is a homomorphism $\phi : \mathbb{U}_3 \rightarrow \mathbb{H}_5$ where $\phi$ is a bijection~\cite{BOSW2023b}.

\begin{lemma}[{\cite[Lemma~2.25]{BOSW2023b}}]
  \label{3reglemma}
  The partial fields $\mathbb{H}_5$ and $\mathbb{U}_3$ are isomorphic.
  In particular, a matroid is $3$-regular if and only if it is $\mathbb{H}_5$-representable.
\end{lemma}

The next corollary then follows from \cref{hi,3reglemma,h5h6}. % \cite[Lemma 5.16]{PvZ2010b}.  (For the equivalence of (ii) and (iii), see \cite[Theorem~1.3]{PvZ2010b}.)

\begin{corollary}
  Let $M$ be a $3$-connected matroid with a $\{U_{2,5}, U_{3,5}\}$-minor.
  The following are equivalent:
  \begin{enumerate}
    \item $M$ is $3$-regular.
    \item $M$ has at least five inequivalent $\GF(5)$-representations.
    \item $M$ has precisely six inequivalent $\GF(5)$-representations.
  \end{enumerate}
\end{corollary}
Alternatively, a matroid is $3$-regular if and only if it is either near-regular or has six inequivalent $\GF(5)$-representations.

%(Stefan had that there exists a PF homomorphism from U_3 -> H_5 in his thesis, but not the other direction).

\begin{lemma}[{\cite[Lemmas 6.2.5 and 7.3.15]{vanZwam2009}}]
  \label{hydrastabs}
  For $i \in \{1, 2, 3, 4\}$, let $N$ be an excluded minor for the class of $\mathbb{H}_{i+1}$-representable matroids, such that $N$ is $\mathbb{H}_i$-representable. Then $N$ is a strong $\mathbb{H}_i$-stabilizer. % for the class of Hi-representable matroids
\end{lemma}

\subsection*{Other partial fields}

Although not of central importance, we also refer to the following partial fields:
$$\mathbb{K}_2 = (\mathbb{Q}(\alpha), \left<-1,\alpha-1,\alpha,\alpha+1\right>)$$
where $\alpha$ is an indeterminate, and

$$\mathbb{S} = (\mathbb{Z}[\zeta], \left<\zeta\right>),$$ where $\zeta$ is a root of $x^2-x+1=0$.
See \cite{vanZwam2009} for further details. % on these partial fields.

\subsection*{Delta-wye exchange}

Let $M$ be a matroid with a triangle $T=\{a,b,c\}$.
Consider a copy of $M(K_4)$ having $T$ as a triangle with $\{a',b',c'\}$ as the complementary triad labelled such that $\{a,b',c'\}$, $\{a',b,c'\}$ and $\{a',b',c\}$ are triangles.
Let $P_{T}(M,M(K_4))$ denote the generalised parallel connection of $M$ with this copy of $M(K_4)$ along the triangle $T$.
Let $M'$ be the matroid $P_{T}(M,M(K_4))\backslash T$ where the elements $a'$, $b'$ and $c'$ are relabelled as $a$, $b$ and $c$ respectively.
The matroid $M'$ is said to be obtained from $M$ by a \emph{\dY\ exchange} on the triangle~$T$, and is denoted $\Delta_T(M)$.
Dually, $M''$ is obtained from $M$ by a \emph{\Yd\ exchange} on the triad $T^*=\{a,b,c\}$ if $(M'')^*$ is obtained from $M^*$ by a \dY\ exchange on $T^*$.  The matroid $M''$ is denoted $\nabla_{T^*}(M)$.
%Dually, letting $M$ be a matroid with a triad $T^*$, we say $\nabla_T(M) = \Delta_{T^*}(M^*)$ is obtained from $M$ by a \emph{\Yd\ exchange} on the triad $T^*=\{a,b,c\}$.

We say that a matroid $M_1$ is \emph{$\Delta Y$-equivalent} to a matroid $M_0$ if $M_1$ can be obtained from $M_0$ by a sequence of \dY\ and \Yd\ exchanges on coindependent triangles and independent triads, respectively.
We let $\Delta(M)$ denote the set of matroids that are $\Delta Y$-equivalent to $M$.
For a set of matroids $\mathcal{M}$, we use $\Delta(\mathcal{M})$ to denote $\bigcup_{M \in \mathcal{M}} \Delta(M)$.
We also use $\Delta^*(M)$ to denote $\Delta(\{M,M^*\})$, and use $\Delta^*(\mathcal{M})$ to denote $\bigcup_{M \in \mathcal{M}} \Delta^*(M)$.

Oxley, Semple, and Vertigan proved that the set of excluded minors for $\mathbb{P}$-representability is closed under \dY\ exchange (see also \cite{AO1993}), and, more generally, segment-cosegment exchange.
\begin{proposition}[{\cite[Theorem~1.1]{OSV2000}}]
  \label{osvdelta}
  Let $\mathbb{P}$ be a partial field, and let $M$ be an excluded minor for the class of $\mathbb{P}$-representable matroids.
  If $M' \in \Delta^*(M)$, then $M'$ is an excluded minor for the class of $\mathbb{P}$-representable matroids.
\end{proposition}

\subsection*{Proxies}

Following \cite{BP20}, in order to simulate a matroid representation over a partial field, we use a representation over a finite field where we have constraints on the subdeterminants appearing in the representation.
We briefly recap this theory (for a more detailed account, see \cite[Section~3]{BP20}).

Let $\mathbb{P}=(R,G)$ be a partial field.
%We write $p\in \mathbb{P}$ when $p\in G\cup \{0\}$. % and $P\subseteq \mathbb{P}$ when $P\subseteq G\cup \{0\}$.
We say that $p \in \mathbb{P}$ is \emph{fundamental} if $1-p \in \mathbb{P}$.
We denote the set of fundamentals of $\mathbb{P}$ by $\mathfrak{F}(\mathbb{P})$.
Let $A$ be a $\mathbb{P}$-matrix.
For a $\mathbb{P}$-matrix $A'$, we write $A' \preceq A$ if $A'$ is a submatrix of a $\mathbb{P}$-matrix that is projectively equivalent to $A$.
The \emph{cross ratios} of $A$ are $$\Cr(A) = \left\{p : \begin{bmatrix}1 & 1 \\ p & 1\end{bmatrix} \preceq A \right\}.$$

Let $F \subseteq \mathfrak{F}(\mathbb{P})$.
 We say that the $\mathbb{P}$-matrix $A$ is {\em $F$-confined} if $\Cr(A) \subseteq F \cup \{0,1\}$.
 If $A$ is an $F$-confined $\mathbb{P}$-matrix  and $\phi: \mathbb{P}\rightarrow \mathbb{P}'$ is a %partial-field
 homomorphism, then $M[A]=M[\phi(A)]$ and $$\Cr(\phi(A)) \subseteq \phi(F),$$ so that $\phi(A)$ is an $\phi(F)$-confined representation over $\mathbb{P}'$.

%Let $\mathbb{P}$ be a partial field.
For a finite field $\mathbb{F}$ and %partial-field
homomorphism $\phi : \mathbb{P} \rightarrow \mathbb{F}$, 
let $F_\phi = \phi(\mathfrak{F}(\mathbb{P}))$.
We say that $(\mathbb{F},\phi)$ is a \emph{proxy} for $\mathbb{P}$ if %$\phi$ can be lifted in the sense of \cref{thm:proxy}.
for any $F_\phi$-confined $\mathbb{F}$-matrix~$A$, there exists a $\mathbb{P}$-matrix~$A'$ such that $\phi(A')$ is scaling-equivalent to $A$.
%For example, the proof of \cref{cor:dyad} shows that $(\GF(11),d)$ is a proxy for $\mathbb{D}$.
The following is a consequence of \cite[Theorem~3.5]{BP20}.

\begin{lemma}
  Let $i \in \{2,3,4,5\}$.
  There exists a prime $p$ and a %partial-field
  homomorphism $\phi : \mathbb{H}_i \rightarrow \GF(p)$ such that $(\GF(p),\phi)$ is a proxy for $\mathbb{H}_i$.
\end{lemma}

In particular, when $(\GF(p),\phi)$ is a proxy for $\mathbb{H}_i$, %for such a prime $p$ and homomorphism $\phi$,
a matroid $M$ is $\mathbb{H}_i$-representable if and only if $M$ has an $F_\phi$-confined $\GF(p)$-representation.
  %Moreover, $A$ is a $\mathbb{H}_i$-representation of $M$ if and only if $\phi(A)$ is a $F_\phi$-confined $\GF(p)$-representation of $M$.
In \cref{proxytable}, we provide the proxies for each Hydra-$i$ partial field that we used in our computations.
%H4 and H5 may not be smallest possible finite field proxies (but H2? and H3 are)

\begin{table}[htb]
  $\begin{array}{lll}
    \text{Partial field }&\text{Finite Field }&\text{Homomorphism}\\
    \hline
    \mathbb{H}_2& \GF(29) & i\mapsto 12\\
    \mathbb{H}_3& \GF(151) & \alpha\mapsto 4\\ %checked :)
    \mathbb{H}_4& \GF(947) &\alpha\mapsto 272, \beta\mapsto 928\\  %checked :)
    \mathbb{H}_5& \GF(3527) &\alpha\mapsto 1249, \beta\mapsto 295, \gamma\mapsto 3517\\
  \end{array}$
  \caption{\label{fig:conff} Proxies for the Hydra-$i$ partial fields.}
  \label{proxytable}
\end{table}

\subsection*{Splitter theorems}

For our computations, we keep track only of $3$-connected matroids with an $\mathcal{N}$-minor (where the matroids in $\mathcal{N}$ are strong stabilizers).
Thus, we require the following consequence of Seymour's Splitter Theorem (see, for example, \cite[Corollary~12.2.2]{oxley}).

\begin{lemma}[{\cite[Corollary~2.8]{BP20}}]
  \label{nearregconn}
  Let $M$ be a $3$-connected matroid with a proper $N$-minor, where $N$ is not near-regular.
  Then, for some $\{M',N'\} \in \{(M,N),(M^*,N^*)\}$, there exists an element $e \in E(M')$ such that $M' \ba e$ is $3$-connected and has an $N'$-minor.
\end{lemma}

Hall, Mayhew, and van Zwam~\cite{HMvZ2011} proved an excluded-minor characterisation for near-regular matroids.
Two of the excluded minors for this class are $U_{2,5}$ and $U_{3,5}$.
In this paper, we deal primarily with $3$-connected matroids that have a $\{U_{2,5},U_{3,5}\}$-minor, so \cref{nearregconn} applies.

To find the excluded minors for $\mathbb{H}_5$-representability on up to $13$ elements, we use the technique referred to as ``splicing'' in \cite[Section~4.5]{BP20}.
Briefly, when $M'$ is a matroid, $M_e$ is an extension of $M'$ by an element $e$, and $M_f$ is an extension of $M'$ by an element $f$, then the extensions $M_e$ and $M_f$ are \emph{compatible} if there exists a matroid $M$ on $E(M') \cup \{e,f\}$ such that $M \ba f = M_e$ and $M \ba e = M_f$.
(A characterisation of compatible extensions was given by Cheung \cite{Cheung1974}.)
In this case, following \cite{BP20}, we say that $M$ is a \emph{splice} of $M_e$ and $M_f$.  (We note that a ``splice'' has a different meaning in \cite{BS09}.)
For a positive integer $n$, an $(n+1)$-element matroid can be constructed as the splice of two $n$-element matroids, $M_e$ and $M_f$ say, where $M_e$ and $M_f$ are both single-element extensions of some $(n-1)$-element matroid $M'$.

Starting from some set $\mathcal{N}$ of ``seed matroids'', we build a catalog of the $3$-connected matroids in the class that have an $N$-minor for some $N \in \mathcal{N}$.
Suppose $M$ is a matroid that should appear in the catalog, where $M$ has an $N$-minor for $N \in \mathcal{N}$.
In order to ensure that $M$ is obtained as the splice of two matroids $M_a$ and $M_b$, we require a guarantee that, for some pair $\{a,b\} \subseteq E(M)$, the matroid $M \ba a,b$ is also in the catalog.
In fact, by closing the catalog under duality, it suffices that one of $M$ and $M^*$ appears in the catalog.
Thus, we require a pair $\{a,b\} \subseteq E(M)$, such that either $M \ba a,b$ or $M / a,b$ is $3$-connected and has an $N$-minor;
such a pair $\{a,b\}$ is known as an \emph{$N$-detachable pair}~\cite{bww3}.
%The next theorem describes the conditions under which $M$ has an $N$-detachable pair (the precise structure of the undefined terms in (iii) and (iv) is given in \cite[Definitions~1.1--1.5 and 4.1--4.2]{bww3}).
%Let $M$ be a $3$-connected matroid, and let $N$ be a $3$-connected minor of $M$.
%A pair $\{a,b\} \subseteq E(M)$ is \emph{$N$-detachable} if either $M\ba a\ba b$ or $M/a/b$ is $3$-connected and has an $N$-minor. 

To describe the conditions under which $M$ has no $N$-detachable pairs, we require some definitions.
\begin{definition}
  Let $P$ be an exactly $3$-separating set in a matroid $M$, with $|P| = 2t \ge 6$.
Suppose $P$ has the following properties:
\begin{enumerate}[label=\rm(\alph*)]
  \item there is a partition $\{L_1,\dotsc,L_t\}$ of $P$ into pairs such that for all distinct $i,j\in\{1,\dotsc,t\}$, the set $L_i\cup L_j$ is a cocircuit,
  \item there is a partition $\{K_1,\dotsc,K_t\}$ of $P$ into pairs such that for all distinct $i,j\in\{1,\dotsc,t\}$, the set $K_i\cup K_j$ is a circuit,
  \item $M / p$ and $M \ba p$ are $3$-connected for each $p \in P$,
  \item for all distinct $i,j\in\{1,\dotsc,t\}$, the matroid $\si(M / a / b)$ is $3$-connected for any $a \in L_i$ and $b \in L_j$, and
  \item for all distinct $i,j\in\{1,\dotsc,t\}$, the matroid $\co(M \ba a \ba b)$ is $3$-connected for any $a \in K_i$ and $b \in K_j$.
\end{enumerate}
Then we say $P$ is a \emph{\spikey} of $M$.
\end{definition}

\begin{definition}
  \label{quadflowerdef}
  A matroid $M$ is a \emph{\quadflower} if $r(M) = r^*(M) = 6$, and there is a labelling $\bigcup_{\ell \in \{1,2,3\}} \{p_\ell,q_\ell,s_\ell,t_\ell\}$ of $E(M)$ such that,
%the elements of $E(M)$ can each be given a distinct label $q_k^\ell$, for $k\in\{1,2,3,4\}$ and $\ell \in\{1,2,3\}$, such that,
for all $(i,j) \in \{(1,2),(2,3),(3,1)\}$,
\begin{enumerate}
  \item $\{p_i,q_i,s_i,t_i\}$ is a circuit and a cocircuit,
  \item the sets $\{p_i,q_i,p_j,s_j\}$, $\{p_i,q_i,q_j,t_j\}$, $\{s_i,t_i,p_j,s_j\}$, and $\{s_i,t_i,q_j,t_j\}$ are circuits, and
  \item the sets $\{p_i,s_i,p_j,q_j\}$, $\{p_i,s_i,s_j,t_j\}$, $\{q_i,t_i,p_j,q_j\}$, and $\{q_i,t_i,s_j,t_j\}$ are cocircuits.
\end{enumerate}
We call $\bigcup_{\ell \in \{1,2,3\}} \{p_\ell,q_\ell,s_\ell,t_\ell\}$ the \emph{\quadflower\ labelling} of $M$, and say that $M$ has \emph{standard labelling} if $E(M)=\bigcup_{\ell \in \{1,2,3\}} \{p_\ell,q_\ell,s_\ell,t_\ell\}$ and the \quadflower\ labelling is obtained by applying the identity map.
\end{definition}
An example of a \quadflower\ is given in \cref{quadflowerfig}.
%Note that %, by orthogonality, a \quadflower\ has no circuits or cocircuits of size at most five other than the $4$-element circuits and cocircuits listed in (i)--(iii).
%However, a \quadflower does not uniquely determine a matroid; there are various possible combinations of $6$-element circuits.
%Moreover,
%a \quadflower\ is $3$-connected.

\begin{figure}[bhtp]
  \begin{tikzpicture}[rotate=90,xscale=0.84,yscale=0.6,line width=1pt]
    %\tikzset{VertexStyle/.append style = {minimum height=4,minimum width=4}}
    \clip (-2.5,-6) rectangle (3.0,2);

    \draw (0,0) -- (2,-2) -- (0,-4);

    \draw (0,0) -- (2.5,1.0) -- (2,-2);
    \draw (0,0) -- (2.25,-0.5);
    \draw (2,-2) -- (1.25,0.5);

    \draw (0,-4) -- (2.5,-5) -- (2,-2);
    \draw (0,-4) -- (2.25,-3.5);
    \draw (2,-2) -- (1.25,-4.5);

    \Vertex[x=1.25,y=0.5,LabelOut=true,L=$s_1$,Lpos=180]{c1}
    \Vertex[x=2.25,y=-0.5,LabelOut=true,L=$q_1$,Lpos=90]{c2}
    \Vertex[x=2.5,y=1.0,LabelOut=true,L=$p_1$,Lpos=180]{c3}
    \Vertex[x=1.5,y=-0.35,LabelOut=true,L=$t_1$,Lpos=135]{c4}

    \Vertex[x=1.25,y=-4.5,LabelOut=true,L=$q_2$]{c1}
    \Vertex[x=2.25,y=-3.5,LabelOut=true,L=$s_2$,Lpos=90]{c2}
    \Vertex[x=2.5,y=-5,LabelOut=true,L=$p_2$]{c3}
    \Vertex[x=1.5,y=-3.65,LabelOut=true,L=$t_2$,Lpos=45]{c4}

    \draw (0,-4) -- (0,0) -- (-2,-2) -- (0,-4);
    \draw (-1,-1) -- (0,-4);
    \draw (-1,-3) -- (0,0);

    \Vertex[LabelOut=true,L=$s_3$,x=-1,y=-3]{z1}
    \Vertex[LabelOut=true,L=$q_3$,x=-1,y=-1,Lpos=180]{z2}
    \Vertex[LabelOut=true,L=$p_3$,x=-2,y=-2]{z3}
    \Vertex[LabelOut=true,L=$t_3$,x=-0.67,y=-2,Lpos=90]{z4}
  \end{tikzpicture}
  \caption{A \quadflower.}
  \label{quadflowerfig}
\end{figure}

\begin{definition}
  \label{tcndef}
  A matroid $M$ is a \emph{\tcn} if $r(M) = r^*(M) = 6$, and there is a labelling
$\{e_1,e_1',e_2,e_2',\dotsc,e_6,e_6'\}$ of $E(M)$ such that
\begin{enumerate}
    \item the following sets are circuits:
$$\{e_1, e_2, e_3', e_4'\}, \{e_3, e_4, e_5', e_6'\}, \{e_5, e_6, e_1', e_2'\},$$
$$\{e_5, e_4, e_1', e_3'\}, \{e_1, e_3, e_2', e_6'\}, \{e_2, e_6, e_5', e_4'\},$$
$$\{e_2, e_3, e_5', e_1'\}, \{e_5, e_1, e_4', e_6'\}, \{e_4, e_6, e_2', e_3'\},$$
$$\{e_4, e_1, e_2', e_5'\}, \{e_2, e_5, e_3', e_6'\}, \{e_3, e_6, e_4', e_1'\},$$
$$\{e_3, e_5, e_4', e_2'\}, \{e_4, e_2, e_1', e_6'\}, \{e_1, e_6, e_3', e_5'\},$$
\item and the following sets are cocircuits:
$$\{e_3, e_4, e_1', e_2'\}, \{e_1, e_2, e_5', e_6'\}, \{e_5, e_6, e_3', e_4'\}$$
$$\{e_1, e_3, e_5', e_4'\}, \{e_5, e_4, e_2', e_6'\}, \{e_2, e_6, e_1', e_3'\}$$
$$\{e_5, e_1, e_2', e_3'\}, \{e_2, e_3, e_4', e_6'\}, \{e_4, e_6, e_5', e_1'\}$$
$$\{e_2, e_5, e_4', e_1'\}, \{e_4, e_1, e_3', e_6'\}, \{e_3, e_6, e_2', e_5'\}$$
$$\{e_4, e_2, e_3', e_5'\}, \{e_3, e_5, e_1', e_6'\}, \{e_1, e_6, e_4', e_2'\}.$$
\end{enumerate}
We call $\{e_1,e_1',e_2,e_2',\dotsc,e_6,e_6'\}$ the \emph{nest labelling} of $M$, and
%Moreover, we say that $M$ is a \tcn\ with \emph{standard labelling} if $E(M)=\{e_1,e_1',e_2,e_2',\dotsc,e_6,e_6'\}$ and the nest labelling of $M$ is the identity map.
say that $M$ has \emph{standard labelling} if $E(M)=\{e_1,e_1',e_2,e_2',\dotsc,e_6,e_6'\}$ and the nest labelling is obtained by applying the identity map.
\end{definition}

Note that a ``\tcn'' is referred to as a ``nest of twisted cubes'' in \cite{bww3}.
%Note that, by orthogonality, a \tcn\ has no circuits or cocircuits of size at most five other than the $4$-element circuits and cocircuits listed in (i) and (ii).

\begin{theorem}[{\cite[Corollary~6.3]{bww3}}]
  \label{detachthm}
  Let $M$ be a $3$-connected matroid and let $N$ be a $3$-connected minor of $M$ with $|E(M)| \ge 11$, $|E(N)| \ge 4$, and $|E(M)|-|E(N)| \ge 5$.
  Then either
  \begin{enumerate}
    \item $M$ has an $N$-detachable pair;\label{pc1}
    \item there is a matroid $M'$ obtained by performing a single $\Delta$-$Y$ or $Y$-$\Delta$ exchange on $M$ such that $M'$ has an $N$-minor and an $N$-detachable pair, or
    \item $M$ has a \spikey\ $P$ such that at most one element of $E(M)-E(N)$ is not in $P$, or\label{pc2}
    %\item there is some $P \subseteq E(M)$ such that $E(M)-E(N) \subseteq P$, and $P$ is a \spikey, \remark{or an \auging\ of a \spikey;} or\label{pc2}
    \item $|E(M)| =12$, and $M$ is either a \quadflower\ or a \tcn.\label{pc3}
  \end{enumerate}
\end{theorem}

The next lemma is used to handle the possibility that \cref{detachthm}(iii) holds.  It is a slight generalisation of {\cite[Lemma~7.2]{BCOSW2018}}, but can be proved using essentially the same argument as given there.
%\remark{TODO: check this}
 
\begin{lemma}[see {\cite[Lemma~7.2]{BCOSW2018}}]
  \label{spikelikes}
  Let $\mathbb{P}$ be a partial field, let $N$ be a non-binary $3$-connected strong stabilizer for the class of $\mathbb{P}$-representable matroids, and let $M$ be an excluded minor for the class of $\mathbb{P}$-representable matroids, where $M$ has an $N$-minor.
  If $M$ has a \spikey\ $P$ such that at most one element of $E(M)-E(N)$ is not in $P$, then $|E(M)| \le |E(N)| + 5$.
\end{lemma}

To handle the possibility that \cref{detachthm}(iv) holds, we need to consider what (partial) fields a \quadflower, or a \tcn, is representable over.  This analysis is done in the next section.

\section{Exceptional \texorpdfstring{$12$}{12}-element matroids}

\label{connsec}

When $M$ is a \quadflower\ or a \tcn, and $N$ is a $3$-connected minor of $M$, we cannot guarantee that $M$ has an $N$-detachable pair.
When enumerating potential excluded minors using splicing, these matroids will be skipped.
Thus, we need to be able to verify that these matroids are not excluded minors for the class of $\mathbb{H}_5$-representable matroids.
It suffices to verify that each of these matroids 
%This is a relatively straightforward exercise: for a \quadflower~$M$, it follows from the fact that $M$
contains either the Fano matroid $F_7$, the non-Fano matroid $F_7^-$, or the relaxation of the non-Fano matroid $F_7^=$, as a proper minor (due to the upcoming \cref{u3minors}).
However, so that these techniques can also be applied more generally (for partial fields other than $\mathbb{H}_5$), we briefly study \quadflower s and \tcns.  We will show that there are, up to isomorphism, three \quadflower s that are representable over some field (one is binary, one is dyadic, and one is representable over all fields of size at least four), whereas any other \quadflower\ has a proper minor that is not representable over any field.
Up to isomorphism, there are three \tcns: one is binary, one is dyadic, while the other is not representable over any field.

\subsection*{Quad-flowers}

First, we consider the circuits in a \quadflower.
We will see that the only sets that can be non-spanning circuits are the $4$-element circuits listed in the definition, and certain $6$-element sets.  Among these $6$-element sets, there are some that are circuits in every \quadflower, and some that can be either a circuit or a basis.
We introduce notation to refer to the latter sets.
Recall that a quad-flower with standard labelling is on the ground set $\bigcup_{\ell \in \{1,2,3\}} \{p_\ell,q_\ell,s_\ell,t_\ell\}$.
For $J \subseteq \{1,2,3\}$, let $X_J$ be the $6$-element set such that $\{p_i,t_i\} \subseteq X_J$ for each $i \in J$, and $\{q_i,s_i\} \subseteq X_J$ for each $i \in \{1,2,3\}-J$.

\begin{lemma}
  \label{quadflowerstructure}
  Let $M$ be a \quadflower\ with standard labelling.
  Then all of the following hold:
  \begin{enumerate}
    \item If $C$ is a non-spanning circuit of $M$, then $|C| \in \{4,6\}$.
    %\item If $C^*$ is a non-cospanning cocircuit of $M$, then $|C^*| \in \{4,6\}$.
    \item If $C$ is a circuit of $M$ with $|C|=4$, then $C$ is one of the $15$ circuits listed in \cref{quadflowerdef}(i)--(ii).
    \item For any $(i,j,k)\in \{(1,2,3),(2,3,1),(3,1,2)\}$ and $Z_i \in \{\{p_i,q_i\},\{s_i,t_i\}\}$, $Z_j \in \{\{p_j,t_j\},\{q_j,s_j\}\}$, and $Z_k \in \{\{p_k,s_k\},\{q_k,t_k\}\}$, the set $Z_i \cup Z_j \cup Z_k$ is a $6$-element circuit of $M$.
    \item If $C$ is a circuit of $M$ with $|C|=6$ but $C$ is not a circuit described in (iii), then $C=X_J$ for some $J \subseteq \{1,2,3\}$.
  \end{enumerate}
\end{lemma}
\begin{proof}
  By orthogonality and the fact that no circuit is properly contained in another, a \quadflower\ has no circuits of size at most five other than the $15$ circuits of size four listed in the definition.
  Since $r(M)=6$, this proves (i) and (ii).
  Moreover, by circuit elimination and orthogonality, it follows that (iii) holds.
  For (iv), let $X \subseteq E(M)$ with $|X|=6$ such that $X$ does not contain one of the $15$ circuits of size four, and $X$ is not a circuit as described in (iii).
  It follows that $X=X_J$ for some $J \subseteq \{1,2,3\}$.
\end{proof}

Although we focus on the circuits in a \quadflower, a similar statement to \cref{quadflowerstructure} can be made regarding the cocircuits.
It then follows easily that a \quadflower\ is $3$-connected.

%When $M$ is a \quadflower\ with standard labelling, there are $8$ possibilities for the set $X_J$, each of which could be either a circuit or a basis in $M$.

We distinguish three particular \quadflower s; we will show that (up to isomorphism) these are the only three \quadflower s that are representable over some field.
We say that
\begin{itemize}
  \item $M$ is a \emph{binary \quadflower} if $M$ has a \quadflower\ labelling such that $X_J$ is a circuit for all $J \subseteq \{1,2,3\}$;
  \item $M$ is a \emph{dyadic \quadflower} if $M$ has a \quadflower\ labelling such that, for $J \subseteq \{1,2,3\}$, the set $X_J$ is a circuit if and only if $|J| \in \{0,2\}$; and
  \item $M$ is a \emph{free \quadflower} if $M$ has a \quadflower\ labelling such that $X_J$ is a basis for all $J \subseteq \{1,2,3\}$.
\end{itemize}
The next lemma illustrates the reasoning behind the first two of these names.

\begin{lemma}
  \label{qflemma}
  Let $M$ be a \quadflower.  Then precisely one of the following holds:
  \begin{enumerate}
    \item $M$ is a binary \quadflower, and $M$ is $\mathbb{F}$-representable if and only if $\mathbb{F}$ has characteristic two,
    \item $M$ is a dyadic \quadflower, and $M$ is $\mathbb{F}$-representable if and only if $\mathbb{F}$ has characteristic not two,
    \item $M$ is a free \quadflower, and $M$ is $\mathbb{F}$-representable if and only if $|\mathbb{F}| \ge 4$, or
    \item $M$ has a proper minor $M'$ such that $M'$ is not representable over any field.
  \end{enumerate}
\end{lemma}
\begin{proof}
  We let $M$ be a \quadflower\ with standard labelling, so \cref{quadflowerstructure} holds.
  Consider $X_J$ for some $J \subseteq \{1,2,3\}$.
  There are $8$ possibilities for the set $X_J$, and each of these could be either a circuit or a basis in a \quadflower\ $M$.
  It is convenient to think of each $X_J$ as a vertex in a graph: consider the $3$-dimensional hypercube graph $Q$ on vertex set $\{X_J : J \subseteq \{1,2,3\}\}$, where distinct vertices $X_I$ and $X_J$ are adjacent if and only if the cardinality of the set difference $I \triangle J$ is one.
  Let $\mathbf{X}$ be a subset of the vertices of $Q$, where $X_J \in \mathbf{X}$ if $X_J$ is a circuit of $M$, whereas $X_J \notin \mathbf{X}$ if $X_J$ is a basis of $M$.
  %Then we can view the \quadflower\ $M$ as a set $\mathbf{X} \subseteq V(Q)$.
  Then, for the \quadflower\ $M$, we obtain an auxiliary set $\mathbf{X} \subseteq V(Q)$.
  %It easily follows that there are $22$ pairwise non-isomorphic \quadflower s.

  %Let $\mathbf{X} \subseteq V(Q)$ be the vertices corresponding to circuits of the \quadflower\ $M$.
  Suppose that $X_I,X_J \in \mathbf{X}$ for some $X_I$ and $X_J$ that are adjacent in $Q$ (so $|I \triangle J| = 1$).
  %Fix some $X_2 \in \{\{p_2,t_2\},\{q_2,s_2\}\}$ and $X_3 \in \{\{p_3,t_3\},\{q_3,s_3\}\}$.
  %If both $\{\{p_i,t_i\} \cup X_2 \cup X_3$ and $\{q_i,s_i\} \cup X_2 \cup X_3$ are circuits,
  Let $Z \subseteq X_I \cap X_J$ such that $|Z|=3$.
  Then it follows that $M/Z$ has an $F_7$-restriction (where $F_7$ is the Fano matroid), so $M$ is not representable over any field that does not have characteristic two.
  %(In particular, a binary \quadflower\ is not representable over any field that does not have characteristic two.)
  On the other hand, if precisely one of $X_I$ and $X_J$ is in $\mathbf{X}$, then it follows that $M/Z$ has an $F^-_7$-restriction (where $F^-_7$ is the non-Fano matroid), so $M$ is not representable over fields with characteristic two.

  We first assume that $M$ is not a binary \quadflower, a dyadic \quadflower, nor a free \quadflower, and show that (iv) holds.
  In particular, since $M$ is not a binary \quadflower\ nor a free \quadflower, $\mathbf{X} \notin \{\emptyset, V(Q)\}$.
  We claim that if $\mathbf{X}$ is not a stable set in $Q$, then $M$ has a proper minor that is not representable over any field.
  Since $\mathbf{X} \notin \{\emptyset, V(Q)\}$, there exist adjacent vertices of $Q$, precisely one of which is in $\mathbf{X}$. %, so $M$ has a proper minor that is not representable over fields with characteristic two.
  Suppose $\mathbf{X}$ is not a stable set in $Q$, so there exist adjacent vertices of $Q$ in $\mathbf{X}$.
  Then there exist distinct $I,J,K \subseteq \{1,2,3\}$ such that $X_I,X_J \in \mathbf{X}$ and $X_K \notin \mathbf{X}$, where $X_J$ is adjacent to both $X_I$ and $X_K$ in $Q$.
  It now follows, by the previous paragraph, that $M$ has a proper minor that is not representable over any field, as claimed.
  %That is, when $\mathbf{X} \notin \{\emptyset, V(Q)\}$, the matroid $M$ has a proper minor that is not representable over any field unless $\mathbf{X}$ is a stable set.

  Now we assume that $\mathbf{X}$ is a stable set in $Q$.
  We note that $M$ is one of four \quadflower s, up to isomorphism (recalling that $M$ is not a dyadic or free \quadflower).
  For any $j \in \{1,2,3\}$, the set $\mathbf{Y}_j = \{X_J : j \in J\}$ or the set $\mathbf{Y}'_{j}=\{X_J : j \notin J\}$ induces a $4$-cycle in $Q$, so at most two of the four sets in $\mathbf{Y}_j$ or $\mathbf{Y}'_j$ are circuits.
  Since $M$ is not a binary \quadflower, a dyadic \quadflower, nor a free \quadflower, there exists some $j \in \{1,2,3\}$ such that either $|\mathbf{Y}_j \cap \mathbf{X}| = 1$ or $|\mathbf{Y}'_j \cap \mathbf{X}| = 1$.
  Up to isomorphism, we may assume that
  $\{p_1,t_1,p_2,t_2,p_3,t_3\}$ is a circuit in $M$, but the sets
  $\{p_1,t_1,p_2,t_2,q_3,s_3\}$,
  $\{p_1,t_1,q_2,s_2,p_3,t_3\}$, and
  $\{p_1,t_1,q_2,s_2,q_3,s_3\}$ are bases.
  Consider $M' = M / p_1$.
  We claim that $M'$ is not representable over any field.
  If $M'$ has an $\mathbb{F}$-representation $A'$ for some field $\mathbb{F}$, then it is easy to verify (see \cite[Theorem~6.4.7]{oxley}, for example) that $A'$ is projectively equivalent to a matrix of the form
  $$\kbordermatrix{
    & s_1 & t_1 & s_2 & t_2 & s_3 & t_3 \\
    q_1 & 0 & -1 & 1 &  1 & 0 & 0 \\
    p_2 & 0 &  0 & * &  0 & 1 & 1 \\
    q_2 & 0 &  0 & 0 &  * & 1 & * \\
    p_3 & 1 &  1 & 0 &  0 & 1 & 0 \\
    q_3 & 1 &  * & 0 &  0 & 0 & *},$$
    where $*$ signifies the entry is non-zero.

    By considering the dependent sets $\{s_1,s_3,t_3,q_1,q_3\}$, $\{s_1,t_1,s_3,q_1,p_2\}$, $\{t_1,s_3,t_3,q_1,p_2\}$, and $\{s_2,t_2,s_3,p_3,q_3\}$, we see that the matrix is in fact of the form
  $$\kbordermatrix{
    & s_1 & t_1 & s_2 & t_2 & s_3 & t_3 \\
    q_1 & 0 & -1 & 1 &  1 & 0 & 0 \\
    p_2 & 0 &  0 & \alpha &  0 & 1 & 1 \\
    q_2 & 0 &  0 & 0 & -\alpha & 1 & 1 \\
    p_3 & 1 &  1 & 0 &  0 & 1 & 0 \\
    q_3 & 1 &  1 & 0 &  0 & 0 & -1},$$
    for some $\alpha \in \mathbb{F} - \{0\}$.
    However, $\alpha=1$ since $\{t_1,p_2,t_2,p_3,t_3\}$ is a circuit in $M'$, but $\alpha \neq 1$ since $\{t_1,q_2,s_2,q_3,s_3\}$ is a basis in $M'$.
  So $M'$ is not representable over any field.

  We have shown that if $M$ is not a binary \quadflower, a dyadic \quadflower, nor a free \quadflower, then (iv) holds.
  Suppose now that $M$ is a \quadflower\ that is representable over some partial field $\mathbb{P}$.
  %We can represent the sets by the hypercube graph $Q_3$, whose vertices correspond to subsets of $\{1,2,3\}$, and two vertices are adjacent if their set difference is one.
  %remaining quadflowers that are representable over some field correspond to 
  %For such a \quadflower\ $M$, %with elements labelled as in the definition,
  Let $A$ be a $\mathbb{P}$-representation for $M$. %, for some partial field $\mathbb{P}$.
  Then it is easy to verify that $A$ is projectively equivalent to the matrix
  $$\kbordermatrix{
      & s_1 & t_1 & s_2 & t_2 &  s_3 &    t_3 \\
p_1 &     1 &     0 &     1 &     1 &      0 &        0 \\
q_1 &     0 &    -1 &     1 &     1 &      0 &        0 \\
p_2 &     0 &     0 &     \alpha &     0 &      1 &        1 \\
q_2 &     0 &     0 &     0 &    -\alpha &      1 &        1 \\
p_3 &     1 &     1 &     0 &     0 & 1 &        0 \\
q_3 &     1 &     1 &     0 &     0 &      0 & -1},$$
  where $\alpha \neq 0$.

  When $M$ is a binary \quadflower, we obtain a $\GF(2^n)$-representation (for any positive integer $n$) by setting $\alpha=1$, so $M$ is representable over any field with characteristic two.
  Since a binary \quadflower\ has an $F_7$-minor, it is not representable over any other field, so (i) holds.

  When $M$ is a dyadic \quadflower, we obtain a $\mathbb{D}$-representation by setting $\alpha= 1$ (where $\det(A[\{p_1,q_2,q_3\},\{s_1,t_2,t_3\}]) = 2$).  In particular, this is an $\mathbb{F}$-representation for any field $\mathbb{F}$ having characteristic not two.  Since a dyadic \quadflower\ has an $F_7^-$-minor, (ii) holds.

  When $M$ is a free \quadflower, then $M$ has an $F_7^=$-minor, where $F_7^=$ is the matroid obtained from $F_7^-$ by relaxing a circuit-hyperplane.  Thus $M$ is not $2$-regular. In particular, as $F_7^=$ has a $U_{2,5}$-minor, $M$ is not binary or ternary.
  However, $M$ is $\mathbb{F}$-representable whenever $|\mathbb{F}| \ge 4$, as then the above matrix is an $\mathbb{F}$-representation after choosing $\alpha \in \mathbb{F} - \{0,1,-1\}$.
  (Alternatively, the above matrix is a $\mathbb{K}_2$-representation, and there is a homomorphism from $\mathbb{K}_2$ to any field with size at least four \cite[Lemma 2.5.33]{vanZwam2009}.)
  This shows that (iii) holds.
\end{proof}

\subsection*{Nests}

Next, we consider \tcns.  We start by considering the circuits.
We omit the proof of the next lemma; it follows from orthogonality, the circuit axioms, and the fact that $r(M)=6$, similar to \cref{quadflowerstructure}.

\begin{lemma}
  \label{tcnstructure}
  Let $M$ be a \tcn\ with standard labelling.
  Then all of the following hold:
  \begin{enumerate}
    \item If $C$ is a non-spanning circuit of $M$, then $|C| \in \{4,6\}$.
    \item If $C$ is a circuit of $M$ with $|C|=4$, then $C$ is one of the $15$ circuits listed in \cref{tcndef}(i).
    \item If $\{e_i,e_j,e_k',e_\ell'\}$ is a circuit of $M$, then $\{e_i,e_j\} \cup (\{e_1',\dotsc,e_6'\}-\{e_k',e_\ell'\})$ and $(\{e_1,\dotsc,e_6\}-\{e_i,e_j\}) \cup \{e_k',e_\ell'\}$ are $6$-element circuits of $M$.
    \item If $C$ is a circuit of $M$ with $|C|=6$ but $C$ is not a circuit described in (iii), then $C\in \{\{e_1,e_2,e_3,e_4,e_5,e_6\},\{e_1',e_2',e_3',e_4',e_5',e_6'\}\}$.
  \end{enumerate}
\end{lemma}
%\begin{proof}
  %Items (i) and (ii) follow from orthogonality, the second circuit axiom, and the fact that $r(M)=6$.
  %Item (iii) follows from circuit elimination and orthogonality.
%%
  %For (iv), let $X \subseteq E(M)$ with $|X|=6$ such that $X$ does not contain one of the $15$ circuits of size four, and $X$ is not a circuit as described in (iii).
  %It follows that \ldots
%\end{proof}

By \cref{tcnstructure}, there are three non-isomorphic \tcns, depending on whether both, none, or one of the sets described in (iv) are circuits.
We say that a \tcn\ is
\begin{itemize}
  \item \emph{tight} if it has a nest labelling such that both $\{e_1,e_2,e_3,e_4,e_5,e_6\}$ and $\{e_1',e_2',e_3',e_4',e_5',e_6'\}$ are circuits;
  \item \emph{loose} if it has a nest labelling such that both $\{e_1,e_2,e_3,e_4,e_5,e_6\}$ and $\{e_1',e_2',e_3',e_4',e_5',e_6'\}$ are bases; and
  \item \emph{slack} if it has a nest labelling such that one of $\{e_1,e_2,e_3,e_4,e_5,e_6\}$ and $\{e_1',e_2',e_3',e_4',e_5',e_6'\}$ is a circuit, and the other is a basis.
\end{itemize}

\begin{lemma}
  \label{tcnlemma}
  Let $M$ be a \tcn.  Then $M$ has a $\{F_7,F_7^-\}$-minor.  Moreover,
  \begin{enumerate}
    \item when $M$ is tight, $M$ is $\mathbb{F}$-representable if and only if $\mathbb{F}$ has characteristic two;
    \item when $M$ is loose, $M$ is $\mathbb{F}$-representable if and only if $\mathbb{F}$ has characteristic not two; and
    \item if $M$ is slack, then $M$ is not representable over any field.
  \end{enumerate}
\end{lemma}
\begin{proof}
  Assume that $M$ has standard labelling, and consider the minor $M_1=M / \{e_4,e_5,e_6\} \ba \{e_1',e_2'\}$.
  By \cref{tcndef}(i) and \cref{tcnstructure}(iii), $M_1$ has circuits $\{e_3,e_3',e_4'\}$, $\{e_3,e_5',e_6'\}$, $\{e_2,e_5',e_4'\}$, $\{e_1,e_4',e_6'\}$, $\{e_2,e_3',e_6'\}$, and $\{e_1,e_3',e_5'\}$.
  Moreover, $\{e_1,e_2,e_3\}$ is either a circuit or a basis in $M_1$.
  It follows that $M$ has an $\{F_7,F_7^-\}$-minor.
  Suppose that $M$ is tight or loose.
  If $M$ is tight, then $M$ has an $F_7$-minor, so $M$ is not representable over any field that does not have characteristic two; and if $M$ is loose, then $M$ has an $F_7^-$-minor, so $M$ is not representable over any field that has characteristic two.
  Let $A$ be the matrix
  %Suppose that $M$ has an $\mathbb{F}$-representation $A$, for some field $\mathbb{F}$. Then it is easy to verify that $A$ is projectively equivalent to the matrix
  $$\kbordermatrix{
      & e_1 & e_3 & e_5 & e_2' & e_4' & e_6' \\
 e_2  &   0 &   1 &   0 &    0 &    1 &    1 \\
 e_4  &   0 &   0 &   1 &    1 &    0 &    1 \\
 e_6  &   1 &   0 &   0 &    1 &    1 &    0 \\
 e_1' &   0 &   1 &   1 &    0 &    0 &    1 \\
 e_3' &   1 &   0 &   1 &    1 &    0 &    0 \\
 e_5' &   1 &   1 &   0 &    0 &    1 &    0 }$$
 over a field $\mathbb{F}$.
 When $\mathbb{F}$ has characteristic two, $A$ is an $\mathbb{F}$-representation of a tight nest; otherwise, $A$ is an $\mathbb{F}$-representation of a loose nest.
 %When $M$ is tight, then $A$ is an $\mathbb{F}$-representation of $M$ whenever $\mathbb{F}$ has characteristic two; and when $M$ is loose, then $A$ is an $\mathbb{F}$-representation of $M$ whenever $\mathbb{F}$ has characteristic other than two.
 This proves (i) and (ii).

  It remains to prove (iii).
  Suppose $M$ is slack, and consider the minor $M_2=M / \{e_1',e_2',e_3'\} \ba \{e_4,e_6\}$.
  By \cref{tcndef}(i) and \cref{tcnstructure}(iii), $M_2$ has circuits $\{e_5,e_1,e_5'\}$, $\{e_1,e_2,e_4'\}$, $\{e_1,e_3,e_6'\}$, $\{e_2,e_3,e_5'\}$, $\{e_2,e_5,e_6'\}$, and $\{e_3,e_5,e_4'\}$.
  Either $\{e_1,e_2,e_3\}$ is a circuit in $M_1$ and $\{e_4',e_5',e_6'\}$ is a basis in $M_2$, or $\{e_1,e_2,e_3\}$ is a basis in $M_1$ and $\{e_4',e_5',e_6'\}$ is a circuit in $M_2$.
  It follows that $M$ has both an $F_7$- and $F_7^-$-minor, so $M$ is not representable over any field.
\end{proof}

\section{Excluded minors for \texorpdfstring{$3$}{3}-regular matroids}
\label{h5u3sec}

In this section, we consider excluded minors for the class of $\mathbb{H}_5$-representable matroids.  Recall that this is, equivalently, the class of $3$-regular matroids.
The following theorem is the recently obtained excluded-minor characterisation of $2$-regular matroids \cite{BOSW2023b,BP20}.  The matroids appearing in this theorem are well known (as they appear in the excluded-minor characterisation for either $\GF(4)$-representable matroids~\cite{GGK2000} or near-regular matroids \cite{HMvZ2011}), with two exceptions.
The matroid $P_8^-$ is obtained from $P_8$ by relaxing one of the two disjoint circuit-hyperplanes.
The matroid $TQ_8$ is a rank-$4$ sparse paving matroid on ground set $\{0,1,\dotsc,7\}$ with eight non-spanning circuits $\big\{\{i, i+2, i+4, i+5\} : i \in \{0, 1, \dotsc , 7\}\big\}$, working modulo 8.

\begin{theorem}[{\cite[Theorem~1.2]{BOSW2023b}}]
  \label{2regexminorconj}
  A matroid $M$ is $2$-regular if and only if $M$ has no minor isomorphic to
$U_{2,6}$, $U_{3,6}$, $U_{4,6}$, $P_6$,
$F_7$, $F_7^*$, $F_7^-$, $(F_7^-)^*$, $F_7^=$, $(F_7^=)^*$,
$\AG(2,3)\ba e$, $(\AG(2,3)\ba e)^*$, $(\AG(2,3)\ba e)^{\Delta Y}$, $P_8$, $P_8^-$, $P_8^=$, and $\TQ_8$.
\end{theorem}

The next lemma follows from \cref{2regexminorconj}, by considering which of the excluded minors for $2$-regular matroids are $3$-regular (see \cite[Table~5]{BP20}).

\begin{lemma}
  \label{u3minors}
  Let $M$ be an excluded minor for the class of $3$-regular matroids.  Then, either
  \begin{enumerate}
    \item $M$ has a $\{U_{2,6}, U_{3,6}, U_{4,6}, P_6\}$-minor, or
    \item $M$ is isomorphic to a matroid in
\begin{multline*}
  %\mathcal{X} =
  \{F_7,F_7^*,F_7^-,(F_7^-)^*,F_7^=,(F_7^=)^*,\\
    \AG(2,3)\ba e, (\AG(2,3)\ba e)^*, (\AG(2,3)\ba e)^{\Delta Y},%\\
  P_8,P_8^-,P_8^=,\TQ_8\}.
\end{multline*}
  \end{enumerate}
\end{lemma}

The following is a %straightforward
consequence of the fact that $U_{2,5}$ is a strong $\mathbb{H}_5$-stabilizer \cite[Lemma 7.3.16]{vanZwam2009}.
\begin{lemma}
  The matroids $U_{2,6}$, $U_{3,6}$, $U_{4,6}$, and $P_6$ are strong $\mathbb{H}_5$-stabilizers. 
\end{lemma}

We also require the following consequence of \cref{qflemma,tcnlemma,u3minors}.

\begin{lemma}
  \label{qfcorr}
  If $M$ is a \quadflower\ or a \tcn, then $M$ is not an excluded minor for the class of $3$-regular matroids.
\end{lemma}
\begin{proof}
  If $M$ is a \tcn, then $M$ has either $F_7$ or $F_7^-$ as a proper minor, by \cref{tcnlemma}, so $M$ is not an excluded minor for the class of $3$-regular matroids, by \cref{u3minors}.
  Suppose that $M$ is a \quadflower\ that is an excluded minor for the class of $3$-regular matroids.
  Then every proper minor of $M$ is $3$-regular.
  By \cref{qflemma}, $M$ is either a binary \quadflower, a dyadic \quadflower, or a free \quadflower.
  But then $M$ has either $F_7$, $F_7^-$, or $F_7^=$ as a proper minor, respectively, contradicting \cref{u3minors}.
\end{proof}

%\subsection*{Excluded minors}

%There are zero on at most 6 elements, twenty-six on 7 elements, seven on 8 elements, and zero on 9 elements.

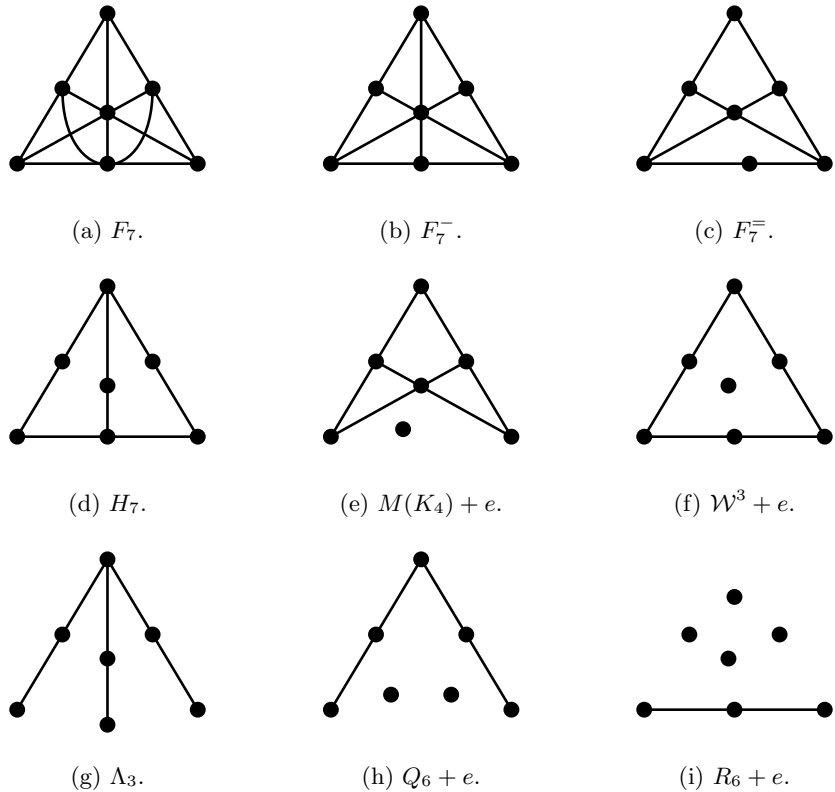
\begin{figure}[b]
  \begin{subfigure}{0.32\textwidth}
    \centering
    \begin{tikzpicture}[rotate=90,yscale=0.8,line width=1pt]
      \tikzset{VertexStyle/.append style = {minimum height=5,minimum width=5}}
      \clip (-0.5,-4.5) rectangle (2.5,-.5);
      \draw (0,-1) -- (2,-2.5) -- (0,-4);
      \draw (1,-1.75) -- (0,-4);
      \draw (1,-3.25) -- (0,-1);
      \draw (0,-1) -- (0,-4);
      \draw (2,-2.5) -- (0,-2.5);
      \draw (1,-3.25) .. controls (-.35,-3.25) and (-.35,-1.75) .. (1,-1.75);

      \SetVertexNoLabel
      \Vertex[x=2,y=-2.5]{a2}
      \Vertex[x=0.68,y=-2.5]{a3}
      \Vertex[x=1,y=-3.25]{a4}
      \Vertex[x=1,y=-1.75]{a5}
      \Vertex[x=0,y=-2.5]{d}
      \Vertex[x=0,y=-1]{e}
      \Vertex[x=0,y=-4]{f}
    \end{tikzpicture}
    \caption{$F_7$.}
  \end{subfigure}
  \begin{subfigure}{0.32\textwidth}
    \centering
    \begin{tikzpicture}[rotate=90,yscale=0.8,line width=1pt]
      \tikzset{VertexStyle/.append style = {minimum height=5,minimum width=5}}
      \clip (-0.5,-4.5) rectangle (2.5,-.5);
      \draw (0,-1) -- (2,-2.5) -- (0,-4);
      \draw (1,-1.75) -- (0,-4);
      \draw (1,-3.25) -- (0,-1);
      \draw (0,-1) -- (0,-4);
      \draw (2,-2.5) -- (0,-2.5);

      \SetVertexNoLabel
      \Vertex[x=2,y=-2.5]{a2}
      \Vertex[x=0.68,y=-2.5]{a3}
      \Vertex[x=1,y=-3.25]{a4}
      \Vertex[x=1,y=-1.75]{a5}
      \Vertex[x=0,y=-2.5]{d}
      \Vertex[x=0,y=-1]{e}
      \Vertex[x=0,y=-4]{f}
    \end{tikzpicture}
    \caption{$F_7^-$.}
  \end{subfigure}
  \begin{subfigure}{0.32\textwidth}
    \centering
    \begin{tikzpicture}[rotate=90,yscale=0.8,line width=1pt]
      \tikzset{VertexStyle/.append style = {minimum height=5,minimum width=5}}
      \clip (-0.5,-4.5) rectangle (2.5,-.5);
      \draw (0,-1) -- (2,-2.5) -- (0,-4);
      \draw (1,-1.75) -- (0,-4);
      \draw (1,-3.25) -- (0,-1);
      \draw (0,-1) -- (0,-4);

      \SetVertexNoLabel

      \Vertex[L=$a$,Lpos=0,LabelOut=true,x=2,y=-2.5]{a2}
      \Vertex[L=$g$,Lpos=90,LabelOut=true,x=0.68,y=-2.5]{a3}
      \Vertex[L=$e$,Lpos=0,LabelOut=true,x=1,y=-3.25]{a4}
      \Vertex[L=$d$,Lpos=180,LabelOut=true,x=1,y=-1.75]{a5}
      \Vertex[L=$f$,Lpos=-90,LabelOut=true,x=0,y=-2.75]{d}
      \Vertex[L=$b$,Lpos=180,LabelOut=true,x=0,y=-1]{e}
      \Vertex[L=$c$,Lpos=0,LabelOut=true,x=0,y=-4]{f}
    \end{tikzpicture}
    \caption{$F_7^=$.}
  \end{subfigure}
  \begin{subfigure}{0.32\textwidth}
    \centering
    \begin{tikzpicture}[rotate=90,yscale=0.8,line width=1pt]
      \tikzset{VertexStyle/.append style = {minimum height=5,minimum width=5}}
      \clip (-0.5,-4.5) rectangle (2.5,-.5);
      \draw (0,-1) -- (2,-2.5) -- (0,-4);
      \draw (0,-1) -- (0,-4);
      \draw (2,-2.5) -- (0,-2.5);

      \SetVertexNoLabel
      \Vertex[x=2,y=-2.5]{a2}
      \Vertex[x=0.68,y=-2.5]{a3}
      \Vertex[x=1,y=-3.25]{a4}
      \Vertex[x=1,y=-1.75]{a5}
      \Vertex[x=0,y=-2.5]{d}
      \Vertex[x=0,y=-1]{e}
      \Vertex[x=0,y=-4]{f}
    \end{tikzpicture}
    \caption{$H_7$.}
  \end{subfigure}
  \begin{subfigure}{0.32\textwidth}
    \centering
    \begin{tikzpicture}[rotate=90,yscale=0.8,line width=1pt]
      \tikzset{VertexStyle/.append style = {minimum height=5,minimum width=5}}
      \clip (-0.5,-4.5) rectangle (2.5,-.5);
      \draw (0,-1) -- (2,-2.5) -- (0,-4);
      \draw (1,-1.75) -- (0,-4);
      \draw (1,-3.25) -- (0,-1);

      \SetVertexNoLabel

      \Vertex[L=$a$,Lpos=0,LabelOut=true,x=2,y=-2.5]{a2}
      \Vertex[L=$g$,Lpos=90,LabelOut=true,x=0.68,y=-2.5]{a3}
      \Vertex[L=$e$,Lpos=0,LabelOut=true,x=1,y=-3.25]{a4}
      \Vertex[L=$d$,Lpos=180,LabelOut=true,x=1,y=-1.75]{a5}
      \Vertex[L=$f$,Lpos=-90,LabelOut=true,x=-25,y=-2.75]{d}
      \Vertex[L=$b$,Lpos=180,LabelOut=true,x=0,y=-1]{e}
      \Vertex[L=$c$,Lpos=0,LabelOut=true,x=0,y=-4]{f}
      \Vertex[x=0.1,y=-2.2]{z}
    \end{tikzpicture}
    \caption{$M(K_4)+e$.}
  \end{subfigure}
  \begin{subfigure}{0.32\textwidth}
    \centering
    \begin{tikzpicture}[rotate=90,yscale=0.8,line width=1pt]
      \tikzset{VertexStyle/.append style = {minimum height=5,minimum width=5}}
      \clip (-0.5,-4.5) rectangle (2.5,-.5);
      \draw (0,-1) -- (2,-2.5) -- (0,-4);
      \draw (0,-1) -- (0,-4);

      \SetVertexNoLabel
      \Vertex[x=2,y=-2.5]{a2}
      \Vertex[x=0.68,y=-2.4]{a3}
      \Vertex[x=1,y=-3.25]{a4}
      \Vertex[x=1,y=-1.75]{a5}
      \Vertex[x=0,y=-2.5]{d}
      \Vertex[x=0,y=-1]{e}
      \Vertex[x=0,y=-4]{f}
    \end{tikzpicture}
    \caption{$\mathcal{W}^3+e$.}
  \end{subfigure}
  \begin{subfigure}{0.32\textwidth}
    \centering
    \begin{tikzpicture}[rotate=90,yscale=0.8,line width=1pt]
      \tikzset{VertexStyle/.append style = {minimum height=5,minimum width=5}}
      \clip (-0.5,-4.5) rectangle (2.5,-.5);
      \draw (0,-1) -- (2,-2.5) -- (0,-4);
      \draw (2,-2.5) -- (-.2,-2.5);

      \SetVertexNoLabel
      \Vertex[x=2,y=-2.5]{a2}
      \Vertex[x=0.68,y=-2.5]{a3}
      \Vertex[x=1,y=-3.25]{a4}
      \Vertex[x=1,y=-1.75]{a5}
      \Vertex[x=-.2,y=-2.5]{d}
      \Vertex[x=0,y=-1]{e}
      \Vertex[x=0,y=-4]{f}
    \end{tikzpicture}
    \caption{$\Lambda_3$.}
  \end{subfigure}
  \begin{subfigure}{0.32\textwidth}
    \centering
    \begin{tikzpicture}[rotate=90,yscale=0.8,line width=1pt]
      \tikzset{VertexStyle/.append style = {minimum height=5,minimum width=5}}
      \clip (-0.5,-4.5) rectangle (2.5,-.5);
      \draw (0,-1) -- (2,-2.5) -- (0,-4);

      \SetVertexNoLabel
      \Vertex[x=2,y=-2.5]{a2}
      \Vertex[x=1,y=-3.25]{a4}
      \Vertex[x=1,y=-1.75]{a5}
      \Vertex[x=0,y=-1]{e}
      \Vertex[x=0,y=-4]{f}
      \Vertex[x=0.2,y=-2.0]{a3}
      \Vertex[x=0.2,y=-3.0]{d}
    \end{tikzpicture}
    \caption{$Q_6+e$.}
  \end{subfigure}
  \begin{subfigure}{0.32\textwidth}
    \centering
    \begin{tikzpicture}[rotate=90,yscale=0.8,line width=1pt]
      \tikzset{VertexStyle/.append style = {minimum height=5,minimum width=5}}
      \clip (-0.5,-4.5) rectangle (2.5,-.5);
      \draw (0,-1) -- (0,-4);

      \SetVertexNoLabel
      \Vertex[x=1.5,y=-2.5]{a2}
      \Vertex[x=0.68,y=-2.4]{a3}
      \Vertex[x=1,y=-3.25]{a4}
      \Vertex[x=1,y=-1.75]{a5}
      \Vertex[x=0,y=-2.5]{d}
      \Vertex[x=0,y=-1]{e}
      \Vertex[x=0,y=-4]{f}
    \end{tikzpicture}
    \caption{$R_6+e$.}
  \end{subfigure}
  \caption{Nine rank-$3$ excluded minors for the class of $3$-regular matroids.}
  \label{relaxationsfig}
\end{figure}

For a matroid $M$, let $M+e$ denote the matroid obtained from $M$ by a free single-element extension.
Consider the set of $10$ matroids that can be obtained from the Fano matroid $F_7$ by relaxing circuit-hyperplanes.
This gives the sequence $$\{F_7\}, \{F_7^-\}, \{F_7^=\}, \{H_7,M(K_4)+e\}, \{\mathcal{W}^3+e,\Lambda_3\}, \{Q_6+e\}, \{P_6+e\}, \{U_{3,7}\};$$ geometric representations of these matroids are given in \cref{relaxationsfig}.
These $10$ matroids give rise to nine $\Delta Y$-equivalence classes (the matroids $H_7$ and $M(K_4)+e$ are $\Delta Y$-equivalent). %The $7$-element excluded minors for the class of $\mathbb{H}_5$-representable matroids are precisely the matroids $\Delta Y$-equivalent to one of these 10 matroids.
%The 26 on 7 elements are in one of 10 dual-closed $\Delta Y$-equivalence classes: $\Delta(U_{2,7})$, $\Delta(U_{3,7})$, $\Delta(F_7)$, $\Delta(F_7^-)$, $\Delta(\Lambda_3)$, $\Delta(Q_6^+)$, $\Delta(R_6^+)$, $\Delta((\mathcal{W}^3)^+)$, $\Delta(F_7^=)$, and $\Delta( )$.
%
%($P_6^+$, $Q_6^+$, and $(\mathcal{W}^3)^+$ denote the free extensions of $P_6$ and $Q_6$ and $\mathcal{W}^3$ respectively.)
Each of these matroids turns out to be an excluded minor for the class of $3$-regular matroids.

%The seven excluded minors on 8 elements are 
%$\AG(2,3)\ba e$, $(\AG(2,3)\ba e)^*$, $(\AG(2,3)\ba e)^{\Delta Y}$, $P_8$, $P_8^-$, $P_8^=$, and $\TQ_8$.  Note that these same seven matroids are the $8$-element excluded minors for $2$-regular matroids.

%Which of these are $\mathbb{H}_4$-representable?
%\end{comment}

\begin{theorem}
  \label{h5exminors}
  There are exactly $33$ matroids that are excluded minors for the class of $3$-regular matroids and have at most $13$ elements.  These are:
  $$F_7, F_7^-, F_7^=, H_7,M(K_4)+e, \mathcal{W}^3+e,\Lambda_3, Q_6+e, P_6+e, U_{3,7},$$ and their duals; $\Delta^{(*)}(U_{2,7})$; and
$\AG(2,3)\ba e$, $(\AG(2,3)\ba e)^*$, $(\AG(2,3)\ba e)^{\Delta Y}$, $P_8$, $P_8^-$, $P_8^=$, and $\TQ_8$.
\end{theorem}
\begin{proof}
  %exhaustive on up to 10 element
  Let $\mathcal{U}^{(n)}$ be the set of all $n$-element $3$-connected $3$-regular matroids with a $\{U_{2,6},U_{3,6},U_{4,6},P_6\}$-minor.
  Using a computer, we exhaustively generated the matroids in $\mathcal{U}^{(n)}$, for each $n=6,7,8,\dotsc,13$ (see \cref{h5table}).
  By \cref{u3minors}, any excluded minor for the class of $3$-regular matroids has at least $7$ elements.
  Let $n \in \{7,8,9,10\}$, and suppose all excluded minors for the class on fewer than $n$ elements are known.
  We generated all $3$-connected single-element extensions of some matroid in $\mathcal{U}^{(n-1)}$; let $\mathcal{Z}^{(n)}$ be this collection of matroids.
  %It follows from \cref{u3minors,nearregconn} that if $M$ is an $n$-element excluded minor not listed in \cref{u3minors}(ii), then $\mathcal{Z}^{(n)}$ contains at least one of $M$ and $M^*$.
  From the collection $\mathcal{Z}^{(n)}$, we filter out any matroids in $\mathcal{U}^{(n)}$, or any matroid containing, as a minor, one of the excluded minors for $3$-regular matroids on fewer than $n$ elements; call the resulting collection $\mathcal{Y}^{(n)}$.
  Clearly, any matroid in $\mathcal{Y}^{(n)}$ is an excluded minor for the class.
  In fact, if $M$ is an $n$-element excluded minor not listed in \cref{u3minors}(ii), then, by \cref{u3minors,nearregconn}, the collection $\mathcal{Y}^{(n)}$ contains at least one of $M$ and $M^*$, so the complete list of $n$-element excluded minors is obtained by closing $\mathcal{Y}^{(n)}$ under duality.
  Now, all excluded minors for the class on at most $10$ elements are known.

  %splicing on 11, with adjusted N minors
  Next consider when $n=11$.
  In order to utilise the contrapositive of \cref{spikelikes}, we consider $3$-connected matroids with a $\{U_{2,5},U_{3,5}\}$-minor, so that, when $|E(M)| = 11$ and $N \in \{U_{2,5},U_{3,5}\}$, we have $|E(M)|-|E(N)| = 6$.
  Let $\mathcal{V}^{(n)}$ be the set of all $n$-element $3$-connected $3$-regular matroids with a $\{U_{2,5},U_{3,5}\}$-minor.
  We exhaustively generated the matroids in $\mathcal{V}^{(n)}$, for each $n \le 11$ (see \cref{h5table2}).
  For any (not necessarily non-isomorphic) pair of matroids $M_e$ and $M_f$ in $\mathcal{V}^{(10)}$, where $M_e$ and $M_f$ are both single-element extensions of some matroid $M' \in \mathcal{V}^{(9)}$, we find the splices of $M_e$ and $M_f$.
  Let $\mathcal{Z}$ be the collection of all $11$-element matroids that can be obtained in this way.
  From the collection $\mathcal{Z}$, we filter out any matroids in $\mathcal{V}^{(11)}$, or any matroid containing, as a minor, one of the excluded minors for $3$-regular matroids on at most $10$ elements; call the resulting collection $\mathcal{Y}^{(11)}$.
  Clearly, any matroid in $\mathcal{Y}^{(11)}$ is an excluded minor for the class.
  In fact, if $M$ is an $11$-element excluded minor, then, by \cref{u3minors,detachthm}, either the collection $\mathcal{Y}^{(11)}$ contains a matroid that is $\Delta Y$-equivalent to $M$ or $M^*$, or $M$ has a \spikey.
  But the latter contradicts \cref{spikelikes}.
  Therefore, the complete list of $11$-element excluded minors is obtained by closing $\mathcal{Y}^{(11)}$ under duality and $\Delta Y$-equivalence.
  Now, all excluded minors for the class on at most $11$ elements are known.

  %\ldots use splicing: need $\mathcal{N}=\{U_{2,5},U_{3,5}\}$ so that $|E(M)| \ge \max\{|E(N)| : N \in \mathcal{N}\} + 6$ (otherwise, we can't guarantee we get all splices: can have matroids where $M \ba e \ba f$ has a $\{U_{2,5},U_{3,5}\}$-minor but not a $\{U_{2,6},U_{3,6},U_{4,6},P_6\}$-minor (while $M \ba e$, $M \ba f$ and $M$ all have $\{U_{2,6},U_{3,6},U_{4,6},P_6\}$-minors).

  For $n \in \{12,13\}$, the approach is similar to when $n=11$, only we return to considering $3$-connected $3$-regular matroids with a $\{U_{2,6},U_{3,6},U_{4,6},P_6\}$-minor, for efficiency of the computations.
  Recall that $\mathcal{U}^{(n)}$ is the set of such matroids on $n$ elements, and we have exhaustively generated the matroids in $\mathcal{U}^{(n)}$ for $n \le 13$.
  Let $n \in \{12,13\}$, and assume we know all excluded minors for the class on $n-1$ elements.
  For any (not necessarily non-isomorphic) pair of matroids $M_e$ and $M_f$ in $\mathcal{U}^{(n-1)}$, where $M_e$ and $M_f$ are both single-element extensions of some matroid $M' \in \mathcal{U}^{(n-2)}$, we find the splices of $M_e$ and $M_f$.
  Let $\mathcal{Z}^{(n)}$ be the collection of all $n$-element matroids that can be obtained in this way.
  We note that $\mathcal{Z}^{(13)}$ contained 7550463 matroids with rank 6, and 6788774 matroids with rank 7.
  From the collection $\mathcal{Z}^{(n)}$, we filter out any matroids in $\mathcal{U}^{(n)}$, or any matroid containing, as a minor, one of the excluded minors for $3$-regular matroids on at most $n-1$ elements; call the resulting collection $\mathcal{Y}^{(n)}$.
  Clearly, any matroid in $\mathcal{Y}^{(n)}$ is an excluded minor for the class.
  In fact, if $M$ is an $n$-element excluded minor, then, by \cref{u3minors,detachthm,spikelikes}, either the collection $\mathcal{Y}^{(n)}$ contains a matroid that is $\Delta Y$-equivalent to either $M$ or $M^*$, or $n=12$ and $M$ is a \quadflower\ or a \tcn. 
  But the latter contradicts \cref{qfcorr}.
  %\remark{(\ldots TODO: deal with the case where $M$ is a \tcn\ or \quadflower)
%Need to check we can't have special 10 or 12 element matroids as minors).}
  Therefore, the complete list of $n$-element excluded minors is obtained by closing $\mathcal{Y}^{(n)}$ under duality and $\Delta Y$-equivalence.
  This completes the proof.
  %for 12-13 elements, splicing with $\mathcal{N}=\{U_{2,6},U_{3,6},U_{4,6},P_6\}$.
  %\ldots check how many non-isomorphic matroids of rank 6?
%
  %for 13 elements, on rank 6 there were 7550463 matroids to consider (after splicing)
  %on rank 7 there were 6788774.
%
  %for 14 elements, on rank 6 there were 52685054 matroids to consider (after splicing)
  %\ldots
\end{proof}

\begin{table}[htbp]
  \begin{tabular}{r|r r r r r r r r r r}
    \hline
$r \ba n$ & 6 & 7 &  8 &   9 &   10 &    11 &    12 &     13 \\
    \hline
        2 & 1 &   &    &     &      &       &       &        \\
        3 & 2 & 2 &  3 &   4 &    3 &     2 &     1 &        \\
        4 & 1 & 2 & 15 &  63 &  160 &   252 &   294 &    238 \\
        5 &   &   &  3 &  63 &  572 &  2969 &  9701 &  21766 \\
        6 &   &   &    &   4 &  160 &  2969 & 29288 & 172233 \\
        7 &   &   &    &     &    3 &   252 &  9701 & 172233 \\
        8 &   &   &    &     &      &     2 &   294 &  21766 \\
        9 &   &   &    &     &      &       &     1 &    238 \\
       %10 &   &   &    &     &      &       &       &        \\
       \hline                                           
    Total & 4 & 4 & 21 & 134 & 264 & 6446 & 49280 & 388474 \\
       \hline
  \end{tabular}
  \caption{The number of $3$-connected $3$-regular $n$-element rank-$r$ matroids with a $\{U_{2,6},U_{3,6},U_{4,6},P_6\}$-minor, for $n \le 13$.}
  \label{h5table}
\end{table}

\begin{table}[hbtp]
  \begin{tabular}{r|r r r r r r r r r r}
    \hline
$r \ba n$ & 5 & 6 & 7 &  8 &   9 &   10 &    11 \\ %&    12 \\
    \hline
        2 & 1 & 1 &   &    &     &      &       \\ %&       \\
        3 & 1 & 3 & 4 &  7 &   7 &    5 &     2 \\ %&     1 \\
        4 &   & 1 & 4 & 32 & 125 &  273 &   384 \\ %&   383 \\
        5 &   &   &   &  7 & 125 & 1074 &  5125 \\ %& 15058 \\
        6 &   &   &   &    &   7 &  273 &  5125 \\ %& 47881 \\
        7 &   &   &   &    &     &    5 &   384 \\ %& 15058 \\
        8 &   &   &   &    &     &      &     2 \\ %&   383 \\
        9 &   &   &   &    &     &      &       \\ %&     1 \\
       10 &   &   &   &    &     &      &       \\ %&       \\
       \hline                                           
    Total & 2 & 5 & 8 & 46 & 264 & 1630 & 11022 \\ %& 78765 \\
       \hline
  \end{tabular}
  \caption{The number of $3$-connected $3$-regular $n$-element rank-$r$ matroids with a $\{U_{2,5},U_{3,5}\}$-minor, for $n \le 11$.}
  \label{h5table2}
\end{table}

%Conjecture: this is the complete list.

Comparing with the excluded minors for $2$-regular matroids, note that instead of $U_{2,6}$, $U_{3,6}$, $U_{4,6}$ and $P_6$, the matroids that appear as excluded minors for $3$-regular matroids are either in $\Delta^{(*)}(U_{2,7})$, or can be obtained from $F_7$ by relaxing zero or more circuit-hyperplanes and then possibly dualising.

In \cref{h5upfs}, we show the excluded minors for $3$-regular matroids, grouped by $\Delta Y$-equivalence classes, and, for those that are $\GF(5)$-representable, we provide the largest integer $i$ for which they are $\mathbb{H}_i$-representable.
We also specify their universal partial field, if this partial field is well known. % (see \cite[Appendix~B]{vanZwam2009}).

\begin{table}[htbp]
  \begin{tabular}{c c c c}
    \hline
    $M$ & $\mathbb{P}_M$ & $\max\{i : M \in \mathcal{M}(\mathbb{H}_i)\}$ & $\left|\Delta(M)\right|$ \\ %& $M \in \mathcal{M}_4$?\\
    %\hline
    %$U_{2,6}$ & $\mathbb{U}_3$ & 3 \\
    %$U_{3,6}$ & $\mathbb{P}_{U_{3,6}}$ & 1 & *\\
    \hline
    $U_{2,7}$   & $\mathbb{U}_4$        & -- & 6 \\
    $F_7$       & $\GF(2)$              & -- & 2 \\
    $F_7^-$     & $\mathbb{D}$          & 2  & 2 \\
    $F_7^=$     & $\mathbb{K}_2$        & 2  & 2 \\ %& *\\
    $M(K_4)+e$ & & 4 & 4 \\
    $\mathcal{W}^3+e$ & & 2 & 2 \\
    $\Lambda_3$ & & 4 & 2 \\
    $Q_6+e$ & & 2 & 2 \\
    $P_6+e$ & & -- & 2 \\
    $U_{3,7}$ & & -- & 2 \\
    \hline
    $\AG(2,3)\ba e$ & $\mathbb{S}$          & -- & 3 \\
    $P_8$       & $\mathbb{D}$          & 2  & 1 \\
    $P_8^-$     & $\mathbb{K}_2$        & 2  & 1 \\ %& *\\
    $P_8^=$     & $\mathbb{H}_4$  & 4  & 1 \\
    $\TQ_8$     & $\mathbb{K}_2$        & 2  & 1 \\ %& *\\
    \hline
  \end{tabular}
  \caption{The excluded minors for %$\mathbb{H}_5$-representable matroids 
	  the class of $3$-regular matroids on at most $13$ elements, and their universal partial fields.  We list one representative~$M$ of each $\Delta Y$-equivalence class $\Delta(M)$.}
  \label{h5upfs}
\end{table}

Recall that the following was proved by Brettell, Oxley, Semple and Whittle~\cite{BOSW2023b}.

\begin{theorem}[{\cite[Theorem 1.1]{BOSW2023b}}]
  \label{boswthm2}
  An excluded minor for the class of $3$-regular matroids has at most $15$ elements.
\end{theorem}

It is natural to ask the viability of using the techniques here (and in \cite{BP20}) to find the excluded minors for $3$-regular matroids on up to $15$ elements, in order to obtain a full excluded-minor characterisation for the class. % of $3$-regular matroids.

For obtaining the excluded minors for the class of $2$-regular matroids, an important optimisation used in \cite{BP20} is that any such excluded minor with at least $9$ elements is $\GF(4)$-representable, for otherwise it would also be an excluded minor for the class of $\GF(4)$-representable matroids (the largest of which is known to have $8$ elements~\cite{GGK2000}).
%Thus, one can work entirely in the world of $\GF(4)$-representable matroids.
Thus, there one can first find certain $\GF(211)$-representable matroids (as a ``proxy'' for $2$-regular matroids), and then the potential excluded minors are $\GF(4)$-representable extensions (or splices) of these matroids.
The same cannot be said for $3$-regular matroids, where we consider certain $\GF(3527)$-representable matroids (as a ``proxy'' for $3$-regular matroids), and then must consider \emph{all} extensions (or splices) of these matroids.

The upshot is, using more time or computing resources, one could probably extend \cref{h5exminors} to matroids on 14 elements, but then another big step in computational power would still be required to tackle matroids on $15$ elements.
A better approach seems some sort of compromise between an exhaustive computation and developing the theory in order to decrease the bound, as in \cite{BOSW2023b}: that is, a computation that exploits structural results.

%TODO: comment on how we can't match 15 as for 2-regular (restricted to quaternary is a big optimisation there).

\section{Excluded minors for \texorpdfstring{$\mathbb{H}_i$}{Hi}-representable matroids when \texorpdfstring{$i \in \{1,2,3,4\}$}{i is 1, 2, 3, or 4}}
\label{hmidsec}

In this section, we use the following notation.
For $i \in \{1,2,3,4,5\}$, let $\mathcal{X}_i$ be the set of excluded minors for the class of $\mathbb{H}_i$-representable matroids.
For $i \in \{1,2,3,4\}$, let $\mathcal{N}_i$ be the set of matroids in $\mathcal{X}_{i+1}$ that are $\mathbb{H}_i$-representable.
Using this notation, we have:

\begin{lemma}
  Let $M \in \mathcal{X}_i$ for some $i \in \{1,2,3,4\}$.
  Then either
  \begin{enumerate}
    \item $M$ has an $\mathcal{N}_i$-minor, or
    \item $M \in \mathcal{X}_{i+1}$.
  \end{enumerate}
\end{lemma}
\noindent
Moreover, the matroids in $\mathcal{N}_i$ are strong $\mathbb{H}_i$-stabilizers by \cref{hydrastabs}.

We consider now the computation of matroids in $\mathcal{X}_i$ on at most $10$ elements, for $i\in \{1,2,3,4\}$.
We first consider when $i=4$, in which case,
%\ldots
%We first consider the excluded minors for the class of $\mathbb{H}_4$-representable matroids.
%
using \cref{h5exminors,h5upfs}, we can start from the following:
\begin{lemma}
  \label{h4minors}
  Suppose $M \in \mathcal{X}_4$ and $|E(M)| \le 13$.
  %Let $M$ be an excluded minor for the class of $\mathbb{H}_4$-representable matroids having at most $13$ elements.
  Then, either
  \begin{enumerate}
    \item $M$ or $M^*$ has a minor in $\{P_8^=,\Lambda_3,H_7,M(K_4)+e\}$, or
    \item $M$ is isomorphic to one of $26$ matroids in $\mathcal{X}_5$; namely, $M$ or $M^*$ is isomorphic to a matroid in
\begin{multline*}
  \{F_7, F_7^-, F_7^=, U_{3,7}, \mathcal{W}^3+e, Q_6+e, P_6+e, \\
    \AG(2,3)\ba e, (\AG(2,3)\ba e)^{\Delta Y}, P_8, P_8^-, \TQ_8\} \cup \Delta(U_{2,7}).
\end{multline*}
  \end{enumerate}
\end{lemma}

%\ldots

\begin{theorem}
  \label{thmh4}
  There are $63$ excluded minors for $\mathbb{H}_4$-representability on at most $10$ elements.
%Moreover, those that are Hydra-3-representable belong to one of 9 $\Delta Y$-equivalence classes
\end{theorem}
\begin{proof}
  Let $M \in \mathcal{X}_4$ with $|E(M)| \le 10$.
  By \cref{h4minors}, if $M$ is not one of the 26 matroids in \cref{h4minors}(ii), then $M$ has an $\mathcal{N}_4$-minor, where $\mathcal{N}_4$ is obtained by closing $\{P_8^=,\Lambda_3,H_7,M(K_4)+e\}$ under duality.

  Let $\mathcal{H}_4^{(n)}$ denote the $n$-element $3$-connected $\mathbb{H}_4$-representable matroids with an $\mathcal{N}_4$-minor.
  Starting from each matroid $N \in \mathcal{N}_4$, and using the fact that $N$ is a strong $\mathbb{H}_4$-stabilizer (by \cref{hydrastabs}), we exhaustively generated the matroids in $\mathcal{H}_4^{(n)}$, for $n=7,8,9,10$, see \cref{h4table}.
  Suppose all matroids in $\mathcal{X}_4$ on fewer than $n$ elements are known, for some $n \le 10$.
  We generated all $3$-connected single-element extensions of some matroid in $\mathcal{H}^{(n-1)}_4$, then removed any matroids in $\mathcal{H}^{(n)}_4$, or any containing, as a minor, an excluded minor for the class of $\mathbb{H}_4$-representable matroids on fewer than $n$ elements.
  It follows from \cref{nearregconn} that, after closing the resulting set of matroids under duality, we obtain all the matroids in $\mathcal{X}_4$ having at most $n$ elements.
\end{proof}

\begin{table}[hbtp]
  \begin{tabular}{r|r r r r}
    \hline
$r \ba n$ & 7 &  8 &   9 &   10 \\ %11
    \hline
        3 & 3 &  3 &   4 &    3 \\ % 1
        4 & 3 & 24 & 127 &  408 \\ % 712
        5 &   &  3 & 127 & 1747 \\ % 11270
        6 &   &    &   4 &  408 \\ % 11270
        7 &   &    &     &    3 \\ % 712
       \hline                      % 1                     
       Total & 6 & 30 & 262 & 2569 \\ % 23966
       \hline
  \end{tabular}
  \caption{The number of $3$-connected $\mathbb{H}_4$-representable $n$-element rank-$r$ matroids with a $\mathcal{N}_4$-minor, for $n \le 10$.}
  \label{h4table}
\end{table}

For the next layer of the hierarchy, the question becomes: which of the 63 excluded minors are $\mathbb{H}_3$-representable?
It turns out that 20 of them are, as given in the next lemma.  We provide $\mathbb{H}_3$-representations for these matroids in the appendix.

Note that some of these matroids are related to the well-known Pappus matroid~\cite[Page~655]{oxley}, which we denote as $\pappus$.
After deleting an element from $\pappus$, the resulting matroid is unique up to isomorphism; we denote this matroid as $\pappus \ba e$.
The matroid $\pappus \ba e$ has a pair of free bases (see \cite[Lemma~2.1]{CW18}); tightening either one of these results in a matroid that we denote $(\pappus \ba e)^+$ (the two possible matroids are isomorphic).

%On the other hand, the other 43 matroids in $\mathcal{X}_4$ are also in $\mathcal{X}_3$.

%give excluded minor characterisation (up to 10 or so elements)

%\section{Excluded minors for \texorpdfstring{$\mathbb{H}_3$}{H3}-representable matroids}

\begin{lemma}
  \label{h3minors}
  If $N \in \mathcal{N}_3$ and $E(N) \le 10$, then $N$ is one of the $20$ matroids obtained by closing the set
      $$\Delta(\{\pappus \ba e, (\pappus \ba e)^+, M_{3240}\}) \cup \{M_{1543}, M_{1559}, M_{1566}, M_{1596}, M_{1598}, Fq\}$$
      under duality.
\end{lemma}

Note that, by \cref{thmh4,h3minors}, there are 43 excluded minors for $\mathbb{H}_4$-representability that are also excluded minors for $\mathbb{H}_3$-representability.

%A matroid is representable over all fields of size at least 5 iff $\mathbb{H}_3$-representable?  No, take $\mathbb{K}_2$-representable matroids.  But maybe we can guess such a partial field by looking at $\mathbb{H}_3$ and $\mathbb{K}_2$ excluded minors.

\begin{theorem}
  \label{thmh3}
  There are $133$ excluded minors for $\mathbb{H}_3$-representability on at most $10$ elements.
%Moreover, those that are Hydra-2-representable belong to one of 35 equivalence classes
\end{theorem}
\begin{proof}
  Essentially the same approach was taken as for the proof of \cref{thmh4}, only using \cref{h3minors}.
  The number of $n$-element $3$-connected $\mathbb{H}_3$-representable matroids with an $\mathcal{N}_3$-minor are given in \cref{h3table}.
\end{proof}

\begin{table}[hbtp]
  \begin{tabular}{r|r r r}
    \hline
$r \ba n$ & 8 &  9 &   10 \\
    \hline
        3 & 2 &  3 &    1 \\
        4 & 9 & 65 &  231 \\
        5 & 2 & 65 & 1144 \\
        6 &   &  3 &  231 \\
        7 &   &    &    1 \\
       \hline                                           
    Total & 13 & 136 & 1608 \\
       \hline
  \end{tabular}
  \caption{The number of $3$-connected $\mathbb{H}_3$-representable $n$-element rank-$r$ matroids with a $\mathcal{N}_3$-minor, for $n \le 10$.}
  \label{h3table}
\end{table}

%\section{Excluded minors for \texorpdfstring{$\mathbb{H}_2$}{H2}-representable matroids}

\begin{lemma}
  \label{h2minors}
  If $N \in \mathcal{N}_2$ and $E(N) \le 10$, then $N$ is one of the $99$ matroids obtained by closing the set
\begin{multline*}
  \{F_7^-, F_7^=, \mathcal{W}^3+e, Q_6+e, P_8, P_8^-, \TQ_8, \\
    M_{1537}, M_{1554}, M_{1556}, M_{1569}, M_{1676}, \Fr_1, \Fr_2, \Fr_3, \\
  M_{3241}, M_{10782}, M_{16834}, M_{17078}, M_{17759}, M_{47015}, \Fs_1, \Fs_2, \Fs_3, \Fs_4, \\
\Ft_1, \Ft_2, \Ft_4, \Ft_5, \Ft_6, \Ft_7, \Ft_8, \Fu_3, \Fu_5, \Fu_8\}
\end{multline*}
under $\Delta Y$-equivalence and duality.
\end{lemma}

Note that, by \cref{thmh3,h2minors}, there are $34$ excluded minors for $\mathbb{H}_3$-representability that are also excluded minors for $\mathbb{H}_2$-representability.

\begin{theorem}
  \label{thmh2}
  There are $621$ excluded minors for $\mathbb{H}_2$-representability on at most $10$ elements.
%Moreover, those that are uniquely GF(5)-representable are in one of 147 equivalence classes
\end{theorem}
\begin{proof}
  Again, we use a similar approach as in the proof of \cref{thmh3,thmh4}, only using \cref{h2minors}.
  The number of $n$-element $3$-connected $\mathbb{H}_2$-representable matroids with an $\mathcal{N}_2$-minor are given in \cref{h2table}.
\end{proof}

\begin{table}[hbtp]
  \begin{tabular}{r|r r r r}
    \hline
$r \ba n$ & 7 &   8 &    9 &    10 \\
    \hline
        3 & 4 &  13 &   32 &    37 \\
        4 & 4 &  89 &  913 &  5125 \\
        5 &   &  13 &  913 & 21839 \\
        6 &   &     &   32 &  5125 \\
        7 &   &     &      &    37 \\
       \hline                                           
    Total & 8 & 115 & 1890 & 32163 \\
       \hline
  \end{tabular}
  \caption{The number of $3$-connected $\mathbb{H}_2$-representable $n$-element rank-$r$ matroids with a $\mathcal{N}_2$-minor, for $n \le 10$.}
  \label{h2table}
\end{table}

\begin{lemma}
  \label{h1minors}
  If $N \in \mathcal{N}_1$ and $E(N) \le 10$, then $N$ is one of the $441$ matroids obtained by closing the set
  \begin{multline*}
    \{M_{1601}, \lambda_4^+, %($M_{1662}$)
      \Ft_3, \Fu_1, \Fu_6, \Fu_7, \UG_{10},
      M_{822}, M_{835}, M_{836}, M_{854}, M_{1495}, M_{1505}, \\
      M_{1507}, \KP_8, %$M_{1511}$:
      M_{1519}, M_{1530}, M_{1535}, M_{1538}, M_{1541}, M_{1542}, M_{1553}, M_{1591}, M_{1593}, M_{1599}, \\
      M_{1600}, M_{1609}, M_{1610}, M_{1618}, M_{1627}, M_{1674}, M_{1680}, M_{1695}, M_{1698}, M_{1713}, M_{1729}, \\
%#on 9 elements:
      M_{3193}, \BBR_9, A_9, \BBR_9^{\Delta\nabla}, M_{9574}, M_{9841}, M_{9844}, \PP_9, M_{10025}, M_{10052}, M_{10241}, M_{10242}, \\
      M_{10244}, M_{10251}, M_{10287}, M_{10368}, M_{10459}, M_{11161}, M_{11162}, M_{11347}, M_{12000}, M_{12575}, \\
      M_{15729}, \TQ_9, M_{16321}, M_{16332}, M_{18663}, M_{46942}, M_{46945}, M_{47635}, M_{48819}, M_{50041}, \\
%#on 10 elements
      \Fv_{6}, \Fv_{7}, \Fv_{8}, \Fv_{11}, \Fv_{12}, \Fv_{16}, \Fv_{18}, \Fv_{19}, \Fv_{21}, \Fv_{24}, \Fv_{25}, \Fv_{27}, \Fv_{37}, \Fv_{41}, \\
    \Fv_{42}, \Fv_{47}, \Fv_{51}, \Fv_{54}, \Fv_{56}, \Fv_{60}, \Fv_{61}, \Fv_{66}, \Fv_{68}, \Fv_{71}, \Fv_{75}, \Fv_{77}, \Fv_{85}, \Fv_{87}, \\
  \Fv_{90}, \Fv_{91}, \Fv_{103}, \GP_{10}, \TQ_{10}, \FF_{10}, \Fv_{14}, \Fv_{22}, \Fv_{26}, \Fv_{28}, \Fv_{30}, \Fv_{32}, \Fv_{35}, \\
\Fv_{36}, \Fv_{39}, \Fv_{40}, \Fv_{43}, \Fv_{44}, \Fv_{46}, \Fv_{48}, \Fv_{49}, \Fv_{50}, \Fv_{53}, \Fv_{55}, \Fv_{57}, \Fv_{58}, \Fv_{59}, \\
\Fv_{62}, \Fv_{63}, \Fv_{65}, \Fv_{70}, \Fv_{72}, \Fv_{73}, \Fv_{78}, \Fv_{79}, \Fv_{81}, \Fv_{82}, \Fv_{83}, \Fv_{84}, \Fv_{89}, \Fv_{92}, \\
\Fv_{93}, \Fv_{94}, \Fv_{95}, \Fv_{96}, \Fv_{101}, \Fv_{102}, \Fv_{104}, \Fv_{106}, \Fv_{111}, \Fv_{112}\}
  \end{multline*}
  under $\Delta Y$-equivalence and duality.
\end{lemma}

Let $M$ be a uniquely $\GF(5)$-representable matroid that is not dyadic.
Then $M$ has an $\mathcal{N}_1$-minor.
In particular, if $|E(M)| \le 10$, then $M$ or $M^*$ is $\Delta Y$-equivalent to a matroid that has an $N$-minor, for some matroid $N$ listed in \cref{h1minors}.

Note that, by \cref{thmh2,h1minors}, there are $180$ excluded minors for $\mathbb{H}_2$-representability that are also excluded minors for the class of $\GF(5)$-representable matroids.

Finally, we obtain our main result, \cref{thm1}, which we restate here.

\begingroup
\def\thetheorem{\ref{thm1}}
\begin{theorem}
There are $2128$ excluded minors for $\GF(5)$-representability on $10$ elements.
\end{theorem}
\addtocounter{theorem}{-1}
\endgroup

\begin{proof}
  Again, we use a similar approach as in the proof of \cref{thmh4,thmh3,thmh2}, only this time using \cref{h1minors}.
  The number of $n$-element rank-$r$ $3$-connected $\GF(5)$-representable matroids with an $\mathcal{N}_1$-minor %(that is, the number of $3$-connected uniquely $\GF(5)$-representable matroids)
  are given in \cref{gf5table}.
\end{proof}

\begin{table}[hbtp]
  \begin{tabular}{r|r r r}
    \hline
$r \ba n$ &   8 &    9 &     10 \\
    \hline
        3 &   4 &   33 &    120 \\
        4 &  46 & 2237 &  51823 \\
        5 &   4 & 2237 & 327186 \\
        6 &     &   33 &  51823 \\
        7 &     &      &    120 \\
       \hline                                           
      Total & 54 & 4540 & 431072 \\
       \hline
  \end{tabular}
  \caption{The number of $3$-connected $n$-element rank-$r$ matroids that are uniquely $\GF(5)$-representable but not dyadic, for $n \le 10$.}
  %\caption{The number of $3$-connected $\GF(5)$-representable $n$-element rank-$r$ matroids with a $\mathcal{N}_1$-minor, for $n \le 10$.}
  \label{gf5table}
\end{table}

\section{Final remarks}
\label{finalsec}

In \cref{gf5excounttable}, we provide the number of excluded minors for $\GF(5)$-representability of each rank and size (up to at most 10 elements).
We also consider the number of $\Delta Y$-equivalence classes for each size.

As remarked prior to \cref{osvdelta}, the excluded minors for representability over a partial field $\mathbb{P}$ are closed not only under $\Delta$-$Y$ exchange, but also under segment-cosegment exchange (a generalisation of $\Delta$-$Y$ exchange)~\cite{OSV2000}.
Thus, we could instead consider the number of equivalence classes under segment-cosegment exchange; however, the only effect would be merging two of the classes on $10$ elements, reducing the number of equivalence classes from $920$ to $919$.

\begin{table}[bhtp]
  \begin{tabular}{c | c c c}
    \hline
  $n$ &   $r$ & \#(matroids) & \#($\Delta Y^{(*)}$-classes) \\
    \hline
    7 & 2 / 5 &            1 &                              \\
      & 3 / 4 &            5 &                              \\ \cline{2-4}
      & Total &           12 &                            4 \\ 
    \hline
    8 & 3 / 5 &            2 &                              \\
      &     4 &           92 &                              \\ \cline{2-4}
      & Total &           96 &                           73 \\
    \hline
    9 & 3 / 6 &            9 &                              \\
      & 4 / 5 &          219 &                              \\ \cline{2-4}
      & Total &          456 &                          131 \\ 
    \hline
   10 & 3 / 7 &          1   &                              \\
      & 4 / 6 &         130  &                              \\
      & 5     &         1866 &                              \\ \cline{2-4}
      & Total &         2128 &                          920 \\
    \hline
  \end{tabular}
  %\begin{tabular}{c | c c c c}
    %\hline
  %$n$ & $r$ & \#(matroids) & \#($\Delta Y^{(*)}$-classes) & \#($\Delta\nabla^{(*)}$-classes)\\
    %\hline
    %7 & 2 / 5 &  1 &   \\
      %& 3 / 4 &  5 &   \\ \cline{2-5}
      %& Total & 12 & 4 & 4 \\
    %\hline
    %8 & 3 / 5 &  2 &   \\
      %&     4 & 92 &   \\ \cline{2-5}
      %& Total & 96 & 73 & 73 \\
    %\hline
    %9 & 3 / 6 &   9 &     \\
      %& 4 / 5 & 219 &     \\ \cline{2-5}
      %& Total & 456 & 131 & 131 \\
    %\hline
   %10 & 3 / 7 & 1    &     \\
      %& 4 / 6 & 130  &     \\
      %& 5     & 1866 &     \\ \cline{2-5}
      %& Total & 2128 & 920 & 919 \\
    %\hline
  %\end{tabular}
  \caption{The number of excluded minors for the class of $\GF(5)$-representable matroids on $n$ elements and rank $r$, and the number of $\Delta Y$-equivalence classes.}
  \label{gf5excounttable}
\end{table}

Of the $920$ dual-closed $\Delta Y$-equivalence classes of $10$-element excluded minors for the class of $\GF(5)$-representable matroids, $513$ of the classes consist of matroids that are not representable over any field.  This corresponds to $1235$ out of the $2128$ excluded minors on $10$ elements.

A partial field $\mathbb{P}$ is \emph{universal} if it is the universal partial field of some matroid~$M$; that is, $\mathbb{P} = \mathbb{P}_M$.
Van Zwam showed that several well-known partial fields are universal~\cite[Table~3.2]{vanZwam2009}, but commented that it was not known if each Hydra-$i$ partial field is universal, for $i \in \{2,3,4,5\}$ \cite[page~92]{vanZwam2009}.
It follows from \cref{3reglemma} that $\mathbb{P}_{U_{2,6}}=\mathbb{H}_5$.
In the process of finding the excluded minors for $\GF(5)$, we were also able to find a matroid $M$ such that $\mathbb{P}_M=\mathbb{H}_i$ for each $i$, thereby showing that each Hydra partial field is universal.
These are given in \cref{hydrauniversal}.  We note that each matroid in this table has at most $9$ elements.
%In the following table, we provide a matroid $M$ such that $\mathbb{P}_M=\mathbb{H}_i$, for each $i$, and $|E(M)| \le 9$.
%Matroids with Hydra Universal partial fields:

\begin{table}[hb]
\begin{tabular}{c | c c c c c}
    \hline
    $\mathbb{P}_M$ &
    $\mathbb{H}_5$ &
    $\mathbb{H}_4$ &
    $\mathbb{H}_3$ &
    $\mathbb{H}_2$ &
    $\GF(5)$ \\
    $M$ &
    $U_{2,6}$ &
    $P_8^=$ &
    $(\pappus \ba e)^+$ &
    $M_{16834}$ &
    $M_{15729}$ \\
    \hline
\end{tabular}
  \caption{Matroids with Hydra universal partial fields}
  \label{hydrauniversal}
\end{table}

\bibliographystyle{abbrv}
\bibliography{lib}

\appendix

\section*{Appendix}

Here we provide representations for the $\GF(5)$-representable matroids appearing as excluded minors for $\mathbb{H}_i$-representable matroids, for some $i \in \{2,3,4,5\}$.

%Note that for $\mathbb{H}_4$-representable matroids, the four inequivalent $\GF(5)$-representations can be obtained by substituting $(\alpha,\beta) \in \{(2,2), (3,3), (3,4), (4,3)\}$.
%%
%For $\mathbb{H}_3$-representable matroids, the three inequivalent $\GF(5)$-representations can be obtained by substituting $\alpha \in \{2,3,4\}$.
%%
%For $\mathbb{H}_2$-representable matroids that are not dyadic, the two inequivalent $\GF(5)$-representations can be obtained by substituting $i \in \{2,3\}$.

In what follows, matroids appear just once.
That is, when $N$ is a excluded minor for the class of $\mathbb{H}_i$-representable matroids, and $N$ is $\mathbb{H}_j$-representable but not $\mathbb{H}_{j+1}$-representable for some $j < i-1$, then $N$ appears in the ``Excluded minors for $\mathbb{H}_i$-representable matroids'' section under the ``$\mathbb{H}_j$-representable'' heading (but not in the ``Excluded minors for $\mathbb{H}_{i-1}$-representable matroids'' section, for example, despite $N$ also being an excluded minor for this class).

Matroids on at most nine elements are typically given the name $M_n$, where $n$ is the number in Mayhew and Royle's catalog of all matroids on at most nine elements.
Exceptions are made for well-known matroids that have notation associated with them, or are free extensions or single-element deletions of such matroids, or for matroids that have been discovered in ongoing work on matroids representable over all fields of size at least four~\cite{Brettell24}.

\subsection*{Excluded minors for \texorpdfstring{$\mathbb{H}_5$}{H5}-representable matroids}
\begin{itemize}
  \item \emph{$\mathbb{H}_4$-representable:} % on 7 elements:}

%The following are $\mathbb{H}_4$-representations of $M(K_4)+e$ and $\Lambda_3$, respectively.

\noindent
\begin{tabular}{p{.45\textwidth} p{0.45\textwidth}}
$M(K_4)+e$:
$\begin{bmatrix}
 1 &  \alpha & \alpha & 1 \\
 0 &       1 &      1 & 1 \\
 1 &       0 & \alpha & \frac{\beta(\alpha-1)}{1-\beta} \\
\end{bmatrix}$ & %
$\Lambda_3$:
$\begin{bmatrix}
 1 &      1 & \alpha+\beta-2\alpha\beta & \alpha\beta-1 \\
 1 & \alpha &                         0 & \alpha(\beta-1) \\
 1 &      0 &           \alpha(1-\beta) & \alpha(\beta-1) \\
\end{bmatrix}$
\end{tabular}

%\item \emph{$\mathbb{H}_4$-representable:} % on 8 elements:}

%And a $\mathbb{H}_4$-representation of
\noindent
\begin{tabular}{p{.45\textwidth} p{0.45\textwidth}}
$P_8^=$:
$\begin{bmatrix}
             1 &                     -1 &                      -1 & 1 \\
             1 & \frac{\alpha+\beta-2\alpha\beta}{\alpha(\beta-1)} &                       0 & 1 \\
             1 &                     -1 & \frac{(\alpha-1)(\beta-1)}{\alpha\beta-1} & 0 \\
 \frac{1}{\beta-1} &                      1 &                      1  & \frac{\alpha(\beta-1)}{\alpha+\beta-2\alpha\beta}\\
\end{bmatrix}$ &
\end{tabular}

\item \emph{$\mathbb{H}_2$-representable:} % on 7 elements:}

\noindent
\begin{tabular}{p{.45\textwidth} p{0.45\textwidth}}
$F_7^-$:
$\begin{bmatrix}
  0 & 1 & 1 & 1 \\
  1 & 0 & 1 & 1 \\
  1 & 1 & 0 & 1
\end{bmatrix}$ & 
$F_7^=$:
$\begin{bmatrix}
  0 & 1 & i & 1 \\
  1 & 0 & 1 & 1 \\
  1 & 1 & 0 & 1
\end{bmatrix}$
\end{tabular}

\noindent
\begin{tabular}{p{.45\textwidth} p{0.45\textwidth}}
$\mathcal{W}^3+e$:
$\begin{bmatrix}
  1 & 0 & i & 1 \\
  i & 1 & 0 & 1 \\
  0 & i & 1 & 1
\end{bmatrix}$ & 
$Q_6+e$:
$\begin{bmatrix}
  \frac{i+1}{2} & 0 & i & 1 \\
  1 & 1 & 1 & 1 \\
  0 & \frac{1-i}{2} & -i & 1
\end{bmatrix}$
\end{tabular}

%\item \emph{$\mathbb{H}_2$-representable:} % on 8 elements:}

\noindent
\begin{tabular}{p{.45\textwidth} p{0.45\textwidth}}
$P_8$:
$\begin{bmatrix}
  0 & 1 & 1 & 2 \\
  1 & 0 & 1 & 1 \\
  1 & 1 & 0 & 1 \\
  2 & 1 & 1 & 0 \\
\end{bmatrix}$ & 
$P_8^-$:
$\begin{bmatrix}
  1 & 1 & 1 & i+1 \\
  1 & 0 & i+1 & i+1 \\
  1 & -i & 1 & 0 \\
  0 & 1 & 1 & 1 \\
\end{bmatrix}$
\end{tabular}

\noindent
\begin{tabular}{p{.45\textwidth} p{0.45\textwidth}}
$\TQ_8$:
$\begin{bmatrix}
  0 & i & 1 & 1 \\
  1 & 0 & i & i-1 \\
  1 & i & 0 & i \\
  1 & i-1 & 1 & 0 \\
\end{bmatrix}$ & 
\end{tabular}
\end{itemize}

\subsection*{Excluded minors for \texorpdfstring{$\mathbb{H}_4$}{H4}-representable matroids}

\begin{itemize}
  \item \emph{$\mathbb{H}_3$-representable:} % on 8 elements:}

%The following are $\mathbb{H}_3$-representations of $\pappus \ba e$ and $(\pappus \ba e)^+$, respectively.

\noindent
\begin{tabular}{p{.45\textwidth} p{0.45\textwidth}}
$\pappus \ba e$:
$\begin{bmatrix}
 1 & 1 &        0 &   \alpha & \frac{\alpha^2}{\alpha-1}\\
 1 & 0 &   \alpha & \alpha-1 & \alpha \\
 0 & 1 & 1-\alpha &        1 & 1 \\
\end{bmatrix}$ & 
$(\pappus \ba e)^+$:
$\begin{bmatrix}
 1 & 1 & 0         & \alpha   & \alpha \\
 1 & 0 & \alpha    & \alpha-1 & \alpha \\
 0 & 1 & 1-\alpha  & 1        & 1 \\
\end{bmatrix}$
\end{tabular}

%And the other $\mathbb{H}_3$-representable excluded minors on $8$ elements:

\noindent
\begin{tabular}{p{.45\textwidth} p{0.45\textwidth}}
$M_{1559}$: 
$\begin{bmatrix}
\alpha(1-\alpha) &               1 & 1 & 1 \\
     1 & \frac{\alpha-1}{\alpha^2} & 1 & 1 \\
     \alpha &               0 & \alpha & 1 \\
     0 &               1 & \alpha & 1 \\
\end{bmatrix}$ & 
$M_{1596}$: 
$\begin{bmatrix}
  1-\alpha &           \alpha &        1 & 1 \\
\frac{-\alpha}{1-\alpha} & \frac{1}{1-\alpha} & 1 & 1 \\
         1 &                0 &   \alpha & 1 \\
         0 &           \alpha &   \alpha & 1 \\
\end{bmatrix}$
\end{tabular}

\noindent
\begin{tabular}{p{.45\textwidth} p{0.45\textwidth}}
$M_{1543}$:
$\begin{bmatrix}
       1 &      1 &        1 & 1 \\
  \alpha & \alpha &        0 & 1 \\
  \alpha &      0 & \alpha^2 & \alpha^2-\alpha+1 \\
       0 &      1 & 1-\alpha & 1-\alpha \\
\end{bmatrix}$ &
$M_{1566}$:
$\begin{bmatrix}
                      1 &      1 & \frac{\alpha-1}{\alpha} & 1 \\
                      1 & \alpha &                  \alpha & 1 \\
\frac{\alpha}{\alpha-1} &      0 &                       1 & 1-\alpha \\
                      0 &      1 &                       1 & 1 \\
\end{bmatrix}$
\end{tabular}

\noindent
\begin{tabular}{p{.45\textwidth} p{0.45\textwidth}}
$M_{1598}$:
$\begin{bmatrix}
  1                                &   \alpha &      1 & 1 \\
  1                                & \alpha-1 & \alpha & \alpha \\
  \frac{\alpha}{\alpha^2-\alpha+1} &        0 & \alpha & 1 \\
  0                                &   \alpha & \alpha & 1 \\
\end{bmatrix}$ &
\end{tabular}

%\item \emph{$\mathbb{H}_3$-representable:} % on 9 elements:}

%the $9$-element rank-$3$ matroid (TODO: try (again) to find nice drawing?)

\noindent 
\begin{tabular}{p{.45\textwidth} p{0.45\textwidth}}
$M_{3240}$:
$\begin{bmatrix}
 1 &   0 & \frac{\alpha^2}{\alpha^2-\alpha+1} & \frac{\alpha^2}{\alpha^2-\alpha+1} & 1 & \alpha \\
 1 & \alpha^2 &                   \alpha &                   0 & \alpha & \alpha^2 \\
 1 & \alpha-1 &                   1 &                   1 & 1 & \alpha-1 \\
\end{bmatrix}$ &
\end{tabular}

%\item \emph{$\mathbb{H}_3$-representable:} % on 10 elements:}

\noindent 
\begin{tabular}{p{.45\textwidth} p{0.45\textwidth}}
$\Fq$:
$\begin{bmatrix}
 1 & 1 & 1 & 1 & 0 \\
 0 & \alpha(\alpha-1) & 0 & \alpha^2-\alpha+1 & \alpha^2-\alpha+1 \\
 \alpha & 0 & \alpha & 1 & 0 \\
 \alpha^2 & \alpha & 0 & \alpha-1 & \alpha-1 \\
 \alpha(1-\alpha) & 0 & \alpha & 1 & 1 \\
\end{bmatrix}$ &
\end{tabular}

\item \emph{$\mathbb{H}_2$-representable:} % on 8 elements:}

\noindent
\begin{tabular}{p{.45\textwidth} p{0.45\textwidth}}
$M_{1537}$:
$\begin{bmatrix}
1 & 1 & 1 & 1 \\
\frac{1}{i} & \frac{1}{i} & 1 & 1 \\
-1 & 0 & i & 1 \\
0 & 1 & i & 1
\end{bmatrix}$ & %
$M_{1554}$:
$\begin{bmatrix}
1 & 1 & 1 & 1 \\
-1 & -1 & 1 & 1 \\
i & 0 & i & 1 \\
0 & -i & i & 1
\end{bmatrix}$
\end{tabular}

\noindent
\begin{tabular}{p{.45\textwidth} p{0.45\textwidth}}
$M_{1556}$:
$\begin{bmatrix}
1 & 1 & 1 & 1 \\
-\frac{1}{i} & -\frac{1}{i} & 1 & 1 \\
-1 & 0 & i & 1 \\
0 & 1 & i & 1
\end{bmatrix}$ & %
$M_{1569}$:
$\begin{bmatrix}
  1 & 1 & 2 & 1 \\
  1 & 1-i & 1-i & 1 \\
  i & 0 & i(1-i) & 1 \\
  0 & 1 & 1 & 1
\end{bmatrix}$
\end{tabular}

\noindent
\begin{tabular}{p{.45\textwidth} p{0.45\textwidth}}
$M_{1676}$:
$\begin{bmatrix}
1 & 1 & 1 & 1 \\
-1 & -1 & 1 & 1 \\
i & 0 & -1 & 1 \\
0 & -i & -1 & 1
\end{bmatrix}$ & %
\end{tabular}

%\item \emph{$\mathbb{H}_2$-representable:} % on 10 elements:}

\noindent
\begin{tabular}{p{.45\textwidth} p{0.45\textwidth}}
$\Fr_1$:
$\begin{bmatrix}
1 & i & 1 & 0 & 0 \\
1 & 1 & 1 & 1 & 1 \\
1 & i & 1 & i & i \\
1 & -i & 0 & 1 & 0 \\
1 & 0 & 0 & 1 & \frac{i+1}{2} \\
\end{bmatrix}$ & %
$\Fr_2$:
$\begin{bmatrix}
1 & 1 & 1 & 1 & 1 \\
1 & i+1 & 0 & 0 & 1-i \\
1 & 0 & \frac{1-i}{2} & 1 & 1 \\
1 & 0 & 1-i & 2 & 0 \\
1 & i+1 & 0 & 2 & i+1 \\
\end{bmatrix}$
\end{tabular}

\noindent
\begin{tabular}{p{.45\textwidth} p{0.45\textwidth}}
$\Fr_3$:
$\begin{bmatrix}
-i & 1 & 0 & 1 & -i \\
0 & 1 & \frac{i+1}{2} & 1 & 0 \\
i & 1 & i & 0 & 0 \\
1 & 1 & 1 & 0 & 1 \\
0 & 1 & 0 & 1 & \frac{1-i}{2} \\
\end{bmatrix}$ & %
\end{tabular}

\item \emph{Uniquely $\GF(5)$-representable:} % on 8 elements:}

\noindent
\begin{tabular}{p{.45\textwidth} p{0.45\textwidth}}
$M_{1601}$:
$\begin{bmatrix}
1 & 1 & 1 & 1 \\
1 & 4 & 3 & 0 \\
1 & 1 & 4 & 4 \\
1 & 4 & 0 & 3 \\
\end{bmatrix}$ & %
$\lambda_4^+$: %($M_{1662}$)
$\begin{bmatrix}
1 & 1 & 1 & 1 \\
1 & 3 & 0 & 4 \\
1 & 3 & 1 & 0 \\
1 & 2 & 3 & 3 \\
\end{bmatrix}$
\end{tabular}
\end{itemize}

\pagebreak
\subsection*{Excluded minors for \texorpdfstring{$\mathbb{H}_3$}{H3}-representable matroids}

\begin{itemize}
  \item \emph{$\mathbb{H}_2$-representable:} % on 9 elements:}

\noindent
\begin{tabular}{p{.46\textwidth} p{0.46\textwidth}}
$M_{3241}$:
$\begin{bmatrix}
  \frac{2}{1+i} & \frac{2}{1-i} & 0 & 1 & 1 & \frac{2}{1-i} \\
  1 & 0 & \frac{1-i}{2} & 1 & 1 & 1 \\
  0 & 1 & 1 & \frac{1}{1-i} & 1 & 1 \\
\end{bmatrix}$ & %
$M_{10782}$:
$\begin{bmatrix}
  0 & i & -i & 0 & 1 \\
  0 & 0 & 1 & \frac{1-i}{2} & 1 \\
  \frac{i}{1+i} & 1 & 0 & 0 & 1 \\
  1 & 1 & 1 & 1 & 1 \\
\end{bmatrix}$
\end{tabular}

\noindent
\begin{tabular}{p{.46\textwidth} p{0.46\textwidth}}
$M_{16834}$:
$\begin{bmatrix}
  1 & 1 & 1 & 1 & 0 \\
  1 & 1-i & 1-i & 0 & 0 \\
  1 & 0 & 2 & 1-i & 1 \\
  1 & 1-i & 0 & 1 & i \\
\end{bmatrix}$ & %
$M_{17078}$:
$\begin{bmatrix}
  1 & 1 & i+1 & \frac{i+1}{2} & 0 \\
  1 & 1 & i & 0 & -i \\
  \frac{i+1}{2} & 1 & 0 & \frac{1}{2} & 1 \\
  0 & 1 & 1 & 1 & 1 \\
\end{bmatrix}$
\end{tabular}

\noindent
\begin{tabular}{p{.46\textwidth} p{0.46\textwidth}}
$M_{17759}$:
$\begin{bmatrix}
  0 & i & 1 & 1 & i+1 \\
  1 & 1 & 1 & 1 & 1 \\
  0 & -1 & 0 & 1 & i \\
  1 & 0 & \frac{i}{i+1} & 1 & i \\
\end{bmatrix}$ & %
$M_{47015}$:
$\begin{bmatrix}
  1 & 1 & 0 & \frac{1-i}{2} & \frac{1+i}{2} \\
  1 & 1 & 1 & 1 & 1 \\
  1 & i+1 & i & 1 & i \\
  1 & i+1 & i & 0 & i+1 \\
\end{bmatrix}$
\end{tabular}

  %\item \emph{$\mathbb{H}_2$-representable:} % on 10 elements:}

%``h3ex10a'', ``h3ex10b'', ``h3ex10c'', and ``h3ex10d''.
\noindent
\begin{tabular}{p{.45\textwidth} p{0.45\textwidth}}
$\Fs_1$:
$\begin{bmatrix}
  \frac{i}{i+1} &             1 & 1 & 0 & i & 1             & \frac{1}{i+1} \\
  \frac{i+1}{2} &             1 & 1 & 1 & 0 & \frac{i+1}{2} & 1 \\
  \frac{1}{i+1} & \frac{1}{i+1} & 1 & 1 & 1 & \frac{1}{i+1} & \frac{1}{i+1} \\
\end{bmatrix}$ & %
$\Fs_2$:
$\begin{bmatrix}
  0 & -i & 0 & i & -i & 1 \\
  1 & 1 & 1 & 1 & 1 & 1 \\
  i & i & 0 & i & 1 & 1 \\
  i & 0 & \frac{i+1}{2} & i & 1 & 1 \\
\end{bmatrix}$
\end{tabular}

\noindent
\begin{tabular}{p{.45\textwidth} p{0.45\textwidth}}
$\Fs_3$:
$\begin{bmatrix}
  0 & \frac{i+1}{2} & 1 & \frac{1-i}{2} & 0 & 0 \\
  1 & 1 & 1 & 0 & i & -1 \\
  1 & 0 & 1 & 1 & 1 & 1 \\
  \frac{i+1}{i} & 0 & 0 & 1 & i+1 & 0 \\
\end{bmatrix}$ & %
$\Fs_4$:
$\begin{bmatrix}
  0 & 0 & 1 & \frac{i+1}{2} & \frac{1-i}{2} & 0 \\
  1 & 1 & 1 & 1 & 1 & 1 \\
  \frac{i+1}{2} & 1 & 1 & \frac{i+1}{2} & \frac{1-i}{2} & \frac{1-i}{2} \\
  1 & 1-i & 1 & 0 & 1-i & -i \\
\end{bmatrix}$
\end{tabular}

%``h3ex10p1'', ``h3ex10p2'', ``h3ex10p4'', ``h3ex10p5'', ``h3ex10p6'', ``h3ex10p7'', and ``h3ex10p8''.
\noindent
\begin{tabular}{p{.45\textwidth} p{0.45\textwidth}}
$\Ft_1$:
$\begin{bmatrix}
  i+1 & \frac{i+1}{2} & 1 & i+1 & 0 \\
  1 & 1 & 1 & 1 & 1 \\
  1 & 1 & 1 & 0 & 0 \\
  i+1 & \frac{i+1}{2} & 1 & i & 1 \\
  0 & 0 & 1 & i & 1 \\
\end{bmatrix}$ & %
$\Ft_2$:
$\begin{bmatrix}
  1 & 0 & 1-i & i+1 & 2 \\
  1 & 1 & 0 & 1-i & 0 \\
  1 & i+1 & i+1 & 1-i & 0 \\
  1 & 1 & 1 & 1 & 1 \\
  1 & 0 & 0 & 1 & \frac{2}{i+1} \\
\end{bmatrix}$
\end{tabular}

\noindent
\begin{tabular}{p{.45\textwidth} p{0.45\textwidth}}
$\Ft_4$:
$\begin{bmatrix}
  1 & i & 0 & \frac{i+1}{2} & i+1 \\
  1 & 1 & 1 & 1 & 1 \\
  1 & i & 0 & i & i \\
  1 & -i & 1 & 0 & 0 \\
  1 & 1 & 1 & \frac{i+1}{2} & 0 \\
\end{bmatrix}$ & %
$\Ft_5$:
$\begin{bmatrix}
  i+1 & i & i+1 & 1 & 0 \\
  i+1 & 0 & \frac{i+1}{2} & 1 & 1 \\
  1 & 0 & 0 & 1 & 1 \\
  1 & 1 & 1 & 1 & 1 \\
  0 & 1 & 0 & 1 & i+1 \\
\end{bmatrix}$
\end{tabular}

\noindent
\begin{tabular}{p{.45\textwidth} p{0.45\textwidth}}
$\Ft_6$:
$\begin{bmatrix}
1 & 1 & 1 & 1 & 1 \\
1 & i & i & 0 & 0 \\
1 & 0 & 1 & 1-i & 1 \\
1 & 0 & 1 & 0 & \frac{1}{i} \\
1 & i & i & 1 & 1 \\
\end{bmatrix}$ & %
$\Ft_7$:
$\begin{bmatrix}
1 & 1 & 1 & 1 & 1 \\
2 & 1 & 1 & 2 & 1 \\
\frac{2}{i+1} & 1 & 0 & 2 & 1 \\
1 & 1 & 1 & i+1 & \frac{i+1}{2} \\
0 & 1 & 0 & i+1 & i \\
\end{bmatrix}$
\end{tabular}

\noindent
\begin{tabular}{p{.45\textwidth} p{0.45\textwidth}}
$\Ft_8$:
$\begin{bmatrix}
1 & 1 & 1 & 1 & 1 \\
1 & i+1 & 1-i & 0 & 0 \\
1 & 0 & 1 & -i & \frac{1-i}{2} \\
1 & 0 & 1 & 0 & 1 \\
1 & 1-i & 0 & 1 & 1 \\
\end{bmatrix}$ & %
%``h3ex10s3'', ``h3ex10s5'', and ``h3ex10s8''.
$\Fu_3$:
$\begin{bmatrix}
  0 & 1 & 1 & i & 1 \\
  0 & 1 & 1 & -1 & 0 \\
  \frac{1}{2} & 0 & 1 & 0 & \frac{1}{2} \\
  1 & 1 & 1 & 1 & 1 \\
  \frac{i+1}{2} & 0 & 1 & 1 & 1 \\
\end{bmatrix}$ %
\end{tabular}

\noindent
\begin{tabular}{p{.45\textwidth} p{0.45\textwidth}}
$\Fu_5$:
$\begin{bmatrix}
  1 & 0 & 0 & \frac{1-i}{2} & \frac{1-i}{2} \\
  1 & 1 & 1 & 1 & 1 \\
  1 & \frac{1}{2} & 0 & \frac{1}{2} & \frac{1}{2} \\
  1 & 1 & 1 & 0 & \frac{1}{2} \\
  1 & \frac{i+1}{2} & 1 & \frac{1-i}{2} & \frac{1}{2} \\
\end{bmatrix}$ &
$\Fu_8$:
$\begin{bmatrix}
  0 & -1 & 1 & i & 1 \\
  1 & 1 & 1 & 0 & i \\
  1 & 1 & 1 & 1 & 1 \\
  \frac{1-i}{2} & 0 & 1 & 0 & 1 \\
  0 & i & 1 & \frac{1+i}{2} & 0 \\
\end{bmatrix}$ 
\end{tabular}

  \item \emph{Uniquely $\GF(5)$-representable:} % on 10 elements:}

%``h3ex10p3'', ``h3ex10s1'', $\UG_{10}$, ``h3ex10s6'', and ``h3ex10s7''.
\noindent
\begin{tabular}{p{.45\textwidth} p{0.45\textwidth}}
$\Ft_3$:
$\begin{bmatrix}
1 & 1 & 1 & 1 & 1 \\
1 & 2 & 2 & 0 & 0 \\
1 & 0 & 1 & 4 & 2 \\
1 & 0 & 1 & 0 & 1 \\
1 & 2 & 2 & 1 & 1 \\
\end{bmatrix}$ & %
$\Fu_1$:
$\begin{bmatrix}
1 & 1 & 1 & 1 & 1 \\
1 & 3 & 3 & 0 & 0 \\
1 & 0 & 1 & 3 & 1 \\
1 & 0 & 1 & 0 & 3 \\
1 & 3 & 3 & 1 & 1 \\
\end{bmatrix}$
\end{tabular}

\noindent
\begin{tabular}{p{.45\textwidth} p{0.45\textwidth}}
$\Fu_6$:
$\begin{bmatrix}
1 & 0 & 4 & 1 & 1 \\
1 & 1 & 0 & 4 & 1 \\
1 & 4 & 3 & 0 & 0 \\
1 & 1 & 1 & 1 & 1 \\
1 & 1 & 0 & 2 & 4 \\
\end{bmatrix}$ & %
$\Fu_7$:
$\begin{bmatrix}
0 & 1 & 1 & 1 & 1 \\
1 & 0 & 0 & 1 & 1 \\
1 & 1 & 3 & 0 & 0 \\
0 & 1 & 4 & 2 & 4 \\
2 & 1 & 0 & 4 & 1 \\
\end{bmatrix}$
\end{tabular}

\noindent
\begin{tabular}{p{.45\textwidth} p{0.45\textwidth}}
$\UG_{10}$:
$\begin{bmatrix}
0 & 2 & 1 & 4 & 1 \\
1 & 1 & 1 & 0 & 4 \\
1 & 1 & 1 & 1 & 1 \\
3 & 0 & 1 & 0 & 1 \\
0 & 4 & 1 & 2 & 0 \\
\end{bmatrix}$ & %
\end{tabular}
\end{itemize}

\subsection*{Excluded minors for \texorpdfstring{$\mathbb{H}_2$}{H2}-representable matroids}
\begin{itemize}
  \item \emph{Uniquely $\GF(5)$-representable:} % on 10 elements:}

\noindent
\begin{tabular}{p{.46\textwidth} p{0.46\textwidth}}
$M_{822}$:
$\begin{bmatrix}
0 & 3 & 1 & 1 & 1\\
4 & 0 & 1 & 1 & 2\\
1 & 1 & 1 & 0 & 1\\
\end{bmatrix}$ & %
$M_{835}$:
$\begin{bmatrix}
1 & 1 & 0 & 1 & 2\\
2 & 0 & 1 & 1 & 1\\
0 & 3 & 2 & 1 & 1\\
\end{bmatrix}$
\end{tabular}

\noindent
\begin{tabular}{p{.46\textwidth} p{0.46\textwidth}}
$M_{836}$:
$\begin{bmatrix}
3 & 0 & 1 & 1 & 2\\
0 & 1 & 1 & 2 & 4\\
1 & 1 & 1 & 1 & 1\\
\end{bmatrix}$ & %
$M_{854}$:
$\begin{bmatrix}
4 & 0 & 2 & 1 & 1\\
0 & 2 & 3 & 3 & 1\\
1 & 1 & 1 & 1 & 1\\
\end{bmatrix}$
\end{tabular}

\noindent
\begin{tabular}{p{.46\textwidth} p{0.46\textwidth}}
$M_{1495}$:
$\begin{bmatrix}
1 & 1 & 2 & 0\\
1 & 0 & 2 & 2\\
1 & 1 & 1 & 1\\
1 & 2 & 0 & 2\\
\end{bmatrix}$ & %
$M_{1505}$:
$\begin{bmatrix}
3 & 1 & 2 & 0\\
1 & 1 & 1 & 1\\
0 & 1 & 2 & 2\\
2 & 1 & 0 & 2\\
\end{bmatrix}$
\end{tabular}

\noindent
\begin{tabular}{p{.46\textwidth} p{0.46\textwidth}}
$M_{1507}$:
$\begin{bmatrix}
4 & 1 & 1 & 0\\
1 & 1 & 1 & 1\\
0 & 1 & 2 & 2\\
2 & 1 & 0 & 2\\
\end{bmatrix}$ & %
$\KP_8$: %$M_{1511}$:
$\begin{bmatrix}
1 & 3 & 3 & 0\\
1 & 1 & 1 & 1\\
1 & 3 & 2 & 2\\
1 & 4 & 0 & 1\\
\end{bmatrix}$
\end{tabular}

\noindent
\begin{tabular}{p{.46\textwidth} p{0.46\textwidth}}
$M_{1519}$:
$\begin{bmatrix}
1 & 4 & 4 & 0\\
1 & 0 & 4 & 1\\
1 & 2 & 3 & 2\\
1 & 1 & 1 & 1\\
\end{bmatrix}$ & %
$M_{1530}$:
$\begin{bmatrix}
1 & 4 & 2 & 0\\
1 & 0 & 2 & 2\\
1 & 1 & 1 & 1\\
1 & 2 & 0 & 2\\
\end{bmatrix}$
\end{tabular}

\noindent
\begin{tabular}{p{.46\textwidth} p{0.46\textwidth}}
$M_{1535}$:
$\begin{bmatrix}
1 & 1 & 4 & 0\\
1 & 4 & 3 & 3\\
1 & 1 & 1 & 1\\
1 & 3 & 0 & 3\\
\end{bmatrix}$ & %
$M_{1538}$:
$\begin{bmatrix}
1 & 2 & 3 & 0\\
1 & 2 & 4 & 4\\
1 & 1 & 1 & 1\\
1 & 4 & 0 & 4\\
\end{bmatrix}$
\end{tabular}

\noindent
\begin{tabular}{p{.46\textwidth} p{0.46\textwidth}}
$M_{1541}$:
$\begin{bmatrix}
2 & 2 & 1 & 0\\
2 & 0 & 1 & 1\\
0 & 2 & 1 & 3\\
1 & 1 & 1 & 1\\
\end{bmatrix}$ & %
$M_{1542}$:
$\begin{bmatrix}
1 & 1 & 1 & 1\\
1 & 4 & 0 & 2\\
1 & 4 & 4 & 0\\
1 & 2 & 4 & 1\\
\end{bmatrix}$
\end{tabular}

\noindent
\begin{tabular}{p{.46\textwidth} p{0.46\textwidth}}
$M_{1553}$:
$\begin{bmatrix}
1 & 0 & 3 & 1\\
1 & 4 & 0 & 4\\
1 & 4 & 1 & 0\\
1 & 1 & 1 & 1\\
\end{bmatrix}$ & %
$M_{1591}$:
$\begin{bmatrix}
3 & 0 & 1 & 3\\
0 & 4 & 1 & 2\\
3 & 1 & 1 & 0\\
1 & 1 & 1 & 1\\
\end{bmatrix}$
\end{tabular}

\noindent
\begin{tabular}{p{.46\textwidth} p{0.46\textwidth}}
$M_{1593}$:
$\begin{bmatrix}
1 & 1 & 1 & 1\\
1 & 4 & 2 & 1\\
1 & 3 & 4 & 3\\
1 & 2 & 3 & 3\\
\end{bmatrix}$ & %
$M_{1599}$:
$\begin{bmatrix}
1 & 1 & 1 & 1\\
1 & 3 & 4 & 3\\
1 & 2 & 1 & 4\\
1 & 4 & 3 & 0\\
\end{bmatrix}$
\end{tabular}

\noindent
\begin{tabular}{p{.46\textwidth} p{0.46\textwidth}}
$M_{1600}$:
$\begin{bmatrix}
1 & 1 & 1 & 1\\
1 & 4 & 1 & 2\\
1 & 1 & 4 & 0\\
1 & 2 & 4 & 1\\
\end{bmatrix}$ & %
$M_{1609}$:
$\begin{bmatrix}
1 & 1 & 1 & 1\\
1 & 3 & 3 & 4\\
1 & 4 & 3 & 0\\
1 & 2 & 0 & 1\\
\end{bmatrix}$
\end{tabular}

\noindent
\begin{tabular}{p{.46\textwidth} p{0.46\textwidth}}
$M_{1610}$:
$\begin{bmatrix}
1 & 1 & 1 & 1\\
1 & 3 & 0 & 4\\
1 & 4 & 2 & 0\\
1 & 2 & 2 & 1\\
\end{bmatrix}$ & %
$M_{1618}$:
$\begin{bmatrix}
1 & 1 & 1 & 1\\
1 & 1 & 4 & 3\\
1 & 3 & 4 & 0\\
1 & 2 & 0 & 1\\
\end{bmatrix}$
\end{tabular}

\noindent
\begin{tabular}{p{.46\textwidth} p{0.46\textwidth}}
$M_{1627}$:
$\begin{bmatrix}
1 & 1 & 1 & 1\\
1 & 2 & 1 & 4\\
1 & 1 & 3 & 0\\
1 & 2 & 3 & 1\\
\end{bmatrix}$ & %
$M_{1674}$:
$\begin{bmatrix}
1 & 1 & 1 & 1\\
1 & 0 & 3 & 4\\
1 & 4 & 3 & 0\\
1 & 2 & 4 & 1\\
\end{bmatrix}$
\end{tabular}

\noindent
\begin{tabular}{p{.46\textwidth} p{0.46\textwidth}}
$M_{1680}$:
$\begin{bmatrix}
1 & 1 & 1 & 1\\
1 & 4 & 4 & 3\\
1 & 3 & 1 & 0\\
1 & 2 & 4 & 1\\
\end{bmatrix}$ & %
$M_{1695}$:
$\begin{bmatrix}
1 & 1 & 1 & 1\\
1 & 4 & 0 & 2\\
1 & 2 & 1 & 0\\
1 & 4 & 2 & 1\\
\end{bmatrix}$
\end{tabular}

\noindent
\begin{tabular}{p{.46\textwidth} p{0.46\textwidth}}
$M_{1698}$:
$\begin{bmatrix}
1 & 1 & 1 & 1\\
1 & 3 & 3 & 2\\
1 & 2 & 3 & 0\\
1 & 4 & 2 & 1\\
\end{bmatrix}$ & %
$M_{1713}$:
$\begin{bmatrix}
1 & 0 & 4 & 1\\
1 & 2 & 4 & 3\\
1 & 2 & 1 & 0\\
1 & 1 & 1 & 1\\
\end{bmatrix}$
\end{tabular}

\noindent
\begin{tabular}{p{.46\textwidth} p{0.46\textwidth}}
$M_{1729}$:
$\begin{bmatrix}
1 & 1 & 1 & 1\\
0 & 1 & 2 & 4\\
1 & 1 & 2 & 0\\
2 & 1 & 4 & 2\\
\end{bmatrix}$ & %
\end{tabular}

%#on 9 elements

\noindent
\begin{tabular}{p{.46\textwidth} p{0.46\textwidth}}
$M_{3193}$:
$\begin{bmatrix}
1 & 0 & 1 & 0 & 3 & 1\\
1 & 4 & 0 & 1 & 0 & 1\\
0 & 1 & 1 & 1 & 1 & 1\\
\end{bmatrix}$ & %
$\BBR_9$:
$\begin{bmatrix}
1 & 0 & 1 & 4 & 1 & 4\\
1 & 3 & 1 & 3 & 0 & 0\\
0 & 1 & 1 & 1 & 1 & 1\\
\end{bmatrix}$
\end{tabular}

\noindent
\begin{tabular}{p{.46\textwidth} p{0.46\textwidth}}
$A_9$:
$\begin{bmatrix}
1 & 3 & 3 & 4 & 0 & 4\\
1 & 3 & 2 & 3 & 3 & 0\\
1 & 1 & 1 & 1 & 1 & 1\\
\end{bmatrix}$ & %
$\BBR_9^{\Delta\nabla}$:
$\begin{bmatrix}
3 & 0 & 1 & 0\\
0 & 3 & 1 & 1\\
1 & 1 & 1 & 1\\
1 & 2 & 0 & 3\\
1 & 0 & 4 & 0\\
\end{bmatrix}$
\end{tabular}

\noindent
\begin{tabular}{p{.46\textwidth} p{0.46\textwidth}}
$M_{9574}$:
$\begin{bmatrix}
2 & 0 & 1 & 1 & 1\\
0 & 3 & 1 & 4 & 0\\
1 & 1 & 1 & 0 & 2\\
0 & 1 & 1 & 2 & 0\\
\end{bmatrix}$ & %
$M_{9841}$:
$\begin{bmatrix}
3 & 0 & 1 & 1 & 0\\
1 & 1 & 1 & 1 & 1\\
0 & 2 & 1 & 2 & 0\\
1 & 2 & 1 & 0 & 2\\
\end{bmatrix}$
\end{tabular}

\noindent
\begin{tabular}{p{.46\textwidth} p{0.46\textwidth}}
$M_{9844}$:
$\begin{bmatrix}
2 & 0 & 1 & 2 & 0\\
1 & 1 & 1 & 1 & 1\\
0 & 2 & 1 & 1 & 0\\
1 & 2 & 1 & 2 & 1\\
\end{bmatrix}$ & %
$\PP_9$:
$\begin{bmatrix}
1 & 1 & 1 & 4 & 4\\
1 & 2 & 1 & 0 & 4\\
1 & 0 & 4 & 4 & 4\\
0 & 1 & 1 & 1 & 1\\
\end{bmatrix}$
\end{tabular}

\noindent
\begin{tabular}{p{.46\textwidth} p{0.46\textwidth}}
$M_{10025}$:
$\begin{bmatrix}
4 & 0 & 1 & 4 & 0\\
1 & 1 & 1 & 1 & 1\\
0 & 4 & 1 & 2 & 0\\
1 & 2 & 1 & 2 & 2\\
\end{bmatrix}$ & %
$M_{10052}$:
$\begin{bmatrix}
4 & 0 & 1 & 4 & 0\\
1 & 1 & 1 & 1 & 1\\
0 & 2 & 1 & 1 & 0\\
1 & 2 & 1 & 2 & 1\\
\end{bmatrix}$
\end{tabular}

\noindent
\begin{tabular}{p{.46\textwidth} p{0.46\textwidth}}
$M_{10241}$:
$\begin{bmatrix}
4 & 0 & 1 & 3 & 4\\
1 & 2 & 1 & 2 & 0\\
0 & 4 & 1 & 1 & 0\\
1 & 1 & 1 & 1 & 1\\
\end{bmatrix}$ & %
$M_{10242}$:
$\begin{bmatrix}
3 & 0 & 1 & 3 & 0\\
1 & 1 & 1 & 1 & 1\\
0 & 2 & 1 & 4 & 2\\
1 & 4 & 1 & 4 & 0\\
\end{bmatrix}$
\end{tabular}

\noindent
\begin{tabular}{p{.46\textwidth} p{0.46\textwidth}}
$M_{10244}$:
$\begin{bmatrix}
4 & 0 & 3 & 1 & 1\\
4 & 1 & 4 & 1 & 1\\
0 & 1 & 2 & 2 & 1\\
1 & 1 & 1 & 1 & 1\\
\end{bmatrix}$ & %
$M_{10251}$:
$\begin{bmatrix}
3 & 0 & 1 & 4 & 3\\
1 & 4 & 1 & 2 & 4\\
0 & 2 & 1 & 1 & 2\\
1 & 1 & 1 & 1 & 1\\
\end{bmatrix}$
\end{tabular}

\noindent
\begin{tabular}{p{.46\textwidth} p{0.46\textwidth}}
$M_{10287}$:
$\begin{bmatrix}
2 & 0 & 1 & 3 & 0\\
1 & 4 & 1 & 0 & 0\\
0 & 1 & 1 & 2 & 2\\
1 & 1 & 1 & 1 & 1\\
\end{bmatrix}$ & %
$M_{10368}$:
$\begin{bmatrix}
4 & 0 & 1 & 2 & 1\\
0 & 2 & 1 & 3 & 0\\
1 & 1 & 1 & 0 & 4\\
0 & 1 & 1 & 1 & 1\\
\end{bmatrix}$
\end{tabular}

\noindent
\begin{tabular}{p{.46\textwidth} p{0.46\textwidth}}
$M_{10459}$:
$\begin{bmatrix}
4 & 0 & 1 & 4 & 0\\
1 & 1 & 1 & 1 & 1\\
0 & 2 & 1 & 1 & 2\\
1 & 2 & 1 & 2 & 0\\
\end{bmatrix}$ & %
$M_{11161}$:
$\begin{bmatrix}
2 & 0 & 1 & 1 & 0\\
1 & 1 & 1 & 1 & 1\\
0 & 3 & 1 & 2 & 0\\
1 & 2 & 1 & 0 & 2\\
\end{bmatrix}$
\end{tabular}

\noindent
\begin{tabular}{p{.46\textwidth} p{0.46\textwidth}}
$M_{11162}$:
$\begin{bmatrix}
3 & 0 & 1 & 1 & 0\\
1 & 1 & 1 & 1 & 1\\
0 & 3 & 1 & 4 & 0\\
1 & 4 & 1 & 0 & 4\\
\end{bmatrix}$ & %
$M_{11347}$:
$\begin{bmatrix}
4 & 0 & 1 & 3 & 0\\
1 & 4 & 1 & 0 & 0\\
0 & 1 & 1 & 4 & 4\\
1 & 1 & 1 & 1 & 1\\
\end{bmatrix}$
\end{tabular}

\noindent
\begin{tabular}{p{.46\textwidth} p{0.46\textwidth}}
$M_{12000}$:
$\begin{bmatrix}
3 & 0 & 1 & 2 & 4\\
1 & 1 & 1 & 1 & 1\\
0 & 3 & 1 & 4 & 0\\
1 & 4 & 1 & 0 & 3\\
\end{bmatrix}$ & %
$M_{12575}$:
$\begin{bmatrix}
2 & 0 & 1 & 1 & 0\\
1 & 1 & 1 & 1 & 1\\
0 & 4 & 1 & 3 & 0\\
1 & 3 & 1 & 0 & 3\\
\end{bmatrix}$
\end{tabular}

\noindent
\begin{tabular}{p{.46\textwidth} p{0.46\textwidth}}
$M_{15729}$:
$\begin{bmatrix}
3 & 0 & 1 & 2 & 2\\
1 & 3 & 1 & 4 & 0\\
0 & 1 & 1 & 3 & 2\\
1 & 1 & 1 & 1 & 1\\
\end{bmatrix}$ & %
$\TQ_9$:
$\begin{bmatrix}
1 & 0 & 2 & 1 & 1\\
4 & 0 & 0 & 2 & 1\\
1 & 4 & 0 & 0 & 2\\
1 & 1 & 1 & 1 & 0\\
\end{bmatrix}$
\end{tabular}

\noindent
\begin{tabular}{p{.46\textwidth} p{0.46\textwidth}}
$M_{16321}$:
$\begin{bmatrix}
3 & 0 & 1 & 2 & 3\\
1 & 2 & 1 & 0 & 1\\
0 & 1 & 1 & 3 & 2\\
1 & 1 & 1 & 1 & 1\\
\end{bmatrix}$ & %
$M_{16332}$:
$\begin{bmatrix}
3 & 0 & 1 & 1 & 1\\
1 & 3 & 1 & 2 & 0\\
0 & 1 & 1 & 4 & 1\\
1 & 1 & 1 & 3 & 0\\
\end{bmatrix}$
\end{tabular}

\noindent
\begin{tabular}{p{.46\textwidth} p{0.46\textwidth}}
$M_{18663}$:
$\begin{bmatrix}
2 & 0 & 1 & 1 & 1\\
1 & 4 & 1 & 3 & 0\\
0 & 1 & 1 & 2 & 1\\
1 & 1 & 1 & 4 & 0\\
\end{bmatrix}$ & %
$M_{46942}$:
$\begin{bmatrix}
3 & 0 & 1 & 3 & 1\\
1 & 1 & 1 & 1 & 1\\
0 & 1 & 3 & 3 & 1\\
3 & 1 & 3 & 1 & 0\\
\end{bmatrix}$
\end{tabular}

\noindent
\begin{tabular}{p{.46\textwidth} p{0.46\textwidth}}
$M_{46945}$:
$\begin{bmatrix}
4 & 0 & 1 & 4 & 2\\
1 & 1 & 1 & 1 & 1\\
0 & 2 & 1 & 1 & 0\\
1 & 2 & 1 & 2 & 2\\
\end{bmatrix}$ & %
$M_{47635}$:
$\begin{bmatrix}
3 & 0 & 1 & 0 & 1\\
1 & 1 & 1 & 1 & 1\\
0 & 2 & 1 & 3 & 2\\
1 & 4 & 1 & 3 & 2\\
\end{bmatrix}$
\end{tabular}

\noindent
\begin{tabular}{p{.46\textwidth} p{0.46\textwidth}}
$M_{48819}$:
$\begin{bmatrix}
4 & 0 & 1 & 2 & 2 \\
1 & 4 & 1 & 0 & 1 \\
0 & 1 & 1 & 4 & 2 \\
1 & 1 & 1 & 1 & 1 \\
\end{bmatrix}$ & %
$M_{50041}$:
$\begin{bmatrix}
3 & 0 & 1 & 1 & 4 \\
1 & 1 & 1 & 1 & 1 \\
0 & 3 & 1 & 4 & 3 \\
1 & 4 & 1 & 0 & 1 \\
\end{bmatrix}$
\end{tabular}

%#on 10 elements

\noindent
\begin{tabular}{p{.46\textwidth} p{0.46\textwidth}}
$\Fv_{6}$:
$\begin{bmatrix}
1 & 4 & 0 & 1 & 1 & 2\\
1 & 4 & 0 & 1 & 3 & 1\\
1 & 2 & 4 & 3 & 2 & 4\\
1 & 1 & 1 & 1 & 1 & 1\\
\end{bmatrix}$ & %
$\Fv_{7}$:
$\begin{bmatrix}
1 & 1 & 1 & 1 & 0 & 3\\
2 & 1 & 1 & 0 & 1 & 1\\
4 & 0 & 1 & 0 & 1 & 1\\
0 & 2 & 1 & 1 & 0 & 3\\
\end{bmatrix}$
\end{tabular}

\noindent
\begin{tabular}{p{.46\textwidth} p{0.46\textwidth}}
$\Fv_{8}$:
$\begin{bmatrix}
4 & 1 & 0 & 1 & 2 & 1\\
1 & 1 & 2 & 1 & 2 & 3\\
1 & 1 & 1 & 1 & 1 & 1\\
0 & 0 & 3 & 1 & 1 & 1\\
\end{bmatrix}$ & %
$\Fv_{11}$:
$\begin{bmatrix}
1 & 3 & 1 & 1 & 1 & 3\\
2 & 0 & 1 & 1 & 3 & 1\\
0 & 1 & 1 & 2 & 1 & 2\\
3 & 2 & 1 & 0 & 3 & 2\\
\end{bmatrix}$
\end{tabular}

\noindent
\begin{tabular}{p{.46\textwidth} p{0.46\textwidth}}
$\Fv_{12}$:
$\begin{bmatrix}
1 & 1 & 1 & 3 & 4 & 0\\
1 & 3 & 2 & 2 & 0 & 3\\
1 & 3 & 2 & 0 & 1 & 2\\
1 & 1 & 1 & 1 & 1 & 1\\
\end{bmatrix}$ & %
$\Fv_{16}$:
$\begin{bmatrix}
1 & 0 & 1 & 1 & 1\\
1 & 0 & 3 & 0 & 2\\
1 & 1 & 0 & 3 & 0\\
4 & 1 & 4 & 0 & 1\\
0 & 1 & 0 & 2 & 4\\
\end{bmatrix}$
\end{tabular}

\noindent
\begin{tabular}{p{.46\textwidth} p{0.46\textwidth}}
$\Fv_{18}$:
$\begin{bmatrix}
0 & 3 & 1 & 1 & 1\\
0 & 2 & 3 & 0 & 1\\
2 & 4 & 2 & 3 & 1\\
4 & 0 & 4 & 0 & 1\\
1 & 1 & 1 & 1 & 1\\
\end{bmatrix}$ & %
$\Fv_{19}$:
$\begin{bmatrix}
1 & 1 & 1 & 1 & 1\\
0 & 2 & 1 & 0 & 1\\
3 & 2 & 1 & 2 & 2\\
2 & 0 & 1 & 0 & 4\\
0 & 2 & 1 & 1 & 3\\
\end{bmatrix}$
\end{tabular}

\noindent
\begin{tabular}{p{.46\textwidth} p{0.46\textwidth}}
$\Fv_{21}$:
$\begin{bmatrix}
4 & 0 & 1 & 0 & 1\\
0 & 0 & 1 & 4 & 4\\
1 & 1 & 1 & 1 & 1\\
1 & 4 & 1 & 4 & 4\\
1 & 3 & 1 & 0 & 1\\
\end{bmatrix}$ & %
$\Fv_{24}$:
$\begin{bmatrix}
1 & 1 & 1 & 1 & 1\\
1 & 3 & 2 & 1 & 3\\
1 & 0 & 3 & 3 & 0\\
1 & 1 & 2 & 3 & 0\\
1 & 3 & 2 & 0 & 2\\
\end{bmatrix}$
\end{tabular}

\noindent
\begin{tabular}{p{.46\textwidth} p{0.46\textwidth}}
$\Fv_{25}$:
$\begin{bmatrix}
1 & 1 & 1 & 1 & 1\\
1 & 1 & 2 & 0 & 1\\
2 & 1 & 4 & 1 & 1\\
0 & 2 & 0 & 2 & 1\\
2 & 2 & 1 & 1 & 1\\
\end{bmatrix}$ & %
$\Fv_{27}$:
$\begin{bmatrix}
1 & 1 & 1 & 1 & 1\\
0 & 3 & 1 & 0 & 4\\
0 & 3 & 1 & 3 & 0\\
3 & 0 & 1 & 0 & 1\\
2 & 3 & 1 & 1 & 1\\
\end{bmatrix}$
\end{tabular}

\noindent
\begin{tabular}{p{.46\textwidth} p{0.46\textwidth}}
$\Fv_{37}$:
$\begin{bmatrix}
1 & 1 & 1 & 1 & 1\\
0 & 2 & 1 & 0 & 4\\
3 & 2 & 1 & 2 & 3\\
2 & 0 & 1 & 0 & 2\\
0 & 2 & 1 & 1 & 1\\
\end{bmatrix}$ & %
$\Fv_{41}$:
$\begin{bmatrix}
1 & 1 & 1 & 1 & 1\\
3 & 1 & 4 & 4 & 1\\
3 & 1 & 4 & 2 & 2\\
0 & 1 & 2 & 1 & 1\\
0 & 1 & 0 & 2 & 4\\
\end{bmatrix}$
\end{tabular}

\noindent
\begin{tabular}{p{.46\textwidth} p{0.46\textwidth}}
$\Fv_{42}$:
$\begin{bmatrix}
0 & 0 & 2 & 1 & 1\\
4 & 0 & 0 & 1 & 1\\
1 & 1 & 1 & 1 & 1\\
1 & 1 & 2 & 1 & 3\\
0 & 4 & 1 & 1 & 3\\
\end{bmatrix}$ & %
$\Fv_{47}$:
$\begin{bmatrix}
1 & 1 & 1 & 1 & 1\\
1 & 4 & 1 & 4 & 4\\
1 & 1 & 0 & 3 & 3\\
1 & 0 & 0 & 1 & 3\\
1 & 1 & 1 & 4 & 0\\
\end{bmatrix}$
\end{tabular}

\noindent
\begin{tabular}{p{.46\textwidth} p{0.46\textwidth}}
$\Fv_{51}$:
$\begin{bmatrix}
0 & 1 & 2 & 1 & 1\\
1 & 1 & 1 & 1 & 1\\
4 & 2 & 0 & 1 & 1\\
0 & 4 & 3 & 1 & 2\\
1 & 0 & 1 & 1 & 0\\
\end{bmatrix}$ & %
$\Fv_{54}$:
$\begin{bmatrix}
0 & 2 & 1 & 1 & 0\\
0 & 2 & 1 & 0 & 4\\
1 & 1 & 1 & 3 & 0\\
4 & 0 & 1 & 0 & 4\\
3 & 0 & 1 & 1 & 1\\
\end{bmatrix}$
\end{tabular}

\noindent
\begin{tabular}{p{.46\textwidth} p{0.46\textwidth}}
$\Fv_{56}$:
$\begin{bmatrix}
1 & 1 & 1 & 1 & 1\\
0 & 1 & 1 & 2 & 1\\
1 & 1 & 4 & 1 & 2\\
2 & 1 & 3 & 1 & 3\\
2 & 1 & 3 & 3 & 0\\
\end{bmatrix}$ & %
$\Fv_{60}$:
$\begin{bmatrix}
2 & 2 & 0 & 1 & 1\\
4 & 4 & 3 & 1 & 0\\
3 & 0 & 1 & 1 & 0\\
1 & 1 & 1 & 1 & 1\\
0 & 3 & 0 & 1 & 1\\
\end{bmatrix}$
\end{tabular}

\noindent
\begin{tabular}{p{.46\textwidth} p{0.46\textwidth}}
$\Fv_{61}$:
$\begin{bmatrix}
2 & 2 & 0 & 1 & 0\\
4 & 4 & 3 & 1 & 0\\
3 & 0 & 1 & 1 & 1\\
1 & 1 & 1 & 1 & 1\\
0 & 3 & 0 & 1 & 1\\
\end{bmatrix}$ & %
$\Fv_{66}$:
$\begin{bmatrix}
1 & 1 & 0 & 3 & 3\\
0 & 1 & 1 & 1 & 1\\
1 & 1 & 0 & 1 & 0\\
3 & 1 & 4 & 0 & 0\\
0 & 1 & 1 & 4 & 3\\
\end{bmatrix}$
\end{tabular}

\noindent
\begin{tabular}{p{.46\textwidth} p{0.46\textwidth}}
$\Fv_{68}$:
$\begin{bmatrix}
1 & 0 & 1 & 1 & 1\\
1 & 0 & 4 & 0 & 2\\
1 & 4 & 0 & 4 & 0\\
2 & 1 & 2 & 0 & 1\\
0 & 1 & 0 & 2 & 1\\
\end{bmatrix}$ & %
$\Fv_{71}$:
$\begin{bmatrix}
0 & 2 & 1 & 2 & 2\\
0 & 2 & 1 & 0 & 4\\
1 & 1 & 1 & 1 & 1\\
3 & 0 & 1 & 0 & 4\\
4 & 0 & 1 & 2 & 3\\
\end{bmatrix}$
\end{tabular}

\noindent
\begin{tabular}{p{.46\textwidth} p{0.46\textwidth}}
$\Fv_{75}$:
$\begin{bmatrix}
1 & 1 & 0 & 3 & 0\\
0 & 1 & 1 & 1 & 1\\
1 & 1 & 0 & 1 & 1\\
3 & 1 & 4 & 0 & 0\\
0 & 1 & 1 & 4 & 2\\
\end{bmatrix}$ & %
$\Fv_{77}$:
$\begin{bmatrix}
1 & 1 & 1 & 1 & 0\\
0 & 2 & 1 & 0 & 2\\
3 & 2 & 1 & 2 & 0\\
2 & 0 & 1 & 0 & 2\\
0 & 2 & 1 & 1 & 1\\
\end{bmatrix}$
\end{tabular}

\noindent
\begin{tabular}{p{.46\textwidth} p{0.46\textwidth}}
$\Fv_{85}$:
$\begin{bmatrix}
4 & 3 & 1 & 3 & 2\\
0 & 1 & 1 & 0 & 2\\
1 & 1 & 1 & 1 & 1\\
1 & 0 & 1 & 0 & 4\\
1 & 1 & 1 & 3 & 2\\
\end{bmatrix}$ & %
$\Fv_{87}$:
$\begin{bmatrix}
1 & 3 & 0 & 1 & 1\\
1 & 1 & 1 & 1 & 1\\
1 & 1 & 3 & 2 & 4\\
1 & 1 & 2 & 2 & 2\\
1 & 4 & 0 & 2 & 3\\
\end{bmatrix}$
\end{tabular}

\noindent
\begin{tabular}{p{.46\textwidth} p{0.46\textwidth}}
$\Fv_{90}$:
$\begin{bmatrix}
0 & 2 & 1 & 2 & 0\\
0 & 1 & 1 & 0 & 4\\
1 & 1 & 1 & 1 & 1\\
1 & 0 & 1 & 0 & 2\\
1 & 1 & 1 & 2 & 3\\
\end{bmatrix}$ & %
$\Fv_{91}$:
$\begin{bmatrix}
0 & 4 & 1 & 1 & 3\\
2 & 0 & 1 & 2 & 0\\
1 & 1 & 1 & 1 & 1\\
1 & 1 & 1 & 3 & 2\\
4 & 2 & 1 & 2 & 0\\
\end{bmatrix}$
\end{tabular}

\noindent
\begin{tabular}{p{.46\textwidth} p{0.46\textwidth}}
$\Fv_{103}$:
$\begin{bmatrix}
1 & 1 & 1 & 1 & 1\\
1 & 0 & 4 & 1 & 1\\
1 & 0 & 4 & 3 & 3\\
1 & 2 & 1 & 3 & 1\\
1 & 1 & 1 & 0 & 4\\
\end{bmatrix}$ &
$\GP_{10}$:
$\begin{bmatrix}
4 & 3 & 0 & 1 & 1\\
3 & 3 & 0 & 0 & 1\\
0 & 0 & 1 & 2 & 1\\
1 & 0 & 2 & 0 & 1\\
1 & 1 & 1 & 1 & 0\\
\end{bmatrix}$
\end{tabular}

\noindent
\begin{tabular}{p{.46\textwidth} p{0.46\textwidth}}
$\TQ_{10}$:
$\begin{bmatrix}
3 & 0 & 0 & 1 & 4\\
0 & 0 & 1 & 2 & 2\\
0 & 1 & 1 & 1 & 1\\
1 & 1 & 1 & 0 & 4\\
4 & 1 & 2 & 4 & 4\\
\end{bmatrix}$ & %
$\FF_{10}$:
$\begin{bmatrix}
1 & 0 & 0 & 1 & 1\\
1 & 1 & 0 & 0 & 1\\
1 & 3 & 1 & 0 & 0\\
1 & 1 & 1 & 2 & 0\\
3 & 3 & 3 & 2 & 2\\
\end{bmatrix}$
\end{tabular}

\noindent
\begin{tabular}{p{.46\textwidth} p{0.46\textwidth}}
$\Fv_{14}$:
$\begin{bmatrix}
1 & 1 & 1 & 1 & 1\\
2 & 1 & 1 & 0 & 1\\
2 & 0 & 0 & 1 & 1\\
0 & 0 & 1 & 1 & 1\\
0 & 4 & 2 & 1 & 0\\
\end{bmatrix}$ & %
$\Fv_{22}$:
$\begin{bmatrix}
4 & 0 & 3 & 0 & 1\\
1 & 1 & 1 & 1 & 1\\
4 & 0 & 0 & 2 & 1\\
0 & 1 & 2 & 0 & 1\\
2 & 4 & 0 & 2 & 1\\
\end{bmatrix}$
\end{tabular}

\noindent
\begin{tabular}{p{.46\textwidth} p{0.46\textwidth}}
$\Fv_{26}$:
$\begin{bmatrix}
0 & 1 & 0 & 2 & 1\\
0 & 0 & 2 & 3 & 1\\
2 & 0 & 0 & 2 & 1\\
1 & 1 & 1 & 1 & 1\\
4 & 1 & 1 & 0 & 1\\
\end{bmatrix}$ & %
$\Fv_{28}$:
$\begin{bmatrix}
0 & 0 & 1 & 2 & 3\\
1 & 1 & 1 & 1 & 1\\
1 & 3 & 1 & 2 & 1\\
1 & 1 & 1 & 3 & 4\\
3 & 1 & 1 & 2 & 1\\
\end{bmatrix}$
\end{tabular}

\noindent
\begin{tabular}{p{.46\textwidth} p{0.46\textwidth}}
$\Fv_{30}$:
$\begin{bmatrix}
1 & 1 & 1 & 1 & 1\\
0 & 2 & 0 & 1 & 1\\
1 & 1 & 3 & 1 & 2\\
2 & 1 & 0 & 1 & 3\\
0 & 0 & 1 & 1 & 1\\
\end{bmatrix}$ & %
$\Fv_{32}$:
$\begin{bmatrix}
2 & 0 & 0 & 1 & 2\\
1 & 1 & 1 & 1 & 1\\
3 & 3 & 1 & 1 & 1\\
0 & 2 & 1 & 1 & 0\\
4 & 4 & 0 & 1 & 1\\
\end{bmatrix}$
\end{tabular}

\noindent
\begin{tabular}{p{.46\textwidth} p{0.46\textwidth}}
$\Fv_{35}$:
$\begin{bmatrix}
1 & 2 & 2 & 1 & 2\\
1 & 1 & 1 & 1 & 1\\
1 & 0 & 2 & 3 & 0\\
1 & 1 & 3 & 1 & 2\\
1 & 4 & 1 & 0 & 0\\
\end{bmatrix}$ & %
$\Fv_{36}$:
$\begin{bmatrix}
4 & 4 & 1 & 0 & 1\\
1 & 1 & 1 & 1 & 0\\
3 & 3 & 0 & 1 & 1\\
0 & 4 & 2 & 0 & 1\\
1 & 0 & 0 & 1 & 1\\
\end{bmatrix}$
\end{tabular}

\noindent
\begin{tabular}{p{.46\textwidth} p{0.46\textwidth}}
$\Fv_{39}$:
$\begin{bmatrix}
0 & 3 & 1 & 3 & 2\\
0 & 2 & 1 & 0 & 3\\
1 & 1 & 1 & 1 & 1\\
1 & 0 & 1 & 0 & 4\\
1 & 4 & 1 & 3 & 2\\
\end{bmatrix}$ & %
$\Fv_{40}$:
$\begin{bmatrix}
0 & 0 & 1 & 2 & 1\\
4 & 0 & 0 & 2 & 1\\
1 & 1 & 1 & 1 & 1\\
2 & 2 & 1 & 2 & 1\\
0 & 2 & 4 & 4 & 1\\
\end{bmatrix}$
\end{tabular}

\noindent
\begin{tabular}{p{.46\textwidth} p{0.46\textwidth}}
$\Fv_{43}$:
$\begin{bmatrix}
0 & 0 & 1 & 2 & 1\\
2 & 1 & 3 & 3 & 1\\
1 & 1 & 1 & 1 & 1\\
3 & 4 & 2 & 1 & 1\\
2 & 3 & 1 & 1 & 1\\
\end{bmatrix}$ & %
$\Fv_{44}$:
$\begin{bmatrix}
0 & 3 & 1 & 3 & 2\\
0 & 2 & 1 & 0 & 3\\
1 & 1 & 1 & 1 & 1\\
1 & 0 & 1 & 0 & 2\\
1 & 4 & 1 & 3 & 2\\
\end{bmatrix}$
\end{tabular}

\noindent
\begin{tabular}{p{.46\textwidth} p{0.46\textwidth}}
$\Fv_{46}$:
$\begin{bmatrix}
1 & 1 & 1 & 1 & 1\\
1 & 4 & 1 & 3 & 4\\
4 & 4 & 1 & 4 & 0\\
0 & 3 & 1 & 3 & 1\\
2 & 4 & 1 & 1 & 1\\
\end{bmatrix}$ & %
$\Fv_{48}$:
$\begin{bmatrix}
3 & 2 & 1 & 2 & 0\\
0 & 1 & 1 & 0 & 1\\
1 & 1 & 1 & 1 & 1\\
2 & 0 & 1 & 0 & 2\\
0 & 1 & 1 & 2 & 0\\
\end{bmatrix}$
\end{tabular}

\noindent
\begin{tabular}{p{.46\textwidth} p{0.46\textwidth}}
$\Fv_{49}$:
$\begin{bmatrix}
3 & 3 & 4 & 3 & 1\\
0 & 2 & 0 & 3 & 1\\
1 & 1 & 1 & 1 & 1\\
3 & 2 & 0 & 2 & 1\\
0 & 0 & 4 & 3 & 1\\
\end{bmatrix}$ & %
$\Fv_{50}$:
$\begin{bmatrix}
1 & 4 & 3 & 3 & 1\\
0 & 0 & 1 & 4 & 1\\
1 & 1 & 1 & 1 & 1\\
3 & 2 & 4 & 1 & 1\\
4 & 4 & 1 & 1 & 1\\
\end{bmatrix}$
\end{tabular}

\noindent
\begin{tabular}{p{.46\textwidth} p{0.46\textwidth}}
$\Fv_{53}$:
$\begin{bmatrix}
1 & 1 & 3 & 1 & 2\\
3 & 2 & 0 & 1 & 1\\
1 & 1 & 1 & 1 & 1\\
0 & 1 & 0 & 1 & 2\\
3 & 4 & 3 & 1 & 1\\
\end{bmatrix}$ & %
$\Fv_{55}$:
$\begin{bmatrix}
2 & 0 & 3 & 1 & 0\\
1 & 1 & 1 & 1 & 1\\
3 & 3 & 3 & 1 & 0\\
0 & 2 & 0 & 1 & 1\\
4 & 4 & 3 & 1 & 1\\
\end{bmatrix}$
\end{tabular}

\noindent
\begin{tabular}{p{.46\textwidth} p{0.46\textwidth}}
$\Fv_{57}$:
$\begin{bmatrix}
0 & 1 & 1 & 4 & 3\\
2 & 2 & 1 & 0 & 0\\
1 & 1 & 1 & 1 & 1\\
2 & 2 & 1 & 4 & 3\\
0 & 4 & 1 & 1 & 1\\
\end{bmatrix}$ & %
$\Fv_{58}$:
$\begin{bmatrix}
1 & 1 & 1 & 1 & 1\\
4 & 0 & 1 & 4 & 4\\
0 & 2 & 1 & 2 & 0\\
3 & 3 & 1 & 1 & 1\\
0 & 0 & 1 & 4 & 3\\
\end{bmatrix}$
\end{tabular}

\noindent
\begin{tabular}{p{.46\textwidth} p{0.46\textwidth}}
$\Fv_{59}$:
$\begin{bmatrix}
3 & 0 & 1 & 0 & 1\\
0 & 0 & 1 & 1 & 1\\
1 & 2 & 1 & 2 & 0\\
1 & 1 & 1 & 1 & 1\\
1 & 3 & 1 & 0 & 2\\
\end{bmatrix}$ & %
$\Fv_{62}$:
$\begin{bmatrix}
1 & 1 & 1 & 1 & 1\\
0 & 2 & 0 & 4 & 1\\
4 & 4 & 1 & 4 & 1\\
2 & 4 & 0 & 4 & 1\\
0 & 0 & 1 & 1 & 1\\
\end{bmatrix}$
\end{tabular}

\noindent
\begin{tabular}{p{.46\textwidth} p{0.46\textwidth}}
$\Fv_{63}$:
$\begin{bmatrix}
0 & 4 & 1 & 2 & 2\\
2 & 0 & 1 & 4 & 0\\
1 & 1 & 1 & 2 & 0\\
1 & 1 & 1 & 1 & 1\\
4 & 2 & 1 & 4 & 1\\
\end{bmatrix}$ & %
$\Fv_{65}$:
$\begin{bmatrix}
0 & 0 & 1 & 3 & 1\\
1 & 1 & 1 & 2 & 3\\
1 & 2 & 1 & 3 & 3\\
1 & 1 & 1 & 1 & 1\\
2 & 1 & 1 & 3 & 3\\
\end{bmatrix}$
\end{tabular}

\noindent
\begin{tabular}{p{.46\textwidth} p{0.46\textwidth}}
$\Fv_{70}$:
$\begin{bmatrix}
1 & 3 & 1 & 3 & 1\\
0 & 1 & 1 & 0 & 4\\
1 & 1 & 1 & 1 & 1\\
4 & 0 & 1 & 0 & 1\\
0 & 1 & 1 & 3 & 0\\
\end{bmatrix}$ & %
$\Fv_{72}$:
$\begin{bmatrix}
2 & 3 & 1 & 3 & 1\\
0 & 1 & 1 & 0 & 3\\
1 & 1 & 1 & 1 & 1\\
4 & 0 & 1 & 0 & 3\\
0 & 1 & 1 & 3 & 1\\
\end{bmatrix}$
\end{tabular}

\noindent
\begin{tabular}{p{.46\textwidth} p{0.46\textwidth}}
$\Fv_{73}$:
$\begin{bmatrix}
4 & 0 & 1 & 0 & 2\\
1 & 1 & 1 & 1 & 1\\
4 & 0 & 1 & 3 & 0\\
0 & 2 & 1 & 0 & 1\\
3 & 1 & 1 & 3 & 1\\
\end{bmatrix}$ & %
$\Fv_{78}$:
$\begin{bmatrix}
4 & 1 & 0 & 1 & 1\\
0 & 0 & 1 & 3 & 1\\
1 & 1 & 1 & 1 & 1\\
3 & 3 & 1 & 3 & 1\\
1 & 2 & 0 & 1 & 1\\
\end{bmatrix}$
\end{tabular}

\noindent
\begin{tabular}{p{.46\textwidth} p{0.46\textwidth}}
$\Fv_{79}$:
$\begin{bmatrix}
4 & 0 & 1 & 0 & 3\\
0 & 0 & 1 & 4 & 3\\
1 & 1 & 1 & 1 & 1\\
1 & 4 & 1 & 4 & 3\\
1 & 3 & 1 & 0 & 2\\
\end{bmatrix}$ & %
$\Fv_{81}$:
$\begin{bmatrix}
0 & 1 & 0 & 1 & 1\\
0 & 1 & 1 & 0 & 1\\
1 & 1 & 1 & 1 & 0\\
3 & 0 & 4 & 0 & 1\\
4 & 0 & 0 & 1 & 1\\
\end{bmatrix}$
\end{tabular}

\noindent
\begin{tabular}{p{.46\textwidth} p{0.46\textwidth}}
$\Fv_{82}$:
$\begin{bmatrix}
4 & 0 & 1 & 2 & 3\\
2 & 1 & 1 & 2 & 0\\
1 & 1 & 1 & 1 & 0\\
0 & 4 & 1 & 1 & 1\\
0 & 4 & 1 & 4 & 3\\
\end{bmatrix}$ & %
$\Fv_{83}$:
$\begin{bmatrix}
4 & 0 & 1 & 0 & 4\\
1 & 1 & 1 & 1 & 1\\
4 & 0 & 1 & 2 & 0\\
0 & 4 & 1 & 0 & 1\\
2 & 1 & 1 & 2 & 1\\
\end{bmatrix}$
\end{tabular}

\noindent
\begin{tabular}{p{.46\textwidth} p{0.46\textwidth}}
$\Fv_{84}$:
$\begin{bmatrix}
1 & 1 & 1 & 3 & 3\\
0 & 4 & 1 & 0 & 3\\
1 & 1 & 1 & 1 & 1\\
2 & 0 & 1 & 0 & 3\\
0 & 0 & 1 & 1 & 1\\
\end{bmatrix}$ & %
$\Fv_{89}$:
$\begin{bmatrix}
1 & 1 & 1 & 1 & 1\\
0 & 4 & 1 & 0 & 1\\
2 & 4 & 1 & 4 & 0\\
2 & 0 & 1 & 0 & 4\\
2 & 4 & 1 & 1 & 1\\
\end{bmatrix}$
\end{tabular}

\noindent
\begin{tabular}{p{.46\textwidth} p{0.46\textwidth}}
$\Fv_{92}$:
$\begin{bmatrix}
1 & 2 & 1 & 2 & 1\\
0 & 1 & 1 & 0 & 2\\
1 & 1 & 1 & 1 & 1\\
2 & 0 & 1 & 0 & 1\\
0 & 1 & 1 & 2 & 0\\
\end{bmatrix}$ & %
$\Fv_{93}$:
$\begin{bmatrix}
1 & 1 & 1 & 2 & 2\\
0 & 2 & 1 & 0 & 2\\
1 & 1 & 1 & 1 & 1\\
4 & 0 & 1 & 0 & 2\\
0 & 0 & 1 & 1 & 1\\
\end{bmatrix}$
\end{tabular}

\noindent
\begin{tabular}{p{.46\textwidth} p{0.46\textwidth}}
$\Fv_{94}$:
$\begin{bmatrix}
2 & 3 & 0 & 1 & 1\\
0 & 0 & 3 & 1 & 1\\
0 & 2 & 1 & 0 & 1\\
4 & 1 & 0 & 1 & 1\\
1 & 0 & 2 & 0 & 1\\
\end{bmatrix}$ & %
$\Fv_{95}$:
$\begin{bmatrix}
1 & 1 & 1 & 1 & 1\\
1 & 1 & 1 & 0 & 0\\
4 & 2 & 1 & 2 & 0\\
4 & 0 & 1 & 0 & 4\\
4 & 4 & 1 & 1 & 1\\
\end{bmatrix}$
\end{tabular}

\noindent
\begin{tabular}{p{.46\textwidth} p{0.46\textwidth}}
$\Fv_{96}$:
$\begin{bmatrix}
1 & 1 & 1 & 1 & 1\\
1 & 1 & 1 & 0 & 0\\
1 & 3 & 2 & 3 & 0\\
1 & 0 & 2 & 0 & 3\\
1 & 1 & 2 & 2 & 2\\
\end{bmatrix}$ & %
$\Fv_{101}$:
$\begin{bmatrix}
2 & 3 & 1 & 1 & 1\\
1 & 0 & 2 & 1 & 2\\
0 & 4 & 0 & 1 & 1\\
1 & 1 & 1 & 1 & 1\\
1 & 0 & 0 & 1 & 3\\
\end{bmatrix}$
\end{tabular}

\noindent
\begin{tabular}{p{.46\textwidth} p{0.46\textwidth}}
$\Fv_{102}$:
$\begin{bmatrix}
3 & 1 & 0 & 1 & 3\\
1 & 1 & 1 & 1 & 1\\
0 & 0 & 1 & 1 & 4\\
3 & 3 & 2 & 1 & 3\\
1 & 2 & 1 & 1 & 3\\
\end{bmatrix}$ & %
$\Fv_{104}$:
$\begin{bmatrix}
1 & 1 & 4 & 0 & 1\\
0 & 0 & 1 & 1 & 1\\
3 & 2 & 0 & 1 & 0\\
1 & 1 & 1 & 0 & 4\\
4 & 1 & 0 & 2 & 4\\
\end{bmatrix}$
\end{tabular}

\noindent
\begin{tabular}{p{.46\textwidth} p{0.46\textwidth}}
$\Fv_{106}$:
$\begin{bmatrix}
2 & 3 & 1 & 3 & 0\\
0 & 1 & 1 & 0 & 1\\
1 & 1 & 1 & 1 & 1\\
4 & 0 & 1 & 0 & 4\\
0 & 1 & 1 & 3 & 0\\
\end{bmatrix}$ & %
$\Fv_{111}$:
$\begin{bmatrix}
1 & 1 & 1 & 1 & 1\\
0 & 2 & 1 & 0 & 4\\
0 & 2 & 1 & 2 & 3\\
3 & 0 & 1 & 0 & 1\\
4 & 2 & 1 & 1 & 1\\
\end{bmatrix}$
\end{tabular}

\noindent
\begin{tabular}{p{.46\textwidth} p{0.46\textwidth}}
$\Fv_{112}$:
$\begin{bmatrix}
1 & 1 & 1 & 1 & 1\\
2 & 1 & 1 & 2 & 2\\
1 & 0 & 1 & 3 & 1\\
4 & 2 & 1 & 4 & 0\\
2 & 4 & 1 & 0 & 1\\
\end{bmatrix}$ & %
\end{tabular}
\end{itemize}

%\subsection*{\texorpdfstring{$10$}{10}-element excluded minors for \texorpdfstring{$\GF(5)$}{GF(5)}-representable matroids}

%\ldots

\begin{comment}
\begin{table}[hb]
  \begin{tabular}{c c}
    \hline
    $F_7^=$ & 304 \\
    $\Lambda_3$ & 315 \\
    $(F_7^=)^*$ & 416 \\
    $\AG(2,3) \ba e$ & 801 \\
    $(\AG(2,3) \ba e)^{\Delta Y}$ & 1349 \\
    $\Lambda_3^*$ & ? \\
    $T_8$ & 1489 \\
    %$Q_8$ & 1493 \\
    $\SP_8$ & 1496 \\
    $P_8$ & 1497 \\
    $\KP_8$ & 1511 \\
    $\LP_8$ & 1513 \\
    $P_8^-$ & 1515 \\
    $\SP_8^=$ & 1547 \\
    $\WQ_8$ & 1558 \\
    $\TQ_8$ & 1562 \\
    $P_8^=$ & 1564 \\
    $\Phi_4$, $V_8^+$ & 1732 \\
    $(\AG(2,3) \ba e)^*$ & 2093 \\
    $\BBR_9$ & 3197 \\
    $\FN_9$ & 3209 \\
    $A_9$, $M_{8591}^{Y\Delta}$ & 3211 \\
    $\pappus$ & 3247 \\
    $\pappus^*$ & ? \\
    $\pappus^{\Delta Y}$ & ? \\
    %$(\pappus^{\Delta Y})^*$ & ? \\
    $(\BBR_9^{\Delta\nabla})^*$ & 9573 \\
    $\KR_9$ & 9875 \\
    $\PP_9$ & 9975 \\
    $\TQ_9'$ & 10063 \\
    $(\TQ_9')^*$ & ? \\
    $(\TQ_9^*)^{Y\Delta}$ & 10291 \\
    $(\KR_9^*)^{Y\Delta}$ & 15656 \\
    $\TQ_9$ & 16006 \\
    $\FX_9$ & 48806 \\
    $\FX_9^*$ & ? \\
    $\BBR_9^{\Delta\nabla}$ & 200572 \\
    $(\BBR_9^{\Delta\nabla})^{Y\Delta}$ & ? \\
    %$((\BBR_9^{\Delta\nabla})^{Y\Delta})^*$ & ? \\
    $\TQ_9^*$ & 200649 \\
    $\PP_9^*$ & 200745 \\
    $\KR_9^{\Delta Y}$ & 203906 \\
    $\KR_9^*$ & 203907 \\
    $\TQ_9^{\Delta Y}$ & 204390 \\
    \hline
  \end{tabular}
  \caption{Some named matroids on at most $9$ elements, and their number in Mayhew and Royle's catalog.}
\end{table}

\end{comment}

\end{document}